\documentclass[10pt, leqno]{article}
\usepackage{amsfonts,cite}
\usepackage{amsmath, amssymb}
\usepackage{euler}
\usepackage{ marvosym }
\usepackage{enumitem}
\usepackage[amsmath,thmmarks]{ntheorem}
\usepackage{chngcntr,color}

\counterwithin*{equation}{section}
\counterwithin*{equation}{subsection}

\newtheorem{theorem}{Theorem}

\numberwithin{section}{part}
\newtheorem{corollary}{Corollary}
\newtheorem{lemma}{Lemma}
\newtheorem*{remark}{Remark}

\newtheorem*{remarks}{Remarks}

\newenvironment{proof}[1][Proof]{\noindent\textbf{#1.} }{\ \rule{0.5em}{0.5em}}

\newcommand{\A}{\mathcal{A}}

\input{tcilatex}

\begin{document}

\title{Finiteness Principles for Smooth Selection}
\date{\today }
\author{Charles Fefferman, Arie Israel, Garving K. Luli \thanks{%
The first author is supported in part by NSF grant DMS-1265524, AFOSR
grant FA9550-12-1-0425, and Grant No 2014055 from the United States-Israel Binational Science Foundation (BSF). The third author is supported in part by NSF grant
DMS-1355968 and by a start-up fund from UC Davis.}}
\maketitle

\section*{Introduction}

In this paper and \cite{fil-2016}, we extend a basic finiteness principle \cite{bs-1994,f-2005},
used in \cite{fb1,fb2} to fit smooth functions $F$ to data. Our results
raise the hope that one can start to understand constrained interpolation
problems in which e.g. the interpolating function $F$ is required to be
nonnegative.

Let us set up notation. We fix positive integers $m$, $n$, $D$. We will work
with the function spaces $C^{m}(\mathbb{R}^{n},\mathbb{R}^{D})$ and $%
C^{m-1,1}(\mathbb{R}^{n},\mathbb{R}^{D})$ and their norms $\Vert F\Vert
_{C^{m}\left( \mathbb{R}^{n},\mathbb{R}^{D}\right) }$ and $\left\Vert
F\right\Vert _{C^{m-1,1}\left( \mathbb{R}^{n},\mathbb{R}^{D}\right) }$.
Here, $C^{m}\left( \mathbb{R}^{n},\mathbb{R}^{D}\right) $ denotes the space
of all functions $F:\mathbb{R}^{n}\rightarrow \mathbb{R}^{D}$ whose
derivatives $\partial ^{\beta }F$ (for all $\left\vert \beta \right\vert
\leq m$) are continuous and bounded on $\mathbb{R}^{n}$, and $%
C^{m-1,1}\left( \mathbb{R}^{n},\mathbb{R}^{D}\right) $ denotes the space of
all $F:\mathbb{R}^{n}\rightarrow \mathbb{R}^{D}$ whose derivatives $\partial
^{\beta }F$ (for all $\left\vert \beta \right\vert \leq m-1$) are bounded
and Lipschitz on $\mathbb{R}^{n}$. When $D=1$, we write $C^{m}\left( \mathbb{%
R}^{n}\right) $ and $C^{m-1,1}\left( \mathbb{R}^{n}\right) $ in place of $%
C^{m}\left( \mathbb{R}^{n},\mathbb{R}^{D}\right) $ and $C^{m-1,1}\left( 
\mathbb{R}^{n},\mathbb{R}^{D}\right) $.

Expressions $c\left( m,n\right) $, $C\left( m,n\right) $, $k\left(
m,n\right) $, etc. denote constants depending only on $m$, $n$; these
expressions may denote different constants in different occurrences. Similar
conventions apply to constants denoted by $C\left( m,n,D\right)$, $k\left(
D\right)$, etc.

If $X$ is any finite set, then $\#\left( X\right) $ denotes the number of
elements in $X$.

We recall the basic finiteness principle of \cite{f-2005}.

\begin{theorem}
\label{Th1}For large enough $k^{\#}=k\left( m,n\right) $ and $C^{\#}=C\left(
m,n\right) $ the following hold:

\begin{description}
\item[(A) $C^{m}\text{ FLAVOR}$] Let $f:E\rightarrow \mathbb{R}$ with $%
E\subset \mathbb{R}^{n}$ finite. Suppose that for each $S\subset E$ with $%
\#\left( S\right) \leq k^{\#}$ there exists $F^{S}\in C^{m}\left( \mathbb{R}%
^{n}\right) $ with norm $\left\Vert F^{S}\right\Vert _{C^{m}\left( \mathbb{R}%
^{n}\right) }\leq 1$, such that $F^{S}=f$ on $S$. Then there exists $F\in
C^{m}\left( \mathbb{R}^{n}\right) $ with norm $\left\Vert F\right\Vert
_{C^{m}\left( \mathbb{R}^{n}\right) }\leq C^{\#}$, such that $F=f$ on $E$.

\item[(B) $C^{m-1,1}\text{ FLAVOR} $] Let $f:E\rightarrow \mathbb{R}$ with $%
E\subset \mathbb{R}^{n}$ arbitrary. Suppose that for each $S\subset E$ with $%
\#\left( S\right) \leq k^{\#}$, there exists $F^{S}\in C^{m-1,1}\left( 
\mathbb{R}^{n}\right) $ with norm $\left\Vert F^{S}\right\Vert
_{C^{m-1,1}\left( \mathbb{R}^{n}\right) }\leq 1$, such that $F^{S}=f$ on $S$%
. Then there exists $F\in C^{m-1,1}\left( \mathbb{R}^{n}\right) $ with norm $%
\left\Vert F\right\Vert _{C^{m-1,1}\left( \mathbb{R}^{n}\right) }\leq C^{\#}$%
, such that $F=f$ on $E. $
\end{description}
\end{theorem}

Theorem \ref{Th1} and several related results were conjectured by Y. Brudnyi
and P. Shvartsman in \cite{bs-1994-a} and \cite{bs-1994} (see also \cite{pavel-1984,pavel-1986,pavel-1987}). The first nontrivial case $m=2$
with the sharp \textquotedblleft finiteness constant\textquotedblright\ $%
k^{\#}=3\cdot 2^{n-1}$ was proven by P. Shvartsman \cite{pavel-1982,pavel-1987}; see also \cite{pavel-1984, bs-2001}. The proof of Theorem \ref{Th1}
for general $m$, $n$ appears in \cite{f-2005}. For general $m$, $n$, the optimal $%
k^{\#}$ is unknown, but see \cite{bm-2007, s-2008}.

The proof \cite{pavel-1984,pavel-1987} of Theorem \ref{Th1} for $m=2$ was based on a generalization of the
following \textquotedblleft finiteness principle for Lipschitz selection" \cite{pavel-1986} for maps of metric spaces.

\begin{theorem}
\label{Th2}For large enough $k^{\#}=k^{\#}\left( D\right) $ and $%
C^{\#}=C^{\#}\left( D\right) $, the following holds. \newline
Let $X$ be a metric space. For each $x\in X$, let $K\left( x\right) \subset 
\mathbb{R}^{D}$ be an affine subspace in $\mathbb{R}^D$ of dimension $\leq d$. Suppose that for each $S\subset X$
with $\#\left( S\right) \leq k^{\#}$ there exists a map $F^{S}:S\rightarrow 
\mathbb{R}^{D}$ with Lipschitz constant $\leq 1$, such that $F^{S}\left(
x\right) \in K\left( x\right) $ for all $x\in S$.\newline
Then there exists a map $F:X\rightarrow \mathbb{R}^{D}$ with Lipschitz
constant $\leq C^{\#}$, such that $F\left( x\right) \in K\left( x\right) $
for all $x\in X$.
\end{theorem}

In fact, P. Shvartsman in \cite{pavel-1986} showed that one can take $k^\# = 2^{d+1}$ in Theorem \ref{Th2} and that the constant $k^\# = 2^{d+1}$ is sharp, see \cite{pavel-1992}.

P. Shvartsman also showed that Theorem \ref{Th2} remains valid when $\mathbb{R}^D$ is replaced by a Hilbert space (see \cite{s-2001}) or a Banach space (see \cite{pavel-2004}). 

It is conjectured in \cite{bs-1994} that Theorem \ref{Th2} should hold for any compact convex subsets $K(x) \subset \mathbb{R}^D$. In \cite{s-2002}, P. Shvartsman provided evidence for this conjecture: He showed that the conjecture holds in the case when $D = 2$ and in the case when $X$ is a finite metric space and the constant $C^\#$ is allowed to depend on the cardinality of $X$.

In this paper we prove finiteness principles for $C^{m}\left( \mathbb{R}^{n},%
\mathbb{R}^{D}\right) $-selection, and for $C^{m-1,1}\left( \mathbb{R}^{n},%
\mathbb{R}^{D}\right) $-selection, in particular providing a proof for the conjecture in \cite{bs-1994} for the case $X = \mathbb{R}^n$. 

\begin{theorem}
\label{Th3}For large enough $k^{\#}=k\left( m,n,D\right) $ and $%
C^{\#}=C\left( m,n,D\right) $, the following hold.

\begin{description}
\item[(A) $C^{m}$ FLAVOR] Let $E\subset \mathbb{R}^{n}$ be finite. For each $%
x\in E$, let $K\left( x\right) \subset \mathbb{R}^{D}$ be convex. Suppose
that for each $S\subset E$ with $\#\left( S\right) \leq k^{\#}$, there
exists $F^{S}\in C^{m}\left( \mathbb{R}^{n},\mathbb{R}^{D}\right) $ with
norm $\left\Vert F^{S}\right\Vert _{C^{m}\left( \mathbb{R}^{n},\mathbb{R}%
^{D}\right) }\leq 1$, such that $F^{S}\left( x\right) \in K\left( x\right) $
for all $x\in S$. \newline
Then there exists $F\in C^{m}\left( \mathbb{R}^{n},\mathbb{R}^{D}\right) $
with norm $\left\Vert F\right\Vert _{C^{m}\left( \mathbb{R}^{n},\mathbb{R}%
^{D}\right) }\leq C^{\#}$, such that $F\left( x\right) \in K\left( x\right) $
for all $x\in E$.

\item[(B) $C^{m-1,1}$ FLAVOR] Let $E \subset \mathbb{R}^n$ be arbitrary. For each $x \in E$, let $K(x)
\subset \mathbb{R}^n$ be closed and convex. Suppose that for
each $S\subset E$ with $\#\left( S\right) \leq k^{\#}$, there
exists $F^{S}\in C^{m-1,1}\left( \mathbb{R}^{n},\mathbb{R}^{D}\right) $ with
norm $\left\Vert F^{S}\right\Vert _{C^{m-1,1}\left( \mathbb{R}^{n},\mathbb{R}%
^{D}\right) }\leq 1$, such that $F^{S}\left( x\right) \in K\left( x\right) $
for all $x\in S$.\newline
Then there exists $F\in C^{m-1,1}\left( \mathbb{R}^{n},\mathbb{R}^{D}\right) 
$ with norm $\left\Vert F\right\Vert _{C^{m-1,1}\left( \mathbb{R}^{n},%
\mathbb{R}^{D}\right) }\leq C^{\#}$, such that $F\left( x\right) \in K\left(
x\right) $ for all $x\in E$.
\end{description}
\end{theorem}

In a forthcoming paper \cite{fil-2016}, we will prove the following closely related result on interpolation by nonnegative functions.

\begin{theorem}
\label{Th4}For large enough $k^{\#}=k\left( m,n\right) $ and $C^{\#}=C\left(
m,n\right) $ the following hold.

\begin{description}
\item[(A) $C^{m}$ FLAVOR] Let $f:E\rightarrow [0,\infty)$ with $%
E\subset \mathbb{R}^{n}$ finite. Suppose that for each $S\subset E$ with $%
\#\left( S\right) \leq k^{\#}$, there exists $F^{S}\in C^{m}\left( \mathbb{R}%
^{n}\right) $ with norm $\left\Vert F^{S}\right\Vert _{C^{m}\left( \mathbb{R}%
^{n}\right) }\leq 1$, such that $F^{S}=f\ $on $S$ and $F^{S}\geq 0$ on $%
\mathbb{R}^{n}$. \newline
Then there exists $F\in C^{m}\left( \mathbb{R}^{n}\right) $ with norm $%
\left\Vert F\right\Vert _{C^{m}\left( \mathbb{R}^{n}\right) }\leq C^{\#}$,
such that $F=f$ on $E$ and $F\geq 0$ on $\mathbb{R}^{n}$.

\item[(B) $C^{m-1,1}$ FLAVOR] Let $f:E\rightarrow [0,\infty)$ with $%
E\subset \mathbb{R}^{n}$ arbitrary. Suppose that for each $S\subset E$ with $%
\#\left( S\right) \leq k^{\#}$, there exists $F^{S}\in C^{m-1,1}\left( 
\mathbb{R}^{n}\right) $ with norm $\left\Vert F^{S}\right\Vert
_{C^{m-1,1}\left( \mathbb{R}^{n}\right) }\leq 1$, such that $F^{S}=f\ $on $S$
and $F^{S}\geq 0$ on $\mathbb{R}^{n}$. \newline
Then there exists $F\in C^{m-1,1}\left( \mathbb{R}^{n}\right) $ with norm $%
\left\Vert F\right\Vert _{C^{m-1,1}\left( \mathbb{R}^{n}\right) }\leq C^{\#}$%
, such that $F=f$ on $E$ and $F\geq 0$ on $\mathbb{R}^{n}$.
\end{description}
\end{theorem}

One might ask how to decide whether there exist $F^S$ as in the above results. For Theorem \ref{Th1} this issue is addressed in \cite{f-2005}; for Theorems \ref{Th3} and \ref{Th4} we address it in this paper and \cite{fil-2016}.

A weaker version of the case $D=1$ of Theorem \ref{Th3} appears in \cite%
{f-2005}. There, each $K\left( x\right) $ is an interval $\left[ f\left(
x\right) -\varepsilon \left( x\right) ,f\left( x\right) +\varepsilon \left(
x\right) \right] $. In place of the conclusion $F\left( x\right) \in K\left(
x\right) $ in Theorem \ref{Th3}, \cite{f-2005} obtains the weaker conclusion 
$F\left( x\right) \in \left[ f\left( x\right) -C\varepsilon \left( x\right)
,f\left( x\right) +C\varepsilon \left( x\right) \right] $ for a constant $C$
determined by $m$, $n$.

Our interest in Theorems \ref{Th3} and \ref{Th4} arises in part from their
possible connection to the interpolation algorithms of Fefferman-Klartag 
\cite{fb1,fb2}. Given a function $f:E\rightarrow \mathbb{R}$ with $E\subset 
\mathbb{R}^{n}$ finite, the goal of \cite{fb1,fb2} is to compute a function $%
F\in C^{m}\left( \mathbb{R}^{n}\right) $ such that $F=f$ on $E$, with norm $%
\left\Vert F\right\Vert _{C^{m}\left( \mathbb{R}^{n}\right) }$ as small as
possible up to a factor $C\left( m,n\right) $. Roughly speaking, the
algorithm in \cite{fb1,fb2} computes such an $F$ using $O\left( N\log
N\right) $ computer operations, where $N=\#\left( E\right) $. The algorithm
is based on ideas from the proof \cite{f-2005} of Theorem \ref{Th1}.
Accordingly, Theorems \ref{Th3} and \ref{Th4} raise the hope that we can
start to understand constrained interpolation problems, in which e.g. the
interpolant $F$ is constrained to be nonnegative everywhere on $\mathbb{R}%
^{n}$.

Theorems \ref{Th3} and \ref{Th4} follow from a more general result on $C^m$ functions whose Taylor polynomials belong to prescribed convex sets. More precisely, let $J_x(F)$ denote the $(m-1)^{\text{rst}}$ degree Taylor polynomial of $F$ at $x$; thus, $J_x(F)$ belongs to the vector space of $\mathcal{P}$ of all such polynomials. Let $E \subset \mathbb{R}^n$ be finite. For each $x \in E$ and $M >0$, let $\Gamma(x,M)$ be a convex subset of $\mathcal{P}$. Under suitable hypotheses on the $\Gamma(x,M)$ we will prove the following:

\theoremstyle{plain} \newtheorem*{thm finiteneness}{Finiteness Principle}%
\begin{thm finiteneness}
Fix $M_0 >0$. Suppose that for each $S \subseteq E$ with $\#(S) \leq k^{\#}(m,n)$ there exists $F^S \in C^m(\mathbb{R}^n)$ with $||F^S||_{C^m(\mathbb{R}^n)} \leq M_0$, such that $J_x(F^S) \in \Gamma(x,M_0)$ for all $x \in S$. 

Then there exists $F \in C^m(\mathbb{R}^n)$ with $||F||_{C^m(\mathbb{R}^n)} \leq CM_0$, such that $J_x(F^S) \in \Gamma(x,CM_0)$ for all $x \in E$.

Here, $C$ depends only on $m$, $n$ and the constants in our assumptions on the $\Gamma(x,M)$.  

\end{thm finiteneness}

See Section \ref{fp-i} below for our assumptions on the $\Gamma(x,M)$. The precise statement of the above finiteness principle is given by Theorem \ref{theorem-fp-for-wsf} in Section \ref{fp-i} and the remark after its proof. That result and the closely related Theorem \ref{theorem-tu1} in Section \ref{tu} form the real content of this paper.

A special case of Theorem \ref{theorem-fp-for-wsf} was proven in \cite{f-2005-a}. In that special case, the sets $\Gamma(x,M)$ have the form $\Gamma(x,M)=f_x +M \cdot \sigma(x)$ with $f_x \in \mathcal{P}$ and $\sigma(x)$ a symmetric convex subset of $\mathcal{P}$. The $\sigma(x)$ are required to be ``Whitney convex"; see \cite{f-2005-a}.

Here, to prove Theorems \ref{theorem-tu1} and \ref{theorem-fp-for-wsf}, we adapt the arguments in \cite{f-2005-a} to more general families of convex sets $\Gamma(x,M)$. In particular, our hypothesis of ``$(C_w,\delta_{\max})$-convexity", formulated in Section \ref{shape-fields} below, generalizes the notion of Whitney convexity to the present context.

However, at one crucial point (Case 2 in the proof
of Lemma \ref{lemma-gn1} below) the argument has no analogue in \cite{f-2005-a}.

We refer the reader to the expository paper \cite{f-2009-b} for an explanation of the main ideas in \cite{f-2005-a}.

This paper is part of a literature on extension, interpolation, and
selection of functions, going back to H. Whitney's seminal work \cite%
{whitney-1934}, and including fundamental contributions by G. Glaeser \cite%
{gl-1958}, Y, Brudnyi and P. Shvartsman \cite%
{bs-1985, bs-1994,bs-1997,bs-1998,bs-2001, pavel-1982,pavel-1984,pavel-1986, pavel-1987, pavel-1992, s-2001,s-2002,pavel-2004, s-2008}, J. Wells 
\cite{jw-1973}, E. Le Gruyer \cite{lgw-2009}, and E. Bierstone, P. Milman,
and W. Paw{\l }ucki \cite{bmp-2003,bmp-2006,bm-2007}, as well as our own
papers \cite{f-2005, f-2006, f-2007, f-2009-b, fb1, fb2, f-2005-a, fl-2014}. See e.g. \cite{f-2009-b}
for the history of the problem, as well as Zobin \cite{zobin-1998,zobin-1999} for a related problem. 

We are grateful to the American Institute of Mathematics, the Banff
International Research Station, the Fields Institute, and the College of
William and Mary for hosting workshops on interpolation and extension. We
are grateful also to the Air Force Office of Scientific Research, the
National Science Foundation, the Office of Naval Research, and the U.S.-Israel Binational Science Foundation for financial support.

We are also grateful to Pavel Shvartsman and Alex Brudnyi for their comments on an earlier version of our manuscript, and to Bo'az Klartag and all the participants of the Eighth Whitney Problems Workshop for their interest in our work. Finally, we are grateful to the referee for careful, expert reading, and for comments that improved our paper. 

\part{Shape Fields and Their Refinements}

\section{Notation and Preliminaries\label{notation-and-preliminaries}}

Fix $m$, $n\geq 1$. We will work with cubes in $\mathbb{R}^{n}$; all our
cubes have sides parallel to the coordinate axes. If $Q$ is a cube, then $%
\delta _{Q}$ denotes the sidelength of $Q$. For real numbers $A>0$, $AQ$
denotes the cube whose center is that of $Q$, and whose sidelength is $%
A\delta _{Q}$.

A \underline{dyadic} cube is a cube of the form $I_{1}\times I_{2}\times
\cdots \times I_{n}\subset \mathbb{R}^{n}$, where each $I_{\nu }$ has the
form $[2^{k}\cdot i_{\nu },2^{k}\cdot \left( i_{\nu }+1\right) )$ for
integers $i_{1},\cdots ,i_{n}$, $k$. Each dyadic cube $Q$ is contained in
one and only one dyadic cube with sidelength $2\delta _{Q}$; that cube is
denoted by $Q^{+}$.

We write $\mathcal{P}$ to denote the vector space of all real-valued
polynomials of degree at most $\left( m-1\right) $ on $\mathbb{R}^{n}$. If $%
x\in \mathbb{R}^{n}$ and $F$ is a real-valued $C^{m-1}$ function on a
neighborhood of $x$, then $J_{x}\left( F\right) $ (the \textquotedblleft
jet" of $F$ at $x$) denotes the $\left( m-1\right) ^{rst}$ order Taylor
polynomial of $F$ at $x$. Thus, $J_{x}\left( F\right) \in \mathcal{P}$.

For each $x\in \mathbb{R}^{n}$, there is a natural multiplication $\odot
_{x} $ on $\mathcal{P}$ (\textquotedblleft multiplication of jets at $x$")
defined by setting%
\begin{equation*}
P\odot _{x}Q=J_{x}\left( PQ\right) \text{ for }P,Q\in \mathcal{P}\text{.}
\end{equation*}%
We write $C^{m}\left( \mathbb{R}^{n}\right) $ to denote the Banach space of
real-valued $C^{m}$ functions $F$ on $\mathbb{R}^{n}$ for which the
norm 
\begin{equation*}
\left\Vert F\right\Vert _{C^{m}\left( \mathbb{R}^{n}\right) }=\sup_{x\in 
\mathbb{R}^{n}}\max_{\left\vert \alpha \right\vert \leq m}\left\vert
\partial ^{\alpha }F\left( x\right) \right\vert
\end{equation*}%
is finite. For $D\geq 1$, we write $C^{m}\left( \mathbb{R}^{n},%
\mathbb{R}^{D}\right) $ to denote the Banach space of all $\mathbb{R}^{D}$%
-valued $C^{m}$ functions $F$ on $\mathbb{R}^{n}$, for which the
norm 
\begin{equation*}
\left\Vert F\right\Vert _{C^{m}\left( \mathbb{R}^{n},\mathbb{R}^{D}\right)
}=\sup_{x\in \mathbb{R}^{n}}\max_{\left\vert \alpha \right\vert \leq
m}\left\Vert \partial ^{\alpha }F\left( x\right) \right\Vert
\end{equation*}%
is finite. Here, we use the Euclidean norm on $\mathbb{R}^{D}$.

If $F$ is a real-valued function on a cube $Q$, then we write $F\in
C^{m}\left( Q\right) $ to denote that $F$ and its derivatives up to $m$-th
order extend continuously to the closure of $Q$. For $F\in C^{m}\left(
Q\right) $, we define 
\begin{equation*}
\left\Vert F\right\Vert _{C^{m}\left( Q\right) }=\sup_{x\in
Q}\max_{\left\vert \alpha \right\vert \leq m}\left\vert \partial ^{\alpha
}F\left( x\right) \right\vert .
\end{equation*}%
Similarly, if $F$ is an $\mathbb{R}^{D}$-valued function on a cube $Q$,
then we write $F\in C^{m}\left( Q,\mathbb{R}^{D}\right) $ to denote that $F$
and its derivatives up to $m$-th order extend continuously to the closure of 
$Q$. For $F\in C^{m}\left( Q,\mathbb{R}^{D}\right) $, we define 
\begin{equation*}
\left\Vert F\right\Vert _{C^{m}\left( Q,\mathbb{R}^{D}\right) }=\sup_{x\in
Q}\max_{\left\vert \alpha \right\vert \leq m}\left\Vert \partial ^{\alpha
}F\left( x\right) \right\Vert \text{,}
\end{equation*}%
where again we use the Euclidean norm on $\mathbb{R}^{D}$.

If $F\in C^{m}\left( Q\right) $ and $x$ belongs to the boundary of $Q$, then
we still write $J_{x}\left( F\right) $ to denote the $\left( m-1\right)
^{rst}$ degree Taylor polynomial of $F$ at $x$, even though $F$ isn't
defined on a full neighborhood of $x\in \mathbb{R}^{n}$.

Let $S\subset \mathbb{R}^{n}$ be non-empty and finite. A \underline{Whitney
field} on $S$ is a family of polynomials 
\begin{equation*}
\vec{P}=\left( P^{y}\right) _{y\in S}\text{ (each }P^{y}\in \mathcal{P}\text{%
),}
\end{equation*}%
parametrized by the points of $S$.

We write $Wh\left( S\right) $ to denote the vector space of all Whitney
fields on $S$.

For $\vec{P}=\left( P^{y}\right) _{y\in S}\in Wh\left( S\right) $, we define
the seminorm 
\begin{equation*}
\left\Vert \vec{P}\right\Vert _{\dot{C}^{m}\left( S\right) }=\max_{x,y\in
S,\left( x\not=y\right), |\alpha| \leq m}\frac{\left\vert \partial ^{\alpha
}\left( P^{x}-P^{y}\right) \left( x\right) \right\vert }{\left\vert
x-y\right\vert ^{m-\left\vert \alpha \right\vert }}\text{.}
\end{equation*}

(If $S$ consists of a single point, then $\left\Vert \vec{P}\right\Vert _{%
\dot{C}^{m}\left( S\right) }=0$.)

We write $\mathcal{M}$ to denote the set of all multiindices $\alpha =\left(
\alpha _{1},\cdots ,\alpha _{n}\right) $ of order $\left\vert \alpha
\right\vert =\alpha _{1}+\cdots +\alpha _{n} \leq m-1$.

We define a (total) order relation $<$ on $\mathcal{M}$, as follows. Let $%
\alpha =\left( \alpha _{1},\cdots ,\alpha _{n}\right) $ and $\beta =\left(
\beta _{1},\cdots ,\beta _{n}\right) $ be distinct elements of $\mathcal{M}$%
. Pick the largest $k$ for which $\alpha _{1}+\cdots +\alpha
_{k}\not=\beta _{1}+\cdots +\beta _{k}$. (There must be at least one such $k$%
, since $\alpha $ and $\beta $ are distinct). Then we say that $\alpha
<\beta $ if $\alpha _{1}+\cdots +\alpha _{k}<\beta _{1}+\cdots +\beta _{k}$.

We also define a (total) order relation $<$ on subsets of $\mathcal{M}$, as
follows. Let $\mathcal{A},\mathcal{B}$ be distinct subsets of $\mathcal{M}$,
and let $\gamma $ be the least element of the symmetric difference $\left( 
\mathcal{A\setminus B}\right) \cup \left( \mathcal{B\setminus A}\right) $
(under the above order on the elements of $\mathcal{M}$). Then we say that $%
\mathcal{A}<\mathcal{B}$ if $\gamma \in \mathcal{A}$.

One checks easily that the above relations $<$ are indeed total order
relations. Note that $\mathcal{M}$ is minimal, and the empty set $\emptyset $
is maximal under $<$. A set $\mathcal{A}\subseteq \mathcal{M}$ is called 
\underline{monotonic} if, for all $\alpha \in \mathcal{A}$ and $\gamma \in 
\mathcal{M}$, $\alpha +\gamma \in \mathcal{M}$ implies $\alpha +\gamma \in 
\mathcal{A}$. We make repeated use of a simple observation:

Suppose $\mathcal{A}\subseteq \mathcal{M}$ is monotonic, $P\in \mathcal{P}$
and $x_{0}\in \mathbb{R}^{n}$. If $\partial ^{\alpha }P\left( x_{0}\right)
=0 $ for all $\alpha \in \mathcal{A}$, then $\partial ^{\alpha }P\equiv 0$
on $\mathbb{R}^{n}$ for $\alpha \in \mathcal{A}$. 

This follows by writing $\partial ^{\alpha }P\left( y\right)
=\sum_{\left\vert \gamma \right\vert \leq m-1-\left\vert \alpha \right\vert }%
\frac{1}{\gamma !}\partial ^{\alpha +\gamma }P\left( x_{0}\right) \cdot
\left( y-x_{0}\right) ^{\gamma }$ and noting that all the relevant $\alpha
+\gamma $ belong to $\mathcal{A}$, hence $\partial ^{\alpha +\gamma }P\left(
x_{0}\right) =0$.

We need a few elementary facts about convex sets. We recall

\theoremstyle{plain} \newtheorem*{thm Helly}{Helly's Theorem}%
\begin{thm Helly}
Let $K_{1},\cdots ,K_{N}\subset \mathbb{R}^{D}$ be convex. Suppose that $K_{i_{1}}\cap \cdots \cap K_{i_{D+1}}$ is nonempty for any $i_{1}, \cdots, i_{D+1}\in
\{1,\cdots ,N\}$. Then $K_{1}\cap \cdots \cap K_{N}$ is nonempty.\end{thm Helly}

See \cite{rock-convex}.

We also use the following \theoremstyle{plain} 
\newtheorem*{thm Trivial
Remark on Convex Sets}{Trivial Remark on Convex Sets}%
\begin{thm Trivial Remark on Convex Sets}
Let $\Gamma $ be a convex set, and let $P_{0}$, $P_{0}~+~P_{\nu }$, $P_{0}~-~P_{\nu }\in \Gamma $ for $\nu =1,\cdots ,\nu _{\max }$. Then for any
real numbers $t_{1},\cdots ,t_{\nu _{\max }}$ with 
\begin{equation*}
\sum_{\nu =1}^{\nu _{\max }}\left\vert t_{\nu }\right\vert \leq 1,
\end{equation*}we have 
\begin{equation*}
P_{0}+\sum_{\nu =1}^{\nu _{\max }}t_{\nu }P_{\nu }\in \Gamma \text{.}
\end{equation*}\end{thm Trivial Remark on Convex Sets}

If $\lambda =\left( \lambda _{1},\cdots ,\lambda _{n}\right) $ is an $n$%
-tuple of positive real numbers, and if $\beta =\left( \beta _{1},\cdots
,\beta _{n}\right) \in \mathbb{Z}^{n}$, then we write $\lambda ^{\beta }$ to
denote 
\begin{equation*}
\lambda _{1}^{\beta _{1}}\cdots \lambda _{n}^{\beta _{n}}.
\end{equation*}%
We write $B_{n}\left( x,r\right) $ to denote the open ball in $\mathbb{R}%
^{n} $ with center $x$ and radius $r$, with respect to the Euclidean metric. 

\section{Shape Fields}\label{shape-fields}

Let $E\subset \mathbb{R}^{n}$ be finite. For each $x\in E$, $M\in \left(
0,\infty \right) $, let $\Gamma \left( x,M\right) \subseteq \mathcal{P}$ be
a (possibly empty) convex set. We say that $\vec{\Gamma}=\left( \Gamma
\left( x,M\right) \right) _{x\in E,M>0}$ is a \underline{shape field} if for
all $x\in E$ and $0<M^{\prime }\leq M<\infty $, we have 
\begin{equation*}
\Gamma \left( x,M^{\prime }\right) \subseteq \Gamma \left( x,M\right) .
\end{equation*}

Let $\vec{\Gamma}=\left( \Gamma \left( x,M\right) \right) _{x\in E,M>0}$ be
a shape field and let $C_{w},\delta _{\max }$ be positive real numbers. We
say that $\vec{\Gamma}$ is \underline{$\left( C_{w},\delta _{\max }\right) $%
-convex} if the following condition holds:

Let $0<\delta \leq \delta _{\max }$, $x\in E$, $M\in \left( 0,\infty \right) 
$, $P_{1}$, $P_{2}$, $Q_{1}$, $Q_{2}\in \mathcal{P}$. Assume that

\begin{itemize}
\item[\refstepcounter{equation}\text{(\theequation)}\label{wsf1}] $P_1,P_2
\in\Gamma(x,M)$;

\item[\refstepcounter{equation}\text{(\theequation)}\label{wsf2}] $%
|\partial^\beta(P_1-P_2)(x)| \leq M\delta^{m-|\beta|}$ for $|\beta| \leq m-1$%
;

\item[\refstepcounter{equation}\text{(\theequation)}\label{wsf3}] $%
|\partial^\beta Q_i(x)|\leq \delta^{-|\beta|}$ for $|\beta| \leq m-1$ for $%
i=1,2$;

\item[\refstepcounter{equation}\text{(\theequation)}\label{wsf4}] $%
Q_1\odot_x Q_1 + Q_2 \odot_x Q_2 =1$.
\end{itemize}

Then

\begin{itemize}
\item[\refstepcounter{equation}\text{(\theequation)}\label{wsf5}] $%
P:=Q_1\odot_x Q_1\odot_x P_1 + Q_2 \odot_x Q_2 \odot_x P_2\in\Gamma(x,C_wM)$.
\end{itemize}

The following lemma follows easily from the definition of $(C_w,\delta_{\max})$%
-convexity.

\begin{lemma}
\label{lemma-wsf1} Suppose $\vec{\Gamma}=\left( \Gamma \left( x,M\right)
\right) _{x\in E,M>0}$ is a $\left( C_{w},\delta _{\max }\right) $-convex
shape field. Let

\begin{itemize}
\item[\refstepcounter{equation}\text{(\theequation)}\label{6}] $0<\delta
\leq \delta_{\max}$, $x \in E$, $M>0$, $P_1,P_2,Q_1,Q_2 \in \mathcal{P}$ and 
$A^{\prime },A^{\prime \prime }>0$.
\end{itemize}

Assume that

\begin{itemize}
\item[\refstepcounter{equation}\text{(\theequation)}\label{7}] $P_1,P_2 \in
\Gamma(x,A^{\prime }M)$;

\item[\refstepcounter{equation}\text{(\theequation)}\label{8}] $\left\vert
\partial ^{\beta }\left( P_{1}-P_{2}\right) \left( x\right) \right\vert \leq
A^{\prime}M \delta^{  m-\left\vert \beta \right\vert }$ for $\left\vert \beta
\right\vert \leq m-1$;

\item[\refstepcounter{equation}\text{(\theequation)}\label{9}] $\left\vert
\partial ^{\beta }Q_{i}\left( x\right) \right\vert \leq A^{\prime \prime 
}\delta^{-\left\vert \beta \right\vert }$ for $\left\vert \beta \right\vert \leq m-1$
and $i=1,2$;

\item[\refstepcounter{equation}\text{(\theequation)}\label{10}] $Q_{1}\odot
_{x}Q_{1}+Q_{2}\odot _{x}Q_{2}=1$.
\end{itemize}

Then

\begin{itemize}
\item[\refstepcounter{equation}\text{(\theequation)}\label{11}] $%
P:=Q_{1}\odot _{x}Q_{1}\odot _{x}P_{1}+Q_{2}\odot _{x}Q_{2}\odot
_{x}P_{2}\in \Gamma \left( x,CM\right) $ with $C$ determined by $A^{\prime }$%
, $A^{\prime \prime }$, $C_{w}$, $m$, and $n$.
\end{itemize}
\end{lemma}

By Lemma \ref{lemma-wsf1} and an induction argument, we also have the following result.

\begin{lemma}
\label{lemma-wsf2} Suppose $\vec{\Gamma}=\left( \Gamma \left( x,M\right)
\right) _{x\in E,M>0}$ is a $\left( C_{w},\delta _{\max }\right) $-convex
shape field. Let

\begin{itemize}
\item[\refstepcounter{equation}\text{(\theequation)}\label{12}] $0<\delta
\leq \delta _{\max }$, $x\in E$, $M>0,A^{\prime },A^{\prime \prime }>0$, $%
P_{1},\cdots P_{k},Q_{1},\cdots ,Q_{k}\in \mathcal{P}$.
\end{itemize}

Assume that

\begin{itemize}
\item[\refstepcounter{equation}\text{(\theequation)}\label{13}] $P_{i}\in
\Gamma \left( x,A^{\prime }M\right) $ for $i=1,\cdots ,k$;

\item[\refstepcounter{equation}\text{(\theequation)}\label{14}] $\left\vert
\partial ^{\beta }\left( P_{i}-P_{j}\right) \left( x\right) \right\vert \leq
A^{\prime}M\delta^{  m-\left\vert \beta \right\vert }$ for $\left\vert \beta
\right\vert \leq m-1$, $i,j=1,\cdots ,k$;

\item[\refstepcounter{equation}\text{(\theequation)}\label{15}] $\left\vert
\partial ^{\beta }Q_{i}\left( x\right) \right\vert \leq A^{\prime \prime }\delta^{
-\left\vert \beta \right\vert }$ for $\left\vert \beta \right\vert \leq m-1$
and $i=1,\cdots ,k $;

\item[\refstepcounter{equation}\text{(\theequation)}\label{16}] $%
\sum_{i=1}^{k}Q_{i}\odot _{x}Q_{i}=1$.
\end{itemize}

Then

\begin{itemize}
\item[\refstepcounter{equation}\text{(\theequation)}\label{17}] $%
\sum_{i=1}^kQ_{i}\odot _{x}Q_{i}\odot _{x}P_{i}\in \Gamma \left( x,CM\right)
, $ with $C$ determined by $A^{\prime }$, $A^{\prime \prime }$, $C_{w}$, $m$%
, $n$, $k$.
\end{itemize}
\end{lemma}

Next, we define the \underline{first refinement} of a shape field $\vec{%
\Gamma}=\left( \Gamma \left( x,M\right) \right) _{x\in E,M>0}$ to be $\vec{%
\Gamma}^{\#}=\left( \Gamma ^{\#}\left( x,M\right) \right) _{x\in E,M>0}$,
where $\Gamma ^{\#}\left( x,M\right) $ consists of those $P^{\#}\in \mathcal{%
P}$ such that for all $y\in E$ there exists $P\in \Gamma \left( y,M\right) $
for which

\begin{itemize}
\item[\refstepcounter{equation}\text{(\theequation)}\label{33}] $\left\vert
\partial ^{\beta }\left( P^{\#}-P\right) \left( x\right) \right\vert \leq
M\left\vert x-y\right\vert ^{m-\left\vert \beta \right\vert }$ for $%
\left\vert \beta \right\vert \leq m-1$.
\end{itemize}

Note that each $\Gamma ^{\#}\left( x,M\right) $ is a (possibly empty) convex
subset of $\mathcal{P}$, and that $M^{\prime }\leq M$ implies $\Gamma
^{\#}\left( x,M^{\prime }\right) \subseteq \Gamma ^{\#}\left( x,M\right) $.
Thus, $\vec{\Gamma}^{\#}$ is again a shape field. Taking $y=x$ in \eqref{33}%
, we see that $\Gamma^\#(x,M) \subset \Gamma(x,M)$.

\begin{lemma}
\label{lemma-wsf3} Let $\vec{\Gamma}=\left( \Gamma \left( x,M\right) \right)
_{x\in E,M>0}$ be a $\left( C_{w},\delta _{\max }\right) $-convex shape
field, and let $\vec{\Gamma}^{\#}=\left( \Gamma ^{\#}\left( x,M\right) \right)
_{x\in E,M>0}$ be the first refinement of $\vec{\Gamma}^{\#}$. Then $\vec{%
\Gamma}^{\#}$ is $\left( C,\delta _{\max }\right) $-convex, where $C$ is
determined by $C_{w}$, $m$, $n$.
\end{lemma}

\begin{proof}
We write $c$, $C$, $C^{\prime }$, etc., to denote constants determined by $%
C_w$, $m$, $n$. These symbols may denote different constants in different
occurrences. Let $0 < \delta \leq \delta_{\max}$, $M > 0$, $x \in E$, $%
P_1^\#, P_2^\#, Q_1^\#, Q_2^\# \in \mathcal{P}$, and assume that

\begin{itemize}
\item[\refstepcounter{equation}\text{(\theequation)}\label{34}] $P_i^\# \in
\Gamma^\#(x,M)$ for $i=1,2$;
\end{itemize}

\begin{itemize}
\item[\refstepcounter{equation}\text{(\theequation)}\label{35}] $\left\vert
\partial ^{\beta }\left( P_{1}^{\#}-P_{2}^{\#}\right) \left( x\right)
\right\vert \leq M\delta ^{m-\left\vert \beta \right\vert }$ for $\left\vert
\beta \right\vert \leq m-1$;

\item[\refstepcounter{equation}\text{(\theequation)}\label{36}] $\left\vert
\partial ^{\beta }Q_{i}^{\#}\left( x\right) \right\vert \leq \delta
^{-\left\vert \beta \right\vert }$ for $\left\vert \beta \right\vert \leq
m-1 $, $i=1,2$; and

\item[\refstepcounter{equation}\text{(\theequation)}\label{37}] $%
Q_{1}^{\#}\odot _{x}Q_{1}^{\#}+Q_{2}^{\#}\odot _{x}Q_{2}^{\#}=1$.
\end{itemize}

Under the above assumptions, we must show that%
\begin{equation}
P^\#:=Q_{1}^{\#}\odot _{x}Q_{1}^{\#}\odot _{x}P_{1}^{\#}+Q_{2}^{\#}\odot
_{x}Q_{2}^{\#}\odot _{x}P_{2}^{\#}\in \Gamma ^{\#}\left( x,CM\right) \text{.} \label{40}
\end{equation}%
By definition of $\Gamma ^{\#}\left( \cdot ,\cdot \right) $, this means that
given any $y\in E$ there exists

\begin{itemize}
\item[\refstepcounter{equation}\text{(\theequation)}\label{38}] $P \in
\Gamma(y,CM)$ such that
\end{itemize}

\begin{itemize}
\item[\refstepcounter{equation}\text{(\theequation)}\label{39}] $%
|\partial^\beta(P^\#-P)(x)| \leq CM|x-y|^{m-|\beta|}$ for $|\beta| \leq m-1$%
, where
\end{itemize}

Thus, to prove Lemma \ref{lemma-wsf3}, we must prove that there exists $P$
satisfying \eqref{38}, \eqref{39}, under the assumptions \eqref{34}$\cdots$%
\eqref{37}. To do so, we start by
defining the functions

\begin{itemize}
\item[\refstepcounter{equation}\text{(\theequation)}\label{41}] $\theta _{i}=%
\frac{Q_{i}^{\#}}{\left[ \left( Q_{1}^{\#}\right) ^{2}+\left(
Q_{2}^{\#}\right) ^{2}\right] ^{1/2}}$ on $B_{n}(x,c_{0}\delta )$ ($i=1,2$).
\end{itemize}

We pick $c_0 <1$ small enough so that \eqref{36}, \eqref{37} guarantee that $%
\theta_i$ is well-defined on $B_n(x,c_0\delta)$ and satisfies

\begin{itemize}
\item[\refstepcounter{equation}\text{(\theequation)}\label{42}] $|\partial
^{\beta }\theta _{i}|\leq C\delta ^{-\left\vert \beta \right\vert }$ on $%
B_{n}\left( x,c_{0}\delta \right) $ for $\left\vert \beta \right\vert \leq m$%
, $i=1,2$,
\end{itemize}

and

\begin{itemize}
\item[\refstepcounter{equation}\text{(\theequation)}\label{43}] $%
\theta_1^2+\theta_2^2=1$ on $B_n(x,c_0\delta)$.
\end{itemize}

Also

\begin{itemize}
\item[\refstepcounter{equation}\text{(\theequation)}\label{44}] $%
J_x(\theta_i)=Q_i^\#$ for $i=1,2$,
\end{itemize}

thanks to \eqref{37}.

We now divide the discussion of \eqref{38}, \eqref{39} into two
cases.

\underline{CASE 1: Suppose $y\in B_n(x,c_0\delta)$.}

By \eqref{34} and the definition of $\Gamma^\#(\cdot, \cdot)$, there exist

\begin{itemize}
\item[\refstepcounter{equation}\text{(\theequation)}\label{45}] $P_i \in
\Gamma(y,M)$ ($i=1,2$)
\end{itemize}

satisfying

\begin{itemize}
\item[\refstepcounter{equation}\text{(\theequation)}\label{46}] $%
|\partial^\beta(P_i^\# - P_i)(x)| \leq M|x-y|^{m- |\beta|}$ for $|\beta|
\leq m-1$, $i=1,2$.
\end{itemize}

Since we are in CASE 1, estimates \eqref{35} and \eqref{46} together imply
that

\begin{itemize}
\item[\refstepcounter{equation}\text{(\theequation)}\label{47}] $%
|\partial^\beta(P_1-P_2)(x)| \leq CM\delta^{m-|\beta|}$ for $|\beta| \leq
m-1 $.
\end{itemize}

Consequently,

\begin{itemize}
\item[\refstepcounter{equation}\text{(\theequation)}\label{48}] $%
|\partial^\beta(P_1-P_2)|\leq CM \delta^{m-|\beta|}$ on $B_n(x,c_0 \delta)$
for $|\beta| \leq m$.
\end{itemize}

(Recall that $P_1,P_2$ are polynomials of degree at most $m-1$.) 

Thanks to \eqref{42}, \eqref{43},  \eqref{45}, \eqref{48}, and the $%
\left( C_{w},\delta _{\max }\right) $-convexity of $\vec{\Gamma}$, we may
apply Lemma \ref{lemma-wsf1} to the polynomials $P_{1}$, $P_{2}$, $Q_{1}$, $%
Q_{2}$, where $Q_{i}=J_{y}\left( \theta _{i}\right) $ for $i=1,2$. This
tells us that

\begin{itemize}
\item[\refstepcounter{equation}\text{(\theequation)}\label{50}] $%
P:=J_{y}\left( \theta _{1}^{2}P_{1}+\theta _{2}^{2}P_{2}\right) \in \Gamma
\left( y,CM\right) $.
\end{itemize}

That is, the $P$ in \eqref{50} satisfies \eqref{38}. We will show that it
also satisfies \eqref{39}.

Thanks to \eqref{40}, \eqref{44}, we have

\begin{itemize}
\item[\refstepcounter{equation}\text{(\theequation)}\label{51}] $P^\# =
J_x(\theta_1^2 P_1^\#+\theta_2^2 P_2^\#)$.
\end{itemize}

In view of \eqref{50}, \eqref{51}, our desired estimate \eqref{39} is
equivalent to the following:

\begin{itemize}
\item[\refstepcounter{equation}\text{(\theequation)}\label{52}] $\left\vert
\partial ^{\beta }\left( \theta _{1}^{2}P_{1}^{\#}+\theta
_{2}^{2}P_{2}^{\#}-J_{y}\left( \theta _{1}^{2}P_{1}+\theta
_{2}^{2}P_{2}\right) \right) \left( x\right) \right\vert \leq CM\left\vert
x-y\right\vert ^{m-\left\vert \beta \right\vert }$ 
\end{itemize}
for $\left\vert \beta
\right\vert \leq m-1$.

Thus, we have reduced the existence of $P$ satisfying \eqref{38}, \eqref{39}
in CASE 1 to the task of proving \eqref{52}.

Since $\theta _{1}^{2}+\theta _{2}^{2}=1$ (see \eqref{43}) and $%
J_{y}P_{1}=P_{1}$, the following holds on $B_n(x,c_0\delta)$: 
\begin{eqnarray*}
&&\left( \theta _{1}^{2}P_{1}^{\#}+\theta _{2}^{2}P_{2}^{\#}-J_{y}\left(
\theta _{1}^{2}P_{1}+\theta _{2}^{2}P_{2}\right) \right) \\
&=&\theta _{1}^{2}\left( P_{1}^{\#}-P_{1}\right) +\theta _{2}^{2}\left(
P_{2}^{\#}-P_{2}\right)+\left[ \theta _{2}^{2}\left( P_{2}-P_{1}\right) -J_{y}\left( \theta
_{2}^{2}\left( P_{2}-P_{1}\right) \right) \right] .
\end{eqnarray*}%
Consequently, the desired estimate (\ref{52}) will follow if we can show
that 
\begin{equation}
\left\vert \partial ^{\beta }\left[ \theta _{i}^2\left(
P_{i}^{\#}-P_{i}\right) \left( x\right) \right] \right\vert \leq
CM\left\vert x-y\right\vert ^{m-\left\vert \beta \right\vert }\text{ for }%
\left\vert \beta \right\vert \leq m-1\text{, }i=1\text{,}2\text{,}
\label{53}
\end{equation}%
and 
\begin{equation}
\left\vert \partial ^{\beta }\left[ \theta _{2}^{2}\left( P_{1}-P_{2}\right)
-J_{y}\left( \theta _{2}^{2}\left( P_{1}-P_{2}\right) \right) \right] \left(
x\right) \right\vert \leq CM\left\vert x-y\right\vert ^{m-\left\vert \beta
\right\vert }\text{ for }\left\vert \beta \right\vert \leq m-1.  \label{54}
\end{equation}

Moreover, (\ref{53}) follows at once from (\ref{42}) and (\ref{46}), since $%
\delta ^{-\left\vert \beta \right\vert }\leq C\left\vert x-y\right\vert
^{-\left\vert \beta \right\vert }$ in CASE\ 1.

To check (\ref{54}), we apply (\ref{42}) and (\ref{48}) to deduce that 
$$\left\vert \partial ^{\beta }\left[ \theta _{2}^{2}\left(
P_{1}-P_{2}\right) \right] \right\vert \leq CM$$ on $B\left( x,c_{0}\delta
\right) $ for $\left\vert \beta \right\vert =m.$

Therefore, (\ref{54}) follows from Taylor's theorem.

This proves the existence of a $P$ satisfying (\ref{38}), (\ref{39}) in\
CASE 1.

\underline{CASE 2: Suppose that $y\not\in B_{n}(x,c_{0}\delta )$.}

Since $P_{1}^{\#}\in \Gamma \left( x, M\right) $ (see (\ref{34})), there
exists

\begin{itemize}
\item[\refstepcounter{equation}\text{(\theequation)}\label{55}] $P_{1}\in
\Gamma \left( y,M\right) $
\end{itemize}

such that

\begin{itemize}
\item[\refstepcounter{equation}\text{(\theequation)}\label{56}] $\left\vert
\partial ^{\beta }\left( P_{1}^{\#}-P_{1}\right) \left( x\right) \right\vert
\leq M\left\vert x-y\right\vert ^{m-\left\vert \beta \right\vert }$ for $%
\left\vert \beta \right\vert \leq m-1$.
\end{itemize}

Thanks to (\ref{37}), we may rewrite $P^\#$ in the form $$%
P^{\#}=P_{1}^{\#}+Q_{2}^{\#}\odot _{x}Q_{2}^{\#}\odot _{x}\left(
P_{2}^{\#}-P_{1}^{\#}\right).$$ Our assumptions (\ref{35}), (\ref{36})
therefore yield the estimates $$\left\vert \partial ^{\beta }\left(
P^{\#}-P_{1}^{\#}\right) \left( x\right) \right\vert \leq CM\delta
^{m-\left\vert \beta \right\vert }$$ for $\left\vert \beta \right\vert \leq
m-1$.

Since we are in CASE 2, it follows that

\begin{itemize}
\item[\refstepcounter{equation}\text{(\theequation)}\label{57}] $\left\vert
\partial ^{\beta }\left( P^{\#}-P_{1}^{\#}\right) \left( x\right)
\right\vert \leq CM\left\vert x-y\right\vert ^{m-\left\vert \beta
\right\vert }$ for $\left\vert \beta \right\vert \leq m-1$.
\end{itemize}

From \eqref{56} and \eqref{57}, we learn that

\begin{itemize}
\item[\refstepcounter{equation}\text{(\theequation)}\label{58}] $\left\vert
\partial ^{\beta }\left( P^{\#}-P_{1}\right) \left( x\right) \right\vert
\leq CM\left\vert x-y\right\vert ^{m-\beta }$ for $\left\vert \beta
\right\vert \leq m-1$.
\end{itemize}

We now know from \eqref{55} and \eqref{58} that $P:=P_1$ satisfies \eqref{38}
and \eqref{39}. Thus, in CASE 2 we again have a polynomial $P$ satisfying %
\eqref{38} and \eqref{39}. We have seen in all cases that there exists $P
\in \mathcal{P}$ satisfying \eqref{38}, \eqref{39}.

The proof of Lemma \ref{lemma-wsf3} is complete.
\end{proof}

Next we define the higher refinements of a given shape field $\vec{\Gamma}_0
= \left( \Gamma_0(x,M) \right)_{x\in E,M>0}$. By induction on $l \geq 0$, we
define $\vec{\Gamma}_l = \left(\Gamma_l(x,M) \right)_{x \in E, M>0}$; to do
so, we start with our given $\vec{\Gamma}_0$, and define $\vec{\Gamma}_{l+1}$
to be the first refinement of $\vec{\Gamma}_l$, for each $l \geq 0$. Thus,
each $\vec{\Gamma}_l$ is a shape field.

By the definition of the first refinement and Lemma \ref{lemma-wsf3}, we have the following result.

\begin{lemma}
\label{lemma-wsf4}

\begin{itemize}
\item[(A)] Let $x,y\in E$, $l\geq 1$, $M>0$, and $P\in \Gamma _{l}(x,M)$.
Then there exists $P^{\prime }\in \Gamma _{l-1}(y,M)$ such that 
\begin{equation*}
\left\vert \partial ^{\beta }\left( P-P^{\prime }\right) \left( x\right)
\right\vert \leq M\left\vert x-y\right\vert ^{m-\left\vert \beta \right\vert
}\text{ for }\left\vert \beta \right\vert \leq m-1\text{.}
\end{equation*}

\item[(B)] If $\vec{\Gamma}_{0}$ is $\left( C_{w},\delta _{\max }\right) $%
-convex, then for each $l\geq 0$, $\vec{\Gamma}_{l}$ is $\left( C_{l},\delta
_{\max }\right) $-convex, where $C_{l}$ is determined by $C_{w}$, $l$, $m$, $%
n$.
\end{itemize}
\end{lemma}

We call $\vec{\Gamma}_{l}$ the $l$-th refinement of $\vec{\Gamma}_{0}$.
(This is consistent with our previous definition of the first refinement.)

\section{Polynomial Bases}

\label{polynomial-bases}

Let $\vec{\Gamma}=\left( \Gamma \left( x,M\right) \right) _{x\in E,M>0}$ be
a shape field. Let $x_{0}\in E$, $M_{0}>0$, $P^{0}\in \mathcal{P}$, $%
\mathcal{A}\subseteq \mathcal{M}$, $P_{\alpha }\in \mathcal{P}$ for $\alpha
\in \mathcal{A}$, $C_{B}>0$, $\delta >0$ be given. Then we say that $\left(
P_{\alpha }\right) _{\alpha \in \mathcal{A}}$ is \underline{an $\left( 
\mathcal{A},\delta ,C_{B}\right) $-basis for $\vec{\Gamma}$ at $\left(
x_{0},M_{0},P^{0}\right) $} if the following conditions are satisfied:

\begin{itemize}
\item[\refstepcounter{equation}\text{(\theequation)}\label{pb1}] $P^{0}\in
\Gamma \left( x_{0},C_{B}M_{0}\right) $.
\end{itemize}

\begin{itemize}
\item[\refstepcounter{equation}\text{(\theequation)}\label{pb2}] $P^{0}+%
\frac{M_{0}\delta ^{m-\left\vert \alpha \right\vert }}{C_{B}}P_{\alpha }$, $%
P^{0}-\frac{M_{0}\delta ^{m-\left\vert \alpha \right\vert }}{C_{B}}P_{\alpha
}\in \Gamma \left( x_{0},C_{B}M_{0}\right) $ for all $\alpha \in \mathcal{A}$%
.
\end{itemize}

\begin{itemize}
\item[\refstepcounter{equation}\text{(\theequation)}\label{pb3}] $\partial
^{\beta }P_{\alpha }\left( x_{0}\right) =\delta _{\alpha \beta }$ (Kronecker
delta) for $\beta ,\alpha \in \mathcal{A}$.
\end{itemize}

\begin{itemize}
\item[\refstepcounter{equation}\text{(\theequation)}\label{pb4}] $\left\vert
\partial ^{\beta }P_{\alpha }\left( x_{0}\right) \right\vert \leq
C_{B}\delta ^{\left\vert \alpha \right\vert -\left\vert \beta \right\vert }$
for all $\alpha \in \mathcal{A}$, $\beta \in \mathcal{M}$.
\end{itemize}

We say that $(P_{\alpha })_{\alpha \in \mathcal{A}}$ is a \underline{weak 
$(\mathcal{A},\delta ,C_{B})$-basis for $\vec{\Gamma}$ at $%
(x_{0},M_{0},P^{0})$} if conditions \eqref{pb1}, \eqref{pb2}, \eqref{pb3}
hold as stated, and condition \eqref{pb4} holds for $\alpha \in \mathcal{A}%
,\beta \in \mathcal{M},\beta \geq \alpha $.

We make a few obvious remarks.

\begin{itemize}
\item[\refstepcounter{equation}\text{(\theequation)}\label{pb5}] Any $(%
\mathcal{A}, \delta, C_B)$-basis for $\vec{\Gamma}$ at $(x_0,M_0,P^0)$ is
also an $(\mathcal{A}, \delta, C_B^{\prime })$-basis for $\vec{\Gamma}$ at $%
(x_0,M_0,P^0)$, whenever $C^{\prime }_B \geq C_B$.
\end{itemize}

\begin{itemize}
\item[\refstepcounter{equation}\text{(\theequation)}\label{pb6}] Any $(%
\mathcal{A}, \delta, C_B)$-basis for $\vec{\Gamma}$ at $(x_0,M_0,P^0)$ is
also an $(\mathcal{A}, \delta^{\prime }, C_B\cdot[\max\{\frac{\delta^{\prime
}}{\delta},\frac{\delta}{\delta^{\prime }} \}]^m)$-basis for $\vec{\Gamma}$
at $(x_0,M_0,P^0)$, for any $\delta^{\prime }>0$.
\end{itemize}

\begin{itemize}
\item[\refstepcounter{equation}\text{(\theequation)}\label{pb7}] {Any weak $(%
\mathcal{A}, \delta, C_B)$-basis for $\vec{\Gamma}$ at $(x_0,M_0,P^0)$ is
also a weak $(\mathcal{A}, \delta^{\prime }, C_B^{\prime })$-basis for $\vec{%
\Gamma}$ at $(x_0,M_0,P^0)$, whenever $0<\delta^{\prime }\leq \delta$ and $%
C^{\prime }_B \geq C_B$. }
\end{itemize}

Note that \eqref{pb1} need not follow from \eqref{pb2}, since $\mathcal{A}$
may be empty.

\begin{itemize}
\item[\refstepcounter{equation}\text{(\theequation)}\label{pb7a}] If $%
\mathcal{A}=\emptyset$, then the existence of an $(\mathcal{A},\delta,C_B)$%
-basis (or a weak $(\mathcal{A},\delta,C_B)$%
-basis) for $\vec{\Gamma}$ at $(x_0,M_0,P^0)$ is equivalent to the assertion
that $P^0 \in \Gamma(x_0, C_BM_0)$.
\end{itemize}

The main result of this section is Lemma \ref{lemma-pb2} below. The proof of Lemma \ref{lemma-pb2} relies on two other lemmas. 

As a consequence of Lemma \ref{lemma-wsf1}, we have the following result.

\begin{lemma}
\label{lemma-pb1}Let $\vec{\Gamma}=\left( \Gamma \left( x,M\right) \right)
_{x\in E,M>0}$be a $\left( C_{w},\delta _{\max }\right) $-convex shape
field. Fix $x_{0}\in E$, $M_{0}>0$, $0<\delta \leq \delta _{\max }$, $C_{1}>0
$, and let $P^{0}$, $\hat{P}$, $\hat{S}\in \mathcal{P}$.

Assume that

\begin{itemize}
\item[\refstepcounter{equation}\text{(\theequation)}\label{pb8}] $P^{0}+%
\frac{1}{C_{1}}\hat{P}$, $P^{0}-\frac{1}{C_{1}}\hat{P}\in \Gamma \left(
x_{0},C_{1}M_0\right) $;
\end{itemize}

\begin{itemize}
\item[\refstepcounter{equation}\text{(\theequation)}\label{pb9}] $\left\vert
\partial ^{\beta }\hat{P}\left( x_{0}\right) \right\vert \leq
C_{1}M_{0}\delta ^{m-\left\vert \beta \right\vert }$ for $\left\vert \beta
\right\vert \leq m-1$; and
\end{itemize}

\begin{itemize}
\item[\refstepcounter{equation}\text{(\theequation)}\label{pb10}] $%
\left\vert \partial ^{\beta }\hat{S}\left( x_{0}\right) \right\vert \leq
C_{1}\delta ^{-\left\vert \beta \right\vert }$ for $\left\vert \beta
\right\vert \leq m-1$.
\end{itemize}

Then

\begin{itemize}
\item[\refstepcounter{equation}\text{(\theequation)}\label{pb11}] $P^{0}+%
\frac{1}{C_2}\hat{S}\odot _{x_{0}}\hat{P}$, $P^{0}-\frac{1}{C_2}\hat{S}\odot
_{x_{0}}\hat{P}\in \Gamma \left( x_{0},C_{2}M_0\right) $, with $C_{2}$
determined by $C_{1}$, $C_{w}$, $m$, $n$.
\end{itemize}
\end{lemma}

We also need the following result, which is immediate\footnote{%
Lemma 16.1 in \cite{f-2005} involves also real numbers $F_{\alpha ,\beta }$
with $\left\vert \beta \right\vert =m$. Setting those $F_{\alpha ,\beta }=0$%
, we recover the Rescaling Lemma stated here.} from Lemma 16.1 in \cite%
{f-2005}.

\begin{lemma}[Rescaling Lemma]
\label{rescaling-lemma} Let $\mathcal{A}\subseteq \mathcal{M}$, and let $C,a$
be positive real numbers. Suppose we are given real numbers $F_{\alpha
,\beta }$, indexed by $\alpha \in \mathcal{A}$, $\beta \in \mathcal{M}$.
Assume that the following conditions are satisfied.

\begin{itemize}
\item[\refstepcounter{equation}\text{(\theequation)}\label{pb18}] $%
F_{\alpha,\alpha}\not=0$ for all $\alpha \in \mathcal{A}$.
\end{itemize}

\begin{itemize}
\item[\refstepcounter{equation}\text{(\theequation)}\label{pb19}] $%
|F_{\alpha,\beta}|\leq C \cdot |F_{\alpha,\alpha}|$ for all $\alpha \in 
\mathcal{A}$, $\beta\in \mathcal{M}$ with $\beta \geq \alpha$.
\end{itemize}

\begin{itemize}
\item[\refstepcounter{equation}\text{(\theequation)}\label{pb20}] $%
F_{\alpha,\beta}=0$ for all $\alpha,\beta\in \mathcal{A}$ with $\alpha
\not=\beta$.
\end{itemize}

Then there exist positive numbers $\lambda_1, \cdots, \lambda_n$ and a map $%
\phi: \mathcal{A} \rightarrow \mathcal{M}$, with the following properties:

\begin{itemize}
\item[\refstepcounter{equation}\text{(\theequation)}\label{pb21}] $c(a) \leq
\lambda_i \leq 1$ for each $i$, where $c(a)$ is determined by $C$, $a$, $m$, 
$n$;
\end{itemize}

\begin{itemize}
\item[\refstepcounter{equation}\text{(\theequation)}\label{pb22}] $%
\phi(\alpha) \leq \alpha$ for each $\alpha \in \mathcal{A}$;
\end{itemize}

\begin{itemize}
\item[\refstepcounter{equation}\text{(\theequation)}\label{pb23}] For each $%
\alpha \in \mathcal{A}$, either $\phi(\alpha) = \alpha$ or $\phi(\alpha)
\not\in \mathcal{A}$.
\end{itemize}

Suppose we define $\hat{F}_{\alpha,\beta}=\lambda^\beta F_{\alpha,\beta}$
for $\alpha \in \mathcal{A}$, $\beta \in \mathcal{M}$, where we recall that $%
\lambda^\beta$ denotes $\lambda^{\beta_1}_1\cdots \lambda^{\beta_n}_n$ for $%
\beta=(\beta_1,\cdots, \beta_n)$. Then

\begin{itemize}
\item[\refstepcounter{equation}\text{(\theequation)}\label{pb24}] $|\hat{F}%
_{\alpha,\beta}|\leq a \cdot |\hat{F}_{\alpha,\phi(\alpha)}|$ for $\alpha
\in \mathcal{A}$, $\beta \in \mathcal{M}\setminus \{ \phi(\alpha)\}$.
\end{itemize}
\end{lemma}

Lemma 3.3 in \cite{f-2005} tells us that any map $\phi: \mathcal{A}
\rightarrow \mathcal{M}$ satisfying \eqref{pb22}, \eqref{pb23} satisfies also

\begin{itemize}
\item[\refstepcounter{equation}\text{(\theequation)}\label{pb25}] $\phi(%
\mathcal{A}) \leq \mathcal{A}$, with equality only if $\phi$ is the identity
map.
\end{itemize}

We are ready to state the main result of this section.

\begin{lemma}[Relabeling Lemma]
\label{lemma-pb2} Let $\vec{\Gamma}=\left( \Gamma \left( x,M\right) \right)
_{x\in E,M>0}$ be a $(C_w,\delta_{\max})$-convex shape field. Let $x_{0}\in E
$, $M_{0}>0$, $0<\delta \leq \delta _{\max }$, $C_{B}>0$, $P^{0}\in \Gamma
\left( x_{0},M_{0}\right) $, $\mathcal{A}\subseteq \mathcal{M}$. Suppose $%
\left( P_{\alpha }^{00}\right) _{\alpha \in \mathcal{A}}$ is a weak $\left( 
\mathcal{A},\delta ,C_{B}\right) $-basis for $\vec{\Gamma}$ at $\left(
x_{0},M_{0},P^{0}\right) $. Then, for some monotonic $\hat{\mathcal{A}}\leq 
\mathcal{A}$, $\vec{\Gamma}$ has an $(\hat{\mathcal{A}},\delta
,C_{B}^{\prime })$-basis at $(x_{0},M_{0},P^{0})$, with $C_{B}^{\prime }$
determined by $C_{B}$, $C_{w}$, $m$, $n$. Moreover, if $\mathcal{A}\not=\emptyset$ and $\max_{\alpha \in 
\mathcal{A},\beta \in \mathcal{M}}\delta ^{|\beta |-|\alpha |}|\partial
^{\beta }P_{\alpha }^{00}(x_{0})|$ exceeds a large enough constant
determined by $C_{B}$, $C_{w}$, $m$, $n$, then we can take $\hat{\mathcal{A}}%
<\mathcal{A}$ (strict inequality).
\end{lemma}

\begin{proof} If $\A$ is empty, then we can take $\hat{\A}$ empty. Thus, Lemma \ref{lemma-pb2} holds trivially for $\A = \emptyset$. We suppose that $\A \not= \emptyset$.

Without loss of generality, we may take $x_0 =0$. We introduce a constant $%
a>0$ to be picked later, satisfying the

\emph{Small $a$ condition: $a$ is less than a small enough constant
determined by $C_B$, $C_w$, $m$, $n$.}

We write $c$, $C$, $C^{\prime }$, etc., to denote constants determined by $%
C_B$, $C_w$, $m$, $n$; and we write $c(a)$, $C(a)$, $C^{\prime }(a)$, etc., to
denote constants determined by $a$, $C_B$, $C_w$, $m$, $n$. These symbols
may denote different constants in different occurrences.

Since $(P_\alpha^{00})_{\alpha \in \mathcal{A}}$ is a weak $(\mathcal{A},
\delta, C_B)$-basis for $\vec{\Gamma}$ at $(0, M_0, P^0)$, we have the
following.

\begin{itemize}
\item[\refstepcounter{equation}\text{(\theequation)}\label{pb26}] $%
P^{0},P^{0}\pm cM_{0}\delta ^{m-\left\vert \alpha \right\vert }P_{\alpha
}^{00}\in \Gamma \left( 0,CM_{0}\right) $ for $\alpha \in \mathcal{A}$.
\end{itemize}

\begin{itemize}
\item[\refstepcounter{equation}\text{(\theequation)}\label{pb27}] $\partial
^{\beta }P_{\alpha }^{00}\left( 0\right) =\delta _{\beta \alpha }$ for $%
\beta ,\alpha \in \mathcal{A}$.
\end{itemize}

\begin{itemize}
\item[\refstepcounter{equation}\text{(\theequation)}\label{pb28}] $%
\left\vert \partial ^{\beta }P_{\alpha }^{00}\left( 0\right) \right\vert
\leq C\delta ^{\left\vert \alpha \right\vert -\left\vert \beta \right\vert }$
for $\alpha \in \mathcal{A}$, $\beta \in \mathcal{M}$, $\beta \geq \alpha $.
\end{itemize}

Thanks to \eqref{pb27}, \eqref{pb28}, the numbers $F_{\alpha,\beta} =
\delta^{|\beta|-|\alpha|}\partial^\beta P_\alpha^{00}(0)$ satisfy %
\eqref{pb18}, \eqref{pb19}, \eqref{pb20}. Applying Lemma \ref%
{rescaling-lemma}, we obtain real numbers $\lambda_1, \cdots, \lambda_n$ and
a map $\phi: \mathcal{A} \rightarrow \mathcal{M}$ satisfying \eqref{pb21}, $%
\cdots$, \eqref{pb24}. We define a linear map $T: \mathbb{R}^n \rightarrow 
\mathbb{R}^n$ by setting

\begin{itemize}
\item[\refstepcounter{equation}\text{(\theequation)}\label{pb29}] $T\left(
x_{1},\cdots ,x_{n}\right) =\left( \lambda _{1}x_{1},\cdots ,\lambda
_{n}x_{n}\right) $ for $\left( x_{1},\cdots ,x_{n}\right) \in \mathbb{R}^{n}$%
.
\end{itemize}

From \eqref{pb21} $\cdots $\eqref{pb24} for our $F_{\alpha ,\beta }$, we
obtain the following.

\begin{itemize}
\item[\refstepcounter{equation}\text{(\theequation)}\label{pb30}] $c\left(
a\right) \leq \lambda _{i}\leq 1$ for $i=1,\cdots ,n$.
\end{itemize}

\begin{itemize}
\item[\refstepcounter{equation}\text{(\theequation)}\label{pb31}] $\phi
\left( \alpha \right) \leq \alpha $ for all $\alpha \in \mathcal{A}$.
\end{itemize}

\begin{itemize}
\item[\refstepcounter{equation}\text{(\theequation)}\label{pb32}] For each $%
\alpha \in \mathcal{A}$, either $\phi \left( \alpha \right) =\alpha $ or $%
\phi \left( \alpha \right) \not\in \mathcal{A}$.
\end{itemize}

\begin{itemize}
\item[\refstepcounter{equation}\text{(\theequation)}\label{pb33}] $\delta
^{\left\vert \beta \right\vert -\left\vert \alpha \right\vert }\left\vert
\partial ^{\beta }\left( P_{\alpha }^{00}\circ T\right) \left( 0\right)
\right\vert \leq a\cdot \delta ^{\left\vert \phi \left( \alpha \right)
\right\vert -\left\vert \alpha \right\vert }\left\vert \partial ^{\phi
\left( \alpha \right) }\left( P_{\alpha }^{00}\circ T\right) \left( 0\right)
\right\vert $ for all $\alpha \in \mathcal{A}$, $\beta \in \mathcal{%
M\setminus }\phi \left( \alpha \right) $.
\end{itemize}

Since the left-hand side of \eqref{pb33} is equal to $\lambda^\beta \geq
c(a) $ for $\beta=\alpha$ (by \eqref{pb27} and \eqref{pb30}), it follows
from \eqref{pb33} that

\begin{itemize}
\item[\refstepcounter{equation}\text{(\theequation)}\label{pb34}] $\delta
^{\left\vert \phi \left( \alpha \right) \right\vert -\left\vert \alpha
\right\vert }\left\vert \partial ^{\phi \left( \alpha \right) }\left(
P_{\alpha }^{00}\circ T\right) \left( 0\right) \right\vert \geq c\left(
a\right) $ for $\alpha \in \mathcal{A}$.
\end{itemize}

We define

\begin{itemize}
\item[\refstepcounter{equation}\text{(\theequation)}\label{pb35}] $\bar{%
\mathcal{A}}=\phi(\mathcal{A})$
\end{itemize}
and introduce a map
\begin{itemize}
\item[\refstepcounter{equation}\text{(\theequation)}\label{pb36}] $\psi: 
\bar{\mathcal{A}} \rightarrow \mathcal{A}$
\end{itemize}
such that
\begin{itemize}
\item[\refstepcounter{equation}\text{(\theequation)}\label{pb37}] $\phi
(\psi (\bar{\alpha}))=\bar{\alpha}$ for all $\bar{\alpha}\in \bar{\mathcal{A}%
}$.
\end{itemize}

Thanks to \eqref{pb31}, \eqref{pb32}, and Lemma 3.3 in \cite{f-2005}
(mentioned above), we have 
\begin{equation*}
\bar{\mathcal{A}}\leq \mathcal{A},
\end{equation*}%
with equality only when $\phi =$identity. Moreover, suppose $\phi =$
identity. Then \eqref{pb27} and \eqref{pb33} show that%
\begin{eqnarray*}
\left\vert \lambda ^{\beta }\delta ^{\left\vert \beta \right\vert
-\left\vert \alpha \right\vert }\partial ^{\beta }P_{\alpha }^{00}\left(
0\right) \right\vert  &=&\delta ^{\left\vert \beta \right\vert -\left\vert
\alpha \right\vert }\left\vert \partial ^{\beta }\left( P_{\alpha
}^{00}\circ T\right) \left( 0\right) \right\vert  \\
&\leq &\left\vert \partial ^{\alpha }\left( P_{\alpha }^{00}\circ T\right)
\left( 0\right) \right\vert  \\
&=&\lambda ^{\alpha }
\end{eqnarray*}%
for $\alpha \in \mathcal{A}$ and $\beta \in \mathcal{M}$. Hence, \eqref{pb30}
yields

\begin{itemize}
\item[\refstepcounter{equation}\text{(\theequation)}\label{pb38}] $\delta
^{\left\vert \beta \right\vert -\left\vert \alpha \right\vert }\left\vert
\partial ^{\beta }P_{\alpha }^{00}\left( 0\right) \right\vert \leq C\left(
a\right) $ for all $\alpha \in \mathcal{A}$ and $\beta \in \mathcal{M}$;
\end{itemize}
this holds provided $\phi =$identity.

If $\max_{\alpha \in \mathcal{A},\beta \in \mathcal{M}}\delta ^{\left\vert
\beta \right\vert -\left\vert \alpha \right\vert }\left\vert \partial
^{\beta }P_{\alpha }^{00}\left( 0\right) \right\vert >C(a) $ with $C\left(
a\right) $ as in \eqref{pb38}, then $\phi $ cannot be the identity map, and
therefore $\bar{\mathcal{A}} < \mathcal{A} $ (strict inequality).

Thus,

\begin{itemize}
\item[\refstepcounter{equation}\text{(\theequation)}\label{pb39}] $\bar{%
\mathcal{A}} \leq \mathcal{A}$, with strict inequality if $\max_{\alpha \in 
\mathcal{A},\beta \in \mathcal{M}}\delta ^{\left\vert \beta \right\vert
-\left\vert \alpha \right\vert }\left\vert \partial ^{\beta }P_{\alpha
}^{00}\left( 0\right) \right\vert $ exceeds a large enough constant
determined by $a$, $C_{w}$, $C_{B}$, $m$, $n$.
\end{itemize}

For $\bar{\alpha} \in \bar{\mathcal{A}}$, we define

\begin{itemize}
\item[\refstepcounter{equation}\text{(\theequation)}\label{pb40}] $b_{\bar{%
\alpha}}=\left[ \delta ^{\left\vert \bar{\alpha}\right\vert -\left\vert \psi
\left( \bar{\alpha}\right) \right\vert }\partial ^{\bar{\alpha}}\left(
P_{\psi \left( \bar{\alpha}\right) }^{00}\circ T\right) \left( 0\right) %
\right] ^{-1}$;
\end{itemize}

estimate \eqref{pb34} with $\alpha =\psi \left( \bar{\alpha}\right) $ gives

\begin{itemize}
\item[\refstepcounter{equation}\text{(\theequation)}\label{pb41}] $%
\left\vert b_{\bar{\alpha}}\right\vert \leq C\left( a\right) $ for all $\bar{%
\alpha}\in \bar{\mathcal{A}}$.
\end{itemize}

For $\bar{\alpha}\in \bar{\mathcal{A}}$, we also define

\begin{itemize}
\item[\refstepcounter{equation}\text{(\theequation)}\label{pb42}] $\bar{P}_{%
\bar{\alpha}}=b_{\bar{\alpha}}\delta ^{\left\vert \bar{\alpha}\right\vert
-\left\vert \psi \left( \bar{\alpha}\right) \right\vert }\cdot P_{\psi
\left( \bar{\alpha}\right) }^{00}$.
\end{itemize}

From \eqref{pb33}, \eqref{pb40}, \eqref{pb42}, we find that

\begin{itemize}
\item[\refstepcounter{equation}\text{(\theequation)}\label{pb43}] $%
\left\vert \delta ^{\left\vert \beta \right\vert -\left\vert \bar{\alpha}%
\right\vert }\partial ^{\beta }\left( \bar{P}_{\bar{\alpha}}\circ T\right)
\left( 0\right) -\delta _{\beta \bar{\alpha }}\right\vert \leq a$ for $\bar{\alpha}%
\in \bar{\mathcal{A}}$, $\beta \in \mathcal{M}$.
\end{itemize}

Note that $\delta ^{m-\left\vert \bar{\alpha}\right\vert }\bar{P}_{\bar{%
\alpha}}=b_{\bar{\alpha}}\delta ^{m-\left\vert \psi \left( \bar{\alpha}%
\right) \right\vert }P_{\psi \left( \bar{\alpha}\right) }^{00}$. Hence, %
\eqref{pb41} and \eqref{pb26} (with $\alpha =\psi \left( \bar{\alpha}\right) 
$) tell us that

\begin{itemize}
\item[\refstepcounter{equation}\text{(\theequation)}\label{pb44}] $P^{0}\pm
c\left( a\right) M_{0}\delta ^{m-\left\vert \bar{\alpha}\right\vert }\bar{P}%
_{\bar{\alpha}}\in \Gamma \left( 0,CM_{0}\right) $,
\end{itemize}

since $\Gamma(0,CM_0)$ is convex.

Next, define

\begin{itemize}
\item[\refstepcounter{equation}\text{(\theequation)}\label{pb45}] $\hat{%
\mathcal{A}}=\left\{ \gamma \in \mathcal{M}:\gamma =\bar{\alpha}+\bar{\gamma}%
\text{ for some }\bar{\alpha}\in \bar{\mathcal{A}}\text{, }\bar{\gamma}\in 
\mathcal{M}\right\} $,
\end{itemize}

and introduce maps $\chi :\hat{\mathcal{A}}\rightarrow \bar{\mathcal{A}}$, $%
\omega :\hat{\mathcal{A}}\rightarrow \mathcal{M}$, such that

\begin{itemize}
\item[\refstepcounter{equation}\text{(\theequation)}\label{pb46}] $\hat{%
\alpha}=\chi \left( \hat{\alpha}\right) +\omega \left( \hat{\alpha}\right) $
for all $\hat{\alpha}\in \hat{\mathcal{A}}$.
\end{itemize}

By definition \eqref{pb45},

\begin{itemize}
\item[\refstepcounter{equation}\text{(\theequation)}\label{pb47}] $\hat{%
\mathcal{A}}$ is monotonic,
\end{itemize}

and $\bar{\mathcal{A}} \subseteq \hat{\mathcal{A}}$, hence

\begin{itemize}
\item[\refstepcounter{equation}\text{(\theequation)}\label{pb48}] $\hat{%
\mathcal{A}}\leq \bar{\mathcal{A}}$.
\end{itemize}

From \eqref{pb39} and \eqref{pb48}, we have

\begin{itemize}
\item[\refstepcounter{equation}\text{(\theequation)}\label{pb49}] $\hat{%
\mathcal{A}} \leq \mathcal{A}$, with strict inequality if $\max_{\alpha \in 
\mathcal{A},\beta \in \mathcal{M}}\delta ^{\left\vert \beta \right\vert
-\left\vert \alpha \right\vert }\left\vert \partial ^{\beta }P_{\alpha
}^{00}\left( 0\right) \right\vert $ exceeds a large enough constant
determined by $a$, $C_{w}$, $C_{B}$, $m$, $n$.
\end{itemize}

For $\hat{\alpha} \in \hat{\mathcal{A}}$, we introduce the monomial

\begin{itemize}
\item[\refstepcounter{equation}\text{(\theequation)}\label{pb50}] $S_{\hat{%
\alpha}}\left( x\right) =\frac{\chi \left( \hat{\alpha}\right) !}{\hat{\alpha%
}!}\lambda ^{-\omega \left( \hat{\alpha}\right) }x^{\omega \left( \hat{\alpha%
}\right) }$ $\left( x\in \mathbb{R}^{n}\right) $.
\end{itemize}

We have

\begin{itemize}
\item[\refstepcounter{equation}\text{(\theequation)}\label{pb51}] $S_{\hat{%
\alpha}}\circ T\left( x\right) =\frac{\chi \left( \hat{\alpha}\right) !}{%
\hat{\alpha}!}x^{\omega \left( \hat{\alpha}\right) }$,
\end{itemize}
hence
\begin{itemize}
\item[\refstepcounter{equation}\text{(\theequation)}\label{pb52}] $\partial
^{\beta }\left( S_{\hat{\alpha}}\circ T\right) \left( 0\right) =\frac{\chi
\left( \hat{\alpha}\right) !\omega \left( \hat{\alpha}\right) !}{\hat{\alpha}%
!}\delta _{\beta \omega \left( \hat{\alpha}\right) }$ for $\hat{\alpha}\in 
\hat{\mathcal{A}},\beta \in \mathcal{M}$.
\end{itemize}

We study the derivatives $\partial ^{\beta }\left( \left[ \left( S_{\hat{%
\alpha}}\bar{P}_{\chi \left( \hat{\alpha}\right) }\right) \circ T\right]
\left( 0\right) \right) $ for $\hat{\alpha}\in \hat{\mathcal{A}},\beta \in 
\mathcal{M}$.

Case 1: If $\beta$ is not of the form $\beta = \omega(\hat{\alpha})+\tilde{%
\beta}$ for some $\tilde{\beta}\in \mathcal{M}$, then \eqref{pb52} gives

\begin{itemize}
\item[\refstepcounter{equation}\text{(\theequation)}\label{pb53}] $\partial
^{\beta }\left[ \left( S_{\hat{\alpha}}\bar{P}_{\chi \left( \hat{\alpha}%
\right) }\right) \circ T\right] \left( 0\right) =0$.
\end{itemize}

Case 2: Suppose $\beta =\omega (\hat{\alpha})+\tilde{\beta}$ for some $%
\tilde{\beta}\in \mathcal{M}$. Then \eqref{pb52} gives%
\begin{eqnarray*}
&&\partial ^{\beta }\left[ \left( S_{\hat{\alpha}}\bar{P}_{\chi \left( \hat{%
\alpha}\right) }\right) \circ T\right] \left( 0\right)  \\
&=&\frac{\beta !}{\omega \left( \hat{\alpha}\right) !\tilde{\beta}!}\left[
\partial ^{\omega \left( \hat{\alpha}\right) }\left( S_{\hat{\alpha}}\circ
T\right) \left( 0\right) \right] \cdot \left[ \partial ^{\tilde{\beta}}\left( 
\bar{P}_{\chi \left( \hat{\alpha}\right) }\circ T\right) \left( 0\right) %
\right]  \\
&=&\frac{\beta !}{\omega \left( \hat{\alpha}\right) !\tilde{\beta}!}\frac{%
\chi \left( \hat{\alpha}\right) !\omega \left( \hat{\alpha}\right) !}{\hat{%
\alpha}!}\cdot \left[ \partial ^{\tilde{\beta}}\left( \bar{P}_{\chi \left( 
\hat{\alpha}\right) }\circ T\right) \left( 0\right) \right]  \\
&=&\frac{\beta !\chi \left( \hat{\alpha}\right) !}{\hat{\alpha}!\tilde{\beta}%
!}\left[ \partial ^{\tilde{\beta}}\left( \bar{P}_{\chi \left( \hat{\alpha}%
\right) }\circ T\right) \left( 0\right) \right] \text{.}
\end{eqnarray*}%
Hence, by \eqref{pb43}, we have 
\begin{equation}
\left\vert \frac{\hat{\alpha}!\tilde{\beta}!}{\beta !\chi \left( \hat{\alpha}%
\right) !}\delta ^{\left\vert \tilde{\beta}\right\vert -\left\vert \chi
\left( \hat{\alpha}\right) \right\vert }\partial ^{\beta }\left[ \left( S_{%
\hat{\alpha}}\cdot \bar{P}_{\chi \left( \hat{\alpha}\right) }\right) \circ T%
\right] \left( 0\right) -\delta _{\tilde{\beta}\chi \left( \hat{\alpha}%
\right) }\right\vert \leq a\text{.}  \label{pb54}
\end{equation}%
Since in this case $\beta =\omega \left( \hat{\alpha}\right) +\tilde{\beta}$
and $\hat{\alpha}=\omega \left( \hat{\alpha}\right) +\chi \left( \hat{\alpha}%
\right) $ (see \eqref{pb46}), we have $\delta _{\tilde{\beta}\chi \left( 
\hat{\alpha}\right) }=\delta _{\beta \hat{\alpha}}$, $|\tilde{\beta}%
|-\left\vert \chi \left( \hat{\alpha}\right) \right\vert =\left\vert \beta
\right\vert -\left\vert \hat{\alpha}\right\vert $, and $\frac{\hat{\alpha}!%
\tilde{\beta}!}{\beta !\chi \left( \hat{\alpha}\right) !}=1$ if $\beta =\hat{%
\alpha}$.

Hence, \eqref{pb54} implies that 
\begin{equation}
\left\vert \delta ^{\left\vert \beta \right\vert -\left\vert \hat{\alpha}%
\right\vert }\partial ^{\beta }\left[ \left( S_{\hat{\alpha}}\cdot \bar{P}%
_{\chi \left( \hat{\alpha}\right) }\right) \circ T\right] \left( 0\right)
-\delta _{\beta \hat{\alpha}}\right\vert \leq Ca  \label{pb55}
\end{equation}%
in Case 2.

Thanks to \eqref{pb46} and \eqref{pb53}, estimate \eqref{pb55} holds also in
Case 1.

Thus, \eqref{pb55} holds for all $\hat{\alpha}\in \hat{\mathcal{A}}$, $\beta
\in \mathcal{M}$. Consequently, 
\begin{equation}
\left\vert \delta ^{\left\vert \beta \right\vert -\left\vert \hat{\alpha}%
\right\vert }\partial ^{\beta }\left( \left[ S_{\hat{\alpha}}\odot _{0}\bar{P%
}_{\chi \left( \hat{\alpha}\right) }\right] \circ T \right)\left( 0\right)
-\delta _{\beta \hat{\alpha}}\right\vert \leq Ca  \label{pb56}
\end{equation}%
for all $\hat{\alpha}\in \hat{\mathcal{A}}$, $\beta \in \mathcal{M}$.

We prepare to apply Lemma \ref{lemma-pb1}, with $\hat{S}=\delta
^{-\left\vert \omega \left( \hat{\alpha}\right) \right\vert }S_{\hat{\alpha}%
} $ and $\hat{P}=M_{0}\delta ^{m-\left\vert \chi \left( \hat{\alpha}\right)
\right\vert }\bar{P}_{\chi \left( \hat{\alpha}\right) }$. From \eqref{pb30}
and \eqref{pb50}, we have 
\begin{equation}
\left\vert \partial ^{\beta }\left( \delta ^{-\left\vert \omega \left( \hat{%
\alpha}\right) \right\vert }S_{\hat{\alpha}}\right) \left( 0\right)
\right\vert \leq C\left( a\right) \delta ^{-\left\vert \beta \right\vert }%
\text{ for }\hat{\alpha}\in \hat{\mathcal{A}}\text{, }\beta \in \mathcal{M}.
\label{pb57}
\end{equation}

From \eqref{pb43} we have 
\begin{equation*}
\left\vert \partial ^{\beta }\left( \bar{P}_{\chi \left( \hat{\alpha}\right)
}\circ T\right) \left( 0\right) \right\vert \leq C\delta ^{\left\vert \chi
\left( \hat{\alpha}\right) \right\vert -\left\vert \beta \right\vert }\text{
for }\hat{\alpha}\in \hat{\mathcal{A}}\text{, }\beta \in \mathcal{M}\text{;}
\end{equation*}%
hence, by \eqref{pb30}, 
\begin{equation}
\left\vert \partial ^{\beta }\left( M_{0}\delta ^{m-\left\vert \chi \left( 
\hat{\alpha}\right) \right\vert }\bar{P}_{\chi \left( \hat{\alpha}\right)
}\right) \left( 0\right) \right\vert \leq C\left( a\right) M_{0}\delta
^{m-\left\vert \beta \right\vert }\text{ for }\hat{\alpha}\in \hat{\mathcal{A%
}}\text{, }\beta \in \mathcal{M}\text{.}  \label{pb58}
\end{equation}%
Also, \eqref{pb44} gives 
\begin{equation}
P^{0}\pm c\left( a\right) M_{0}\delta ^{m-\left\vert \chi \left( \hat{\alpha}%
\right) \right\vert }\bar{P}_{\chi \left( \hat{\alpha}\right) }\in \Gamma
\left( 0,CM_{0}\right) \text{ for }\hat{\alpha}\in \hat{\mathcal{A}}\text{.}
\label{pb59}
\end{equation}

Our results \eqref{pb57}, \eqref{pb58}, \eqref{pb59} are the hypotheses of
Lemma \ref{lemma-pb1} for $\hat{S}$, $\hat{P}$ as given above. Applying that
lemma, we learn that 
\begin{equation*}
P^{0}\pm c\left( a\right) \left( \delta ^{-\left\vert \omega \left( \hat{%
\alpha}\right) \right\vert }S_{\hat{\alpha}}\right) \odot _{0}\left(
M_{0}\delta ^{m-\left\vert \chi \left( \hat{\alpha}\right) \right\vert }\bar{%
P}_{\chi \left( \hat{\alpha}\right) }\right) \in \Gamma \left( 0,C\left(
a\right) M_{0}\right) \text{ for }\hat{\alpha}\in \hat{\mathcal{A}}\text{.}
\end{equation*}%
Recalling \eqref{pb46}, we conclude that 
\begin{equation}
P^{0}\pm c\left( a\right) M_{0}\delta ^{m-\left\vert \hat{\alpha}\right\vert
}S_{\hat{\alpha}}\odot _{0}\bar{P}_{\chi \left( \hat{\alpha}\right) }\in
\Gamma \left( 0,C\left( a\right) M_{0}\right) \text{ for }\hat{\alpha}\in 
\hat{\mathcal{A}}\text{.}  \label{pb60}
\end{equation}

Next, \eqref{pb56} and the \emph{small }$a$ \emph{condition} tell us that
there exists a matrix of real numbers $\left( b_{\gamma \hat{\alpha}}\right)
_{\gamma ,\hat{\alpha}\in \hat{\mathcal{A}}}$ satisfying 
\begin{equation}
\sum_{\hat{\alpha}\in \hat{\mathcal{A}}}b_{\gamma \hat{\alpha}}\partial
^{\beta }\left( \left[ S_{\hat{\alpha}}\odot _{0}\bar{P}_{\chi \left( \hat{%
\alpha}\right) }\right] \circ T\right) \left( 0\right) \cdot \delta
^{\left\vert \beta \right\vert -\left\vert \hat{\alpha}\right\vert }=\delta
_{\gamma \beta }\text{ for }\gamma ,\beta \in \hat{\mathcal{A}}  \label{pb61}
\end{equation}%
and 
\begin{equation}
\left\vert b_{\gamma \hat{\alpha}}\right\vert \leq 2\text{, for all }\gamma ,%
\hat{\alpha}\in \hat{\mathcal{A}}\text{.}  \label{pb62}
\end{equation}%
From \eqref{pb61} we have 
\begin{equation}
\partial ^{\beta }\left\{ \sum_{\hat{\alpha}\in \hat{\mathcal{A}}}b_{\gamma 
\hat{\alpha}}\delta ^{\left\vert \gamma \right\vert -\left\vert \hat{\alpha}%
\right\vert }\left( \left[ S_{\hat{\alpha}}\odot _{0}\bar{P}_{\chi \left( 
\hat{\alpha}\right) }\right] \circ T\right) \right\} \left( 0\right) =\delta
_{\gamma \beta }\text{ for }\gamma ,\beta \in \hat{\mathcal{A}}\text{.}
\label{pb63}
\end{equation}%
Also, \eqref{pb56} and \eqref{pb62} imply that 
\begin{equation}
\left\vert \partial ^{\beta }\left\{ \sum_{\hat{\alpha}\in \hat{\mathcal{A}}%
}b_{\gamma \hat{\alpha}}\delta ^{\left\vert \gamma \right\vert -\left\vert 
\hat{\alpha}\right\vert }\left( \left[ S_{\hat{\alpha}}\odot _{0}\bar{P}%
_{\chi \left( \hat{\alpha}\right) }\right] \circ T\right) \right\} \left(
0\right) \right\vert \leq C\delta ^{\left\vert \gamma \right\vert
-\left\vert \beta \right\vert }\text{ for }\gamma \in \hat{\mathcal{A}}%
,\beta \in \mathcal{M}\text{.}  \label{pb64}
\end{equation}%
Since 
\begin{eqnarray*}
&&M_{0}\delta ^{m-\left\vert \gamma \right\vert }\left\{ \sum_{\hat{\alpha}%
\in \hat{\mathcal{A}}}b_{\gamma \hat{\alpha}}\delta ^{\left\vert \gamma
\right\vert -\left\vert \hat{\alpha}\right\vert }\left( \left[ S_{\hat{\alpha%
}}\odot _{0}\bar{P}_{\chi \left( \hat{\alpha}\right) }\right] \right)
\right\}  \\
&=&\sum_{\hat{\alpha}\in \hat{\mathcal{A}}}b_{\gamma \hat{\alpha}}\cdot
\left\{ M_{0}\delta ^{m-\left\vert \hat{\alpha}\right\vert }\cdot \left[ S_{%
\hat{\alpha}}\odot _{0}\bar{P}_{\chi \left( \hat{\alpha}\right) }\right]
\right\} \text{,}
\end{eqnarray*}%
we learn from \eqref{pb60}, \eqref{pb62} and the Trivial Remark on Convex
Sets in Section \ref{notation-and-preliminaries}, that 
\begin{equation}
P^{0}\pm c\left( a\right) M_{0}\delta ^{m-\left\vert \gamma \right\vert
}\left\{ \sum_{\hat{\alpha}\in \hat{\mathcal{A}}}b_{\gamma \hat{\alpha}%
}\delta ^{\left\vert \gamma \right\vert -\left\vert \hat{\alpha}\right\vert
}\left( \left[ S_{\hat{\alpha}}\odot _{0}\bar{P}_{\chi \left( \hat{\alpha}%
\right) }\right] \right) \right\} \in \Gamma \left( 0,C\left( a\right)
M_{0}\right)   \label{pb65}
\end{equation}%
for $\gamma \in \hat{\mathcal{A}}$.

We define%
\begin{equation}
P_{\gamma }=\lambda ^{\gamma }\sum_{\hat{\alpha}\in \hat{\mathcal{A}}%
}b_{\gamma \hat{\alpha}}\delta ^{\left\vert \gamma \right\vert -\left\vert 
\hat{\alpha}\right\vert }\left[ S_{\hat{\alpha}}\odot _{0}\bar{P}_{\chi
\left( \hat{\alpha}\right) }\right] \text{ for }\gamma \in \hat{\mathcal{A}}%
\text{.}  \label{pb66}
\end{equation}

From \eqref{pb30}, we have $\left\vert \lambda ^{\gamma }\right\vert \leq
C\left( a\right) $. Recall from \eqref{pb26} that $P^{0}\in \Gamma \left(
0,CM_{0}\right) $. Hence, we deduce from \eqref{pb65} (and from the
convexity of $\Gamma \left( 0,C\left( a\right) M_{0}\right) $) that

\begin{itemize}
\item[\refstepcounter{equation}\text{(\theequation)}\label{pb67}] $%
P^{0},P^{0}+ c\left( a\right) M_{0}\delta ^{m-\left\vert \gamma
\right\vert }P_{\gamma },P^{0}-c\left( a\right) M_{0}\delta ^{m-\left\vert
\gamma \right\vert }P_{\gamma }\in \Gamma \left( 0,C\left( a\right)
M_{0}\right) $ for $\gamma \in \hat{\mathcal{A}}$.
\end{itemize}

Also, \eqref{pb63} and \eqref{pb66} give 
\begin{equation}
\partial ^{\beta }P_{\gamma }\left( 0\right) =\delta _{\beta \gamma }\text{
for }\beta ,\gamma \in \hat{\mathcal{A}}\text{.}  \label{pb68}
\end{equation}

From \eqref{pb30}, \eqref{pb64}, \eqref{pb66}, we have 
\begin{equation}
\left\vert \partial ^{\beta }P_{\gamma }\left( 0\right) \right\vert \leq
C\left( a\right) \delta ^{\left\vert \gamma \right\vert -\left\vert \beta
\right\vert }\text{ for }\gamma \in \hat{\mathcal{A}}\text{, }\beta \in 
\mathcal{M}\text{.}  \label{pb69}
\end{equation}

Our results \eqref{pb67}, \eqref{pb68}, \eqref{pb69} show that

\begin{itemize}
\item[\refstepcounter{equation}\text{(\theequation)}\label{pb70}] $\left(
P_{\gamma }\right) _{\gamma \in \hat{\mathcal{A}}}$ is an $\left( \hat{%
\mathcal{A}},\delta ,C\left( a\right) \right) $-basis for $\vec{\Gamma}$ at $%
\left( 0,M_{0},P^{0}\right) $.
\end{itemize}

We now pick $a$ to to be a constant determined by $C_{B}$, $C_{w}$, $m$, $n$%
, small enough to satisfy our \emph{small }$a$ \emph{condition}. Then %
\eqref{pb47}, \eqref{pb49}, and \eqref{pb70} immediately imply the
conclusions of Lemma \ref{lemma-pb2}.

The proof of that lemma is complete.
\end{proof}

The next result is a consequence of the Relabeling Lemma (Lemma \ref%
{lemma-pb2}).

\begin{lemma}[Control $\Gamma $ Using Basis]
\label{lemma-pb3} Let $\vec{\Gamma}=\left( \Gamma \left( x,M\right) \right)
_{x\in E,M>0}$ be a $\left( C_{w},\delta _{\max }\right) $-convex shape
field. Let $x_{0}\in E$, $M_{0}>0$, $0<\delta \leq \delta _{\max }$, $C_{B}>0
$, $\mathcal{A\subseteq M}$, and let $P$, $P^{0}\in \mathcal{P}$. Suppose $\vec{%
\Gamma}$ has an $\left( \mathcal{A},\delta ,C_{B}\right) $-basis at $\left(
x_{0},M_{0},P^{0}\right) $. Suppose also that

\begin{itemize}
\item[\refstepcounter{equation}\text{(\theequation)}\label{pb71}] $P\in
\Gamma \left( x_{0},C_{B}M_{0}\right) $,
\end{itemize}

\begin{itemize}
\item[\refstepcounter{equation}\text{(\theequation)}\label{pb72}] $\partial
^{\beta }\left( P-P^{0}\right) \left( x_{0}\right) =0$ for all $\beta \in 
\mathcal{A}$, and
\end{itemize}

\begin{itemize}
\item[\refstepcounter{equation}\text{(\theequation)}\label{pb73}] $%
\max_{\beta \in \mathcal{M}}\delta ^{\left\vert \beta \right\vert
}\left\vert \partial ^{\beta }\left( P-P^{0}\right) \left( x_{0}\right)
\right\vert \geq M_{0}\delta ^{m}$.
\end{itemize}

Then there exist $\hat{\mathcal{A}} \subseteq \mathcal{M}$ and $\hat{P}^0
\in \mathcal{P}$ with the following properties.

\begin{itemize}
\item[\refstepcounter{equation}\text{(\theequation)}\label{pb74}] $\hat{%
\mathcal{A}}$ is monotonic.
\end{itemize}

\begin{itemize}
\item[\refstepcounter{equation}\text{(\theequation)}\label{pb75}] $\hat{%
\mathcal{A}} < \mathcal{A}$ (strict inequality).
\end{itemize}

\begin{itemize}
\item[\refstepcounter{equation}\text{(\theequation)}\label{pb76}] $\vec{%
\Gamma}$ has an $(\hat{\mathcal{A}},\delta,C_B^{\prime })$-basis at $%
(x_0,M_0,\hat{P}^0)$, with $C_B^{\prime }$ determined by $C_B$, $C_w$, $m$, $%
n$.
\end{itemize}

\begin{itemize}
\item[\refstepcounter{equation}\text{(\theequation)}\label{pb77}] $\partial
^{\beta }\left( \hat{P}^{0}-P^{0}\right) \left( x_{0}\right) =0$ for all $%
\beta \in \mathcal{A}$.
\end{itemize}

\begin{itemize}
\item[\refstepcounter{equation}\text{(\theequation)}\label{pb78}] $%
\left\vert \partial ^{\beta }\left( \hat{P}^{0}-P^{0}\right) \left(
x_{0}\right) \right\vert \leq M_{0}\delta ^{m-\left\vert \beta \right\vert }$
for all $\beta \in \mathcal{M}$.
\end{itemize}
\end{lemma}

\begin{proof}
We write $c$, $C$, $C^{\prime }$, etc., to denote constants determined by $%
C_B$, $C_w$, $m$, $n$. These symbols may denote different constants in
different occurrences.

Let $(P_\alpha)_{\alpha \in \mathcal{A}}$ be an $(\mathcal{A},\delta,C_B)$%
-basis for $\vec{\Gamma}$ at $(x_0,M_0,P^0)$. By definition,

\begin{itemize}
\item[\refstepcounter{equation}\text{(\theequation)}\label{pb79}] $P^{0}\in
\Gamma (x_{0},C_{B}M_{0})$,
\end{itemize}

\begin{itemize}
\item[\refstepcounter{equation}\text{(\theequation)}\label{pb80}] $P^{0}\pm
cM_{0}\delta ^{m-\left\vert \alpha \right\vert }P_{\alpha }\in \Gamma \left(
x_{0},C_{B}M\right) $ for all $\alpha \in \mathcal{A}$,
\end{itemize}

\begin{itemize}
\item[\refstepcounter{equation}\text{(\theequation)}\label{pb81}] $\partial
^{\beta }P_{\alpha }\left( x_{0}\right) =\delta _{\beta \alpha }$ for $\beta
,\alpha \in \mathcal{A}$,
\end{itemize}

\begin{itemize}
\item[\refstepcounter{equation}\text{(\theequation)}\label{pb82}] $%
\left\vert \partial ^{\beta }P_{\alpha }\left( x_{0}\right) \right\vert \leq
C\delta ^{\left\vert \alpha \right\vert -\left\vert \beta \right\vert }$ for
all $\alpha \in A,\beta \in \mathcal{M}$.
\end{itemize}

Assuming there exists $P \in \mathcal{P}$ such that \eqref{pb71}, \eqref{pb72}, \eqref{pb73} hold, then there exists $P$ such that \eqref{pb71}, 
\eqref{pb72} hold, and also 

\begin{itemize}
\item[\refstepcounter{equation}\text{(\theequation)}\label{pb83}] $%
\max_{\beta \in \mathcal{M}}\delta ^{\left\vert \beta \right\vert
}\left\vert \partial ^{\beta }\left( P-P^{0}\right) \left( x_{0}\right)
\right\vert =M_{0}\delta ^{m}$.
\end{itemize}

(We just take our new $P$ to be a convex combination of our original $P$ and $P_0$.)

We pick $\gamma \in \mathcal{M}$ to achieve the above $\max$. Thanks to %
\eqref{pb72}, we have

\begin{itemize}
\item[\refstepcounter{equation}\text{(\theequation)}\label{pb84}] $\gamma
\not\in \mathcal{A}$,
\end{itemize}

hence

\begin{itemize}
\item[\refstepcounter{equation}\text{(\theequation)}\label{pb85}] $\mathcal{A%
} \cup \{ \gamma \} < \mathcal{A}$ (strict inequality).
\end{itemize}

We set

\begin{itemize}
\item[\refstepcounter{equation}\text{(\theequation)}\label{pb86}] $\hat{P}^0=%
\frac{1}{2}(P^0+P)$.
\end{itemize}

From \eqref{pb71}, \eqref{pb79}, \eqref{pb80}, we have

\begin{itemize}
\item[\refstepcounter{equation}\text{(\theequation)}\label{pb87}] $\hat{P}%
^{0},\hat{P}^{0}\pm c^{\prime }M_{0}\delta ^{m-\left\vert \alpha \right\vert
}P_{\alpha }\in \Gamma \left( x_{0},CM_{0}\right) $ for $\alpha \in \mathcal{%
A}$.
\end{itemize}

Also,

\begin{itemize}
\item[\refstepcounter{equation}\text{(\theequation)}\label{pb88}] $\hat{P}%
^{0}\pm \frac{1}{2}\left( P-P^{0}\right) \in \Gamma \left( x_{0},CM\right) $
\end{itemize}

since $\hat{P}^{0}+\frac{1}{2}\left( P-P^{0}\right) =P$ and $\hat{P}^{0}-%
\frac{1}{2}\left( P-P^{0}\right) =P^{0}$.

From \eqref{pb72} and \eqref{pb83}, we have

\begin{itemize}
\item[\refstepcounter{equation}\text{(\theequation)}\label{pb89}] $\partial
^{\beta }\left( \hat{P}^{0}-P^{0}\right) \left( x_{0}\right) =0$ for all $%
\beta \in \mathcal{A}$,
\end{itemize}

and

\begin{itemize}
\item[\refstepcounter{equation}\text{(\theequation)}\label{pb90}] $%
\left\vert \partial ^{\beta }\left( \hat{P}^{0}-P^{0}\right) \left(
x_{0}\right) \right\vert \leq M_{0}\delta ^{m-\left\vert \beta \right\vert }$
for all $\beta \in \mathcal{M}$.
\end{itemize}

We define

\begin{itemize}
\item[\refstepcounter{equation}\text{(\theequation)}\label{pb91}] $P_{\gamma
}^{\#}=\left[ \partial ^{\gamma }\left( P-P^{0}\right) \left( x_{0}\right) %
\right] ^{-1}\cdot \left( P-P^{0}\right) $.
\end{itemize}

We are not dividing by zero here; by \eqref{pb83} and the definition of $%
\gamma$, we have

\begin{itemize}
\item[\refstepcounter{equation}\text{(\theequation)}\label{pb92}] $%
\left\vert \partial ^{\gamma }\left( P-P^{0}\right) \left( x_{0}\right)
\right\vert ^{-1}=M_{0}^{-1}\delta ^{\left\vert \gamma \right\vert -m}$.
\end{itemize}

From \eqref{pb72}, \eqref{pb84}, \eqref{pb91}, we have

\begin{itemize}
\item[\refstepcounter{equation}\text{(\theequation)}\label{pb93}] $\partial
^{\beta }P_{\gamma }^{\#}\left( x_{0}\right) =\delta _{\beta \gamma }$ for
all $\beta \in \mathcal{A}\cup \{\gamma \}$.
\end{itemize}

Also, \eqref{pb83}, \eqref{pb91}, \eqref{pb92} give

\begin{itemize}
\item[\refstepcounter{equation}\text{(\theequation)}\label{pb94}] $%
\left\vert \partial ^{\beta }P_{\gamma }^{\#}\left( x_{0}\right) \right\vert
\leq M_{0}^{-1}\delta ^{\left\vert \gamma \right\vert -m}\cdot M_{0}\delta
^{m-\left\vert \beta \right\vert }=\delta ^{\left\vert \gamma \right\vert
-\left\vert \beta \right\vert }$ for all $\beta \in \mathcal{M}$.
\end{itemize}

From \eqref{pb91}, \eqref{pb92}, we have $P-P^0 = \sigma M_0
\delta^{m-|\gamma|}P_\gamma^\#$ for $\sigma=1$ or $\sigma =-1$. Therefore, %
\eqref{pb88} implies that

\begin{itemize}
\item[\refstepcounter{equation}\text{(\theequation)}\label{pb95}] $\hat{P}%
^{0}\pm cM_{0}\delta ^{m-\left\vert \gamma \right\vert }P_{\gamma }^{\#}\in
\Gamma \left( x_{0},CM_{0}\right) $.
\end{itemize}

From \eqref{pb87}, \eqref{pb95} and the Trivial Remark on Convex Sets in
Section \ref{notation-and-preliminaries}, we conclude that

\begin{itemize}
\item[\refstepcounter{equation}\text{(\theequation)}\label{pb96}] $\hat{P}%
^{0}+sM_{0}\delta ^{m-\left\vert \gamma \right\vert }P_{\gamma
}^{\#}+\sum_{\alpha \in \mathcal{A}}t_{\alpha }\cdot M_{0}\delta
^{m-\left\vert \alpha \right\vert }P_{\alpha }\in \Gamma \left(
x_{0},CM_{0}\right) $, 
\end{itemize}
whenever $\left\vert s\right\vert $, $\left\vert
t_{\alpha }\right\vert \leq c$ (all $\alpha \in \mathcal{A}$) for a small
enough $c$.

For $\alpha \in \mathcal{A}$, we define

\begin{itemize}
\item[\refstepcounter{equation}\text{(\theequation)}\label{pb97}] $P_{\alpha
}^{\#}=P_{\alpha }-\left[ \partial ^{\gamma }P_{\alpha }\left( x_{0}\right) %
\right] \cdot P_{\gamma }^{\#}$.
\end{itemize}

Fix $\alpha \in \mathcal{A}$. If $\beta \in \mathcal{A}$, then \eqref{pb81}, %
\eqref{pb84}, \eqref{pb93} imply%
\begin{equation*}
\partial ^{\beta }P_{\alpha }^{\#}\left( x_{0}\right) =\partial ^{\beta
}P_{\alpha }\left( x_{0}\right) -\left[ \partial ^{\gamma }P_{\alpha }\left(
x_{0}\right) \right] \cdot \partial ^{\beta }P_{\gamma }^{\#}\left(
x_{0}\right) =\delta _{\beta \alpha }\text{.}
\end{equation*}

On the other hand, \eqref{pb84} and \eqref{pb93} yield%
\begin{equation*}
\partial ^{\gamma }P_{\alpha }^{\#}\left( x_{0}\right) =\partial ^{\gamma
}P_{\alpha }\left( x_{0}\right) -\left[ \partial ^{\gamma }P_{\alpha }\left(
x_{0}\right) \right] \cdot \partial ^{\gamma }P_{\gamma }^{\#}\left(
x_{0}\right) =0=\delta _{\gamma \alpha }\text{.}
\end{equation*}
Thus, 
\begin{equation*}
\partial ^{\beta }P_{\alpha }^{\#}\left( x_{0}\right) =\delta _{\beta \alpha
}\text{ for }\beta \in \mathcal{A\cup }\left\{ \gamma \right\} \text{, }%
\alpha \in \mathcal{A}\text{.}
\end{equation*}%
Together with \eqref{pb93}, this tells us that

\begin{itemize}
\item[\refstepcounter{equation}\text{(\theequation)}\label{pb98}] $\partial
^{\beta }P_{\alpha }^{\#}\left( x_{0}\right) =\delta _{\beta \alpha }$ for $%
\beta ,\alpha \in \mathcal{A}\cup \{\gamma \}$.
\end{itemize}

Next, we learn from \eqref{pb82}, \eqref{pb94}, \eqref{pb97} that 
\begin{eqnarray*}
\left\vert \partial ^{\beta }P_{\alpha }^{\#}\left( x_{0}\right) \right\vert
&\leq &\left\vert \partial ^{\beta }P_{\alpha }\left( x_{0}\right)
\right\vert +\left\vert \partial ^{\gamma }P_{\alpha }\left( x_{0}\right)
\right\vert \cdot \left\vert \partial ^{\beta }P_{\gamma }^{\#}\left(
x_{0}\right) \right\vert  \\
&\leq &C\delta ^{\left\vert \alpha \right\vert -\left\vert \beta \right\vert
}+C\delta ^{\left\vert \alpha \right\vert -\left\vert \gamma \right\vert
}\cdot \delta ^{\left\vert \gamma \right\vert -\left\vert \beta \right\vert }
\\
&\leq &C^{\prime }\delta ^{\left\vert \alpha \right\vert -\left\vert \beta
\right\vert }\text{ for }\alpha \in \mathcal{A},\beta \in \mathcal{M}\text{.}
\end{eqnarray*}

Together with \eqref{pb94}, this tells us that

\begin{itemize}
\item[\refstepcounter{equation}\text{(\theequation)}\label{pb99}] $%
\left\vert \partial ^{\beta }P_{\alpha }^{\#}\left( x_{0}\right) \right\vert
\leq C\delta ^{\left\vert \alpha \right\vert -\left\vert \beta \right\vert }$
for all $\alpha \in \mathcal{A}\cup \{\gamma \},\beta \in \mathcal{M}$.
\end{itemize}

Next, note that for $\alpha \in \mathcal{A}$, we have%
\begin{equation*}
M_{0}\delta ^{m-\left\vert \alpha \right\vert }P_{\alpha }^{\#}=M_{0}\delta
^{m-\left\vert \alpha \right\vert }P_{\alpha }-\left[ \delta ^{\left\vert
\gamma \right\vert -\left\vert \alpha \right\vert }\partial ^{\gamma
}P_{\alpha }\left( x_{0}\right) \right] \cdot M_{0}\delta ^{m-\left\vert
\gamma \right\vert }P_{\gamma }^{\#}\text{,}
\end{equation*}%
with $\left\vert \left[ \delta ^{\left\vert \gamma \right\vert -\left\vert
\alpha \right\vert }\partial ^{\gamma }P_{\alpha }\left( x_{0}\right) \right]
\right\vert \leq C$ by \eqref{pb82}.

Therefore, \eqref{pb96} shows that 
\begin{equation*}
\hat{P}^{0}\pm cM_{0}\delta ^{m-\left\vert \alpha \right\vert }P_{\alpha
}^{\#}\in \Gamma \left( x_{0},CM_{0}\right) \text{ for }\alpha \in \mathcal{A%
}\text{,}
\end{equation*}%
provided we take $c$ small enough. Together with \eqref{pb95}, this yields

\begin{itemize}
\item[\refstepcounter{equation}\text{(\theequation)}\label{pb100}] $\hat{P}%
^{0}\pm cM_{0}\delta ^{m-\left\vert \alpha \right\vert }P_{\alpha }^{\#}\in
\Gamma \left( x_{0},CM_{0}\right) $ for all $\alpha \in \mathcal{A}\cup
\{\gamma \}$.
\end{itemize}

Our results \eqref{pb98}, \eqref{pb99}, \eqref{pb100} tell us
that $\left( P_{\alpha }^{\#}\right) _{\alpha \in \mathcal{A\cup }\left\{
\gamma \right\} }$ is an $\left( \mathcal{A\cup }\left\{ \gamma \right\}
,\delta ,C\right) $-basis \ for $\vec{\Gamma}$ at $\left( x_{0},M_{0},\hat{P}%
^{0}\right) $.

Consequently, the Relabeling Lemma (Lemma \ref{lemma-pb2}) produces a set $%
\hat{\mathcal{A}}\subseteq \mathcal{M}$ with the following properties.

\begin{itemize}
\item[\refstepcounter{equation}\text{(\theequation)}\label{pb101}] $\hat{%
\mathcal{A}}$ is monotonic.
\end{itemize}

\begin{itemize}
\item[\refstepcounter{equation}\text{(\theequation)}\label{pb102}] $\hat{%
\mathcal{A}} \leq \mathcal{A} \cup \{ \gamma\}< \mathcal{A}$, see %
\eqref{pb85}.
\end{itemize}

\begin{itemize}
\item[\refstepcounter{equation}\text{(\theequation)}\label{pb103}] $\vec{%
\Gamma}$ has an $(\hat{\mathcal{A}},\delta,C^{\prime })$-basis at $(x_0,M_0,%
\hat{P}^0)$.
\end{itemize}

Our results \eqref{pb89}, %
\eqref{pb90}, \eqref{pb101}, \eqref{pb102}, \eqref{pb103} are the conclusions \eqref{pb74}$\cdots $\eqref{pb78} of Lemma %
\ref{lemma-pb3}.

The proof of that lemma is complete.
\end{proof}

\section{The Transport Lemma}

\label{transport-lemma}

In this section, we prove the following result.

\begin{lemma}[Transport Lemma]
\label{lemma-transport}Let $\vec{\Gamma}_{0}=\left( \Gamma _{0}\left(
x,M\right) \right) _{x\in E,M>0}$ be a shape field. For $l\geq 1$, let $\vec{%
\Gamma}_{l}=\left( \Gamma _{l}\left( x,M\right) \right) _{x\in E,M>0}$ be
the $l$-th refinement of $\vec{\Gamma}_{0}$.

\begin{itemize}
\item[\refstepcounter{equation}\text{(\theequation)}\label{t1}] Suppose $%
\mathcal{A\subseteq M}$ is monotonic and $\hat{\mathcal{A}} \subseteq 
\mathcal{M}$ (not necessarily monotonic).
\end{itemize}

Let $x_{0}\in E$, $M_{0}>0$, $l_{0}\geq 1$, $\delta >0$, $C_{B}$, $\hat{C}%
_{B}$, $C_{DIFF}>0$. Let $P^{0}$, $\hat{P}^{0}\in \mathcal{P}$.
Assume that the following hold.

\begin{itemize}
\item[\refstepcounter{equation}\text{(\theequation)}\label{t2}] $\vec{\Gamma}%
_{l_{0}}$ has an $\left( \mathcal{A},\delta ,C_{B}\right) $-basis at $\left(
x_{0},M_{0},P^{0}\right) $, and an $\left(\hat{\mathcal{A}} ,\delta ,\hat{C}%
_{B}\right) $-basis at $\left(x_{0}, M_{0},\hat{P}^{0}\right) $.
\end{itemize}

\begin{itemize}
\item[\refstepcounter{equation}\text{(\theequation)}\label{t3}] $%
\partial^\beta(P^0-\hat{P}^0)\equiv 0$ for $\beta \in \mathcal{A}$.
\end{itemize}

\begin{itemize}
\item[\refstepcounter{equation}\text{(\theequation)}\label{t4}] $%
|\partial^\beta (P^0 - \hat{P}^0)(x_0)|\leq C_{DIFF}M_0\delta^{m-|\beta|}$
for $\beta \in \mathcal{M}$.
\end{itemize}

Let $y_0 \in E$, and suppose that

\begin{itemize}
\item[\refstepcounter{equation}\text{(\theequation)}\label{t5}] $%
|x_0-y_0|\leq \epsilon_0\delta $,
\end{itemize}
where $\epsilon_0$ is a a small enough constant determined by $C_B$, $\hat{C}%
_B$, $C_{DIFF}$, $m$, $n$. Then there exists $\hat{P}^\# \in \mathcal{P}$
with the following properties.

\begin{itemize}
\item[\refstepcounter{equation}\text{(\theequation)}\label{t6}] $\vec{\Gamma}%
_{l_0-1}$ has both an $(\mathcal{A},\delta,C_B^{\prime })$-basis and an $(%
\hat{\mathcal{A}},\delta,C_B^{\prime })$-basis at $(y_0,M_0,\hat{P}^\#)$,
with $C_B^{\prime }$ determined by $C_B$, $\hat{C}_B$, $C_{DIFF}$, $m$, $n$.
\end{itemize}

\begin{itemize}
\item[\refstepcounter{equation}\text{(\theequation)}\label{t7}] $%
\partial^\beta(\hat{P}^\#-P^0)\equiv0$ for $\beta \in \mathcal{A}$.
\end{itemize}

\begin{itemize}
\item[\refstepcounter{equation}\text{(\theequation)}\label{t8}] $%
|\partial^\beta(\hat{P}^\#-P^0)(x_0)|\leq C^{\prime }M_0\delta^{m-|\beta|}$
for $\beta \in \mathcal{M}$, with $C^{\prime }$ determined by $C_B$, $\hat{C}%
_B$, $C_{DIFF}$, $m$, $n$.
\end{itemize}
\end{lemma}

\begin{remark}
Note that $\mathcal{A}$ and $\hat{\mathcal{A}}$ play different r\^oles here;
see \eqref{t1}, \eqref{t3}, and \eqref{t7}.
\end{remark}

\begin{proof}[Proof of the Transport Lemma]
In the trivial case $\mathcal{A}=\hat{\mathcal{A}}=\emptyset $, the
Transport Lemma holds simply because (by definition of the $l$-th
refinement) there exists $\hat{P}^{\#}\in \Gamma _{l_{0}-1}\left(
y_{0},C_{B}M_{0}\right) $ such that 
\begin{eqnarray*}
\left\vert \partial ^{\beta }\left( \hat{P}^{\#}-P^{0}\right) \left(
x_{0}\right) \right\vert &\leq &C_{B}M_{0}\left\vert x_{0}-y_{0}\right\vert
^{m-\left\vert \beta \right\vert } \\
&\leq &C_{B}M_{0}\delta ^{m-\left\vert \beta \right\vert }\text{ for }\beta
\in \mathcal{M}\text{.}
\end{eqnarray*}%
(Recall that $P^{0}\in \Gamma _{l_{0}}\left( x_{0},C_{B}M_{0}\right) $ since 
$\vec{\Gamma}_{l_{0}}$ has an $\left( \mathcal{A},\delta ,C_{B}\right) $%
-basis at $\left( x_{0},M_{0},P^{0}\right) $.)

From now on, we suppose that 
\begin{equation}
\#\left( \mathcal{A}\right) +\#\left( {\hat{\mathcal{A}}}\right) \not=0\text{%
.}  \label{t9}
\end{equation}

In proving the Transport Lemma, we do not yet take $\epsilon _{0}$ to be a
constant determined by $C_{B}$, $\hat{C}_{B}$, $C_{DIFF}$, $m$, $n$. Rather,
we make the following

\begin{itemize}
\item[\refstepcounter{equation}\text{(\theequation)}\label{t10}] \emph{Small 
$\epsilon_0$ assumption: $\epsilon_0$ is less than a small enough constant
determined by $C_{B}$, $\hat{C}_{B}$, $C_{DIFF}$, $m$, $n$. }
\end{itemize}

Assuming \eqref{t1}$\cdots $\eqref{t5} and \eqref{t9}, \eqref{t10}, we will
prove that there exists $\hat{P}^{\#}\in \mathcal{P}$ satisfying \eqref{t6}, %
\eqref{t7}, \eqref{t8}. Once we do so, we may then pick $\epsilon _{0}$ to
be a constant determined by $C_{B}$, $\hat{C}_{B}$, $C_{DIFF}$, $m$, $n$,
small enough to satisfy \eqref{t10}. That will complete the proof of the
Transport Lemma.

Thus, assume \eqref{t1}$\cdots$\eqref{t5} and \eqref{t9}, \eqref{t10}.

We write $c$, $C$, $C^{\prime }$, etc. to denote ``controlled constants'',
i.e., constants determined by $C_{B}$, $\hat{C}_{B}$, $C_{DIFF}$, $m$, $n$.
These symbols may denote different controlled constants in different
occurrences.

Let $(P_\alpha)_{\alpha \in \mathcal{A}}$ be an $(\mathcal{A},\delta,C_B)$%
-basis for $\vec{\Gamma}_{l_0}$ at $(x_0,M_0,P^0)$, and let $(\hat{P}%
_\alpha)_{\alpha \in \hat{\mathcal{A}}}$ be an $(\hat{\mathcal{A}},\delta,%
\hat{C}_B)$-basis for $\vec{\Gamma}_{l_0}$ at $(x_0,M_0,\hat{P}^0)$.

By definition, the following hold.

\begin{itemize}
\item[\refstepcounter{equation}\text{(\theequation)}\label{t11}] $%
P^{0}+c_{0}\sigma M_{0}\delta ^{m-\left\vert \alpha \right\vert }P_{\alpha
}\in \Gamma _{l_{0}}\left( x_{0},CM_{0}\right) $ for $\alpha \in \mathcal{A}$%
, $\sigma \in \left\{ 1,-1\right\} $.
\end{itemize}

\begin{itemize}
\item[\refstepcounter{equation}\text{(\theequation)}\label{t12}] $\hat{P}%
^{0}+\hat{c}_{0}\sigma M_{0}\delta ^{m-\left\vert \alpha \right\vert }\hat{P}%
_{\alpha }\in \Gamma _{l_{0}}\left( x_{0},CM_{0}\right) $ for $\alpha \in 
\hat{\mathcal{A}}$, $\sigma \in \left\{ 1,-1\right\} $.
\end{itemize}

\begin{itemize}
\item[\refstepcounter{equation}\text{(\theequation)}\label{t13}] $\partial
^{\beta }P_{\alpha }\left( x_{0}\right) =\delta _{\beta \alpha }$ for $%
\alpha ,\beta \in \mathcal{A}$.
\end{itemize}

\begin{itemize}
\item[\refstepcounter{equation}\text{(\theequation)}\label{t14}] $\partial
^{\beta }\hat{P}_{\alpha }\left( x_{0}\right) =\delta _{\beta \alpha }$ for $%
\alpha ,\beta \in \hat{\mathcal{A}}$.
\end{itemize}

\begin{itemize}
\item[\refstepcounter{equation}\text{(\theequation)}\label{t15}] $\left\vert
\partial ^{\beta }P_{\alpha }\left( x_{0}\right) \right\vert \leq C\delta
^{\left\vert \alpha \right\vert -\left\vert \beta \right\vert }$ for $\alpha
\in \mathcal{A}$, $\beta \in \mathcal{M}$.
\end{itemize}

\begin{itemize}
\item[\refstepcounter{equation}\text{(\theequation)}\label{t16}] $\left\vert
\partial ^{\beta }\hat{P}_{\alpha }\left( x_{0}\right) \right\vert \leq
C\delta ^{\left\vert \alpha \right\vert -\left\vert \beta \right\vert }$ for 
$\alpha \in \hat{\mathcal{A}}$, $\beta \in \mathcal{M}$.
\end{itemize}

We fix controlled constants $c_{0},\hat{c}_{0}$ as in \eqref{t11}, %
\eqref{t12}. Recall that $\vec{\Gamma}_{l_{0}}$ is the first refinement of $%
\vec{\Gamma}_{l_{0}-1}$. Therefore, by \eqref{t5}, \eqref{t10}, \eqref{t11}, there
exists $\tilde{P}_{\alpha ,\sigma }\in \Gamma _{l_{0}-1}(y_{0},CM_{0})$ ($%
\alpha \in \mathcal{A},\sigma \in \{1,-1\}$) such that 
\begin{eqnarray*}
&&\left\vert \partial ^{\beta }\left( \tilde{P}_{\alpha ,\sigma }-\left[
P^{0}+c_{0}\sigma M_{0}\delta ^{m-\left\vert \alpha \right\vert }P_{\alpha }%
\right] \right) \left( x_{0}\right) \right\vert \\
&\leq &CM_{0}\left\vert x_{0}-y_{0}\right\vert ^{m-\left\vert \beta
\right\vert }\leq C\epsilon _{0}M_{0}\delta ^{m-\left\vert \beta \right\vert
}\text{, for }\beta \in \mathcal{M}\text{.}
\end{eqnarray*}

Writing 
\begin{equation*}
E_{\alpha ,\sigma }=\frac{\tilde{P}_{\alpha ,\sigma }-\left[
P^{0}+c_{0}\sigma M_{0}\delta ^{m-\left\vert \alpha \right\vert }P_{\alpha }%
\right] }{c_{0}\sigma M_{0}\delta ^{m-\left\vert \alpha \right\vert }},
\end{equation*}
we have

\begin{itemize}
\item[\refstepcounter{equation}\text{(\theequation)}\label{t17}] $%
P^{0}+c_{0}\sigma M_{0}\delta ^{m-\left\vert \alpha \right\vert }\left(
P_{\alpha }+E_{\alpha ,\sigma }\right) \in \Gamma _{l_{0}-1}\left(
y_{0},CM_{0}\right) $ for $\alpha \in \mathcal{A},\sigma \in \left\{
1,-1\right\} $,
\end{itemize}
and
\begin{itemize}
\item[\refstepcounter{equation}\text{(\theequation)}\label{t18}] $\left\vert
\partial ^{\beta }E_{\alpha ,\sigma }\left( x_{0}\right) \right\vert \leq
C\epsilon _{0}\delta ^{\left\vert \alpha \right\vert -\left\vert \beta
\right\vert }$ for $\alpha \in \mathcal{A},\beta \in \mathcal{M}$, $\sigma
\in \left\{ 1,-1\right\} $.
\end{itemize}

Similarly, we obtain $\hat{E}_{\alpha ,\sigma }\in \mathcal{P}$ $\left( \alpha \in 
\hat{\mathcal{A}},\sigma \in \left\{ 1,-1\right\} \right) $, satisfying

\begin{itemize}
\item[\refstepcounter{equation}\text{(\theequation)}\label{t19}] $\hat{P}%
^{0}+\hat{c}_{0}\sigma M_{0}\delta ^{m-\left\vert \alpha \right\vert }\left( 
\hat{P}_{\alpha }+\hat{E}_{\alpha ,\sigma }\right) \in \Gamma
_{l_{0}-1}\left( y_{0},CM_{0}\right) $ for $\alpha \in \hat{\mathcal{A}}$, $%
\sigma \in \left\{ 1,-1\right\} $,
\end{itemize}
and
\begin{itemize}
\item[\refstepcounter{equation}\text{(\theequation)}\label{t20}] $\left\vert
\partial ^{\beta }\hat{E}_{\alpha ,\sigma }\left( x_{0}\right) \right\vert
\leq C\epsilon _{0}\delta ^{\left\vert \alpha \right\vert -\left\vert \beta
\right\vert }$ for $\alpha \in \hat{\mathcal{A}},\beta \in \mathcal{M}$, $%
\sigma \in \left\{ 1,-1\right\} $.
\end{itemize}

We introduce the following polynomials:

\begin{itemize}
\item[\refstepcounter{equation}\text{(\theequation)}\label{t21}] 
\begin{equation*}
\hat{P}^{\prime }=\frac{1}{2}\left[ \#\left( \mathcal{A}\right) +\#\left( 
\hat{\mathcal{A}}\right) \right] ^{-1}\left( 
\begin{array}{c}
\sum_{\alpha \in \mathcal{A}\text{,}\sigma =\pm 1}\left\{ P^{0}+c_{0}\sigma
M_{0}\delta ^{m-\left\vert \alpha \right\vert }\left( P_{\alpha }+E_{\alpha
,\sigma }\right) \right\}  \\ 
+\sum_{\alpha \in \hat{\mathcal{A}}\text{,}\sigma =\pm 1}\left\{ \hat{P}^{0}+%
\hat{c}_{0}\sigma M_{0}\delta ^{m-\left\vert \alpha \right\vert }\left( \hat{%
P}_{\alpha }+\hat{E}_{\alpha ,\sigma }\right) \right\} 
\end{array}%
\right) 
\end{equation*}%
(see \eqref{t9});
\end{itemize}

\begin{itemize}
\item[\refstepcounter{equation}\text{(\theequation)}\label{t22}] 
\begin{eqnarray*}
P_{\alpha }^{\prime } &=&\frac{1}{2c_{0}M_{0}\delta ^{m-\left\vert \alpha
\right\vert }}\left( 
\begin{array}{c}
\left\{ P^{0}+c_{0}M_{0}\delta ^{m-\left\vert \alpha \right\vert }\left(
P_{\alpha }+E_{\alpha ,1}\right) \right\}  \\ 
-\left\{ P^{0}-c_{0}M_{0}\delta ^{m-\left\vert \alpha \right\vert }\left(
P_{\alpha }+E_{\alpha ,-1}\right) \right\} 
\end{array}%
\right)  \\
&=&P_{\alpha }+\frac{1}{2}\left( E_{\alpha ,1}+E_{\alpha ,-1}\right) \text{
for }\alpha \in \mathcal{A}\text{;}
\end{eqnarray*}
\end{itemize}

\begin{itemize}
\item[\refstepcounter{equation}\text{(\theequation)}\label{t23}] 
\begin{eqnarray*}
\hat{P}_{\alpha }^{\prime } &=&\frac{1}{2\hat{c}_{0}M_{0}\delta
^{m-\left\vert \alpha \right\vert }}\left( 
\begin{array}{c}
\left\{ \hat{P}^{0}+\hat{c}_{0}M_{0}\delta ^{m-\left\vert \alpha \right\vert
}\left( \hat{P}_{\alpha }+\hat{E}_{\alpha ,1}\right) \right\}  \\ 
-\left\{ \hat{P}^{0}-\hat{c}_{0}M_{0}\delta ^{m-\left\vert \alpha
\right\vert }\left( \hat{P}_{\alpha }+\hat{E}_{\alpha ,-1}\right) \right\} 
\end{array}%
\right)  \\
&=&\hat{P}_{\alpha }+\frac{1}{2}\left( \hat{E}_{\alpha ,1}+\hat{E}_{\alpha
,-1}\right) \text{ for }\alpha \in \hat{\mathcal{A}}.
\end{eqnarray*}
\end{itemize}

For a small enough controlled constant $c_1$, we have

\begin{itemize}
\item[\refstepcounter{equation}\text{(\theequation)}\label{t24}] $\hat{P}%
^{\prime }+c_{1}M_{0}\delta ^{m-\left\vert \alpha \right\vert }P_{\alpha
}^{\prime },\hat{P}^{\prime }-c_{1}M_{0}\delta ^{m-\left\vert \alpha
\right\vert }P_{\alpha }^{\prime }\in \Gamma _{l_{0}-1}\left(
y_{0},CM_{0}\right) $ for $\alpha \in \mathcal{A}$,
\end{itemize}

and

\begin{itemize}
\item[\refstepcounter{equation}\text{(\theequation)}\label{t25}] $\hat{P}%
^{\prime }+c_{1}M_{0}\delta ^{m-\left\vert \alpha \right\vert }\hat{P}%
_{\alpha }^{\prime },\hat{P}^{\prime }-c_{1}M_{0}\delta ^{m-\left\vert
\alpha \right\vert }\hat{P}_{\alpha }^{\prime }\in \Gamma _{l_{0}-1}\left(
y_{0},CM_{0}\right) $ for $\alpha \in \hat{\mathcal{A}}$,
\end{itemize}
because each of the polynomials in \eqref{t24}, \eqref{t25} is a convex
combination of the polynomials in \eqref{t17}, \eqref{t19}. (In fact, we can take $c_1=\frac{c_0}{\#(\mathcal{A})+\#(\hat{\mathcal{A}})}$.)

From \eqref{t24}, \eqref{t25} and the Trivial Remark on Convex Sets in
Section \ref{notation-and-preliminaries}, we obtain the following, for a
small enough controlled constant $c_2$.

\begin{itemize}
\item[\refstepcounter{equation}\text{(\theequation)}\label{t26}] $\hat{P}%
^{\prime }+\sum_{\alpha \in \mathcal{A}}s_{\alpha }M_{0}\delta
^{m-\left\vert \alpha \right\vert }P_{\alpha }^{\prime }+\sum_{\alpha \in 
\hat{\mathcal{A}}}t_{\alpha }M_{0}\delta ^{m-\left\vert \alpha \right\vert }%
\hat{P}_{\alpha }^{\prime }\in \Gamma _{l_{0}-1}\left( y_{0},CM_{0}\right) $%
, whenever $\left\vert s_{\alpha }\right\vert \leq c_{2}$ for all $\alpha
\in \mathcal{A}$ and $\left\vert t_{\alpha }\right\vert \leq c_{2}$ for all $%
\alpha \in \hat{\mathcal{A}}$.
\end{itemize}

Note also that \eqref{t21} may be written in the equivalent form

\begin{itemize}
\item[\refstepcounter{equation}\text{(\theequation)}\label{t27}] 
\begin{eqnarray*}
&&\hat{P}^{\prime }=P^{0}+\left[ \frac{\#\left( \hat{\mathcal{A}}\right) }{%
\#\left( \mathcal{A}\right) +\#\hat{\left( \mathcal{A}\right) }}\right]
\left( \hat{P}^{0}-P^{0}\right)  \\
&&+\frac{1}{2\left[ \#\left( \mathcal{A}\right) +\#\left( \hat{\mathcal{A}}%
\right) \right] }\left\{ 
\begin{array}{c}
\sum_{\alpha \in \mathcal{A}\text{,}\sigma =\pm 1}c_{0}\sigma M_{0}\delta
^{m-\left\vert \alpha \right\vert }E_{\alpha ,\sigma } \\ 
+\sum_{\alpha \in \hat{\mathcal{A}}\text{,}\sigma =\pm 1}\hat{c}_{0}\sigma
M_{0}\delta ^{m-\left\vert \alpha \right\vert }\hat{E}_{\alpha ,\sigma }%
\end{array}%
\right\} \text{.}
\end{eqnarray*}
\end{itemize}

Consequently, \eqref{t3}, \eqref{t4}, \eqref{t18}, \eqref{t20} tell us that

\begin{itemize}
\item[\refstepcounter{equation}\text{(\theequation)}\label{t28}] $\left\vert
\partial ^{\beta }\left( \hat{P}^{\prime }-P^{0}\right) \left( x_{0}\right)
\right\vert \leq C\epsilon _{0}M_{0}\delta ^{m-\left\vert \beta \right\vert }
$ for $\beta \in \mathcal{A}$;
\end{itemize}

\begin{itemize}
\item[\refstepcounter{equation}\text{(\theequation)}\label{t29}] $\left\vert
\partial ^{\beta }\left( \hat{P}^{\prime}-P^{0}\right) \left( x_{0}\right)
\right\vert \leq CM_{0}\delta ^{m-\left\vert \beta \right\vert }$ for $\beta
\in \mathcal{M}$.
\end{itemize}

Similarly, \eqref{t13}$\cdots$\eqref{t16}, \eqref{t18}, \eqref{t20}, and %
\eqref{t22}, \eqref{t23} together imply the estimates

\begin{itemize}
\item[\refstepcounter{equation}\text{(\theequation)}\label{t30}] $\left\vert
\partial ^{\beta }P_{\alpha }^{\prime }\left( x_{0}\right) -\delta _{\beta
\alpha }\right\vert \leq C\epsilon _{0}\delta ^{\left\vert \alpha
\right\vert -\left\vert \beta \right\vert }$ for $\alpha ,\beta \in \mathcal{%
A}$;
\end{itemize}

\begin{itemize}
\item[\refstepcounter{equation}\text{(\theequation)}\label{t31}] $\left\vert
\partial ^{\beta }\hat{P}_{\alpha }^{\prime }\left( x_{0}\right) -\delta
_{\beta \alpha }\right\vert \leq C\epsilon _{0}\delta ^{\left\vert \alpha
\right\vert -\left\vert \beta \right\vert }$ for $\alpha ,\beta \in \hat{%
\mathcal{A}}$;
\end{itemize}

\begin{itemize}
\item[\refstepcounter{equation}\text{(\theequation)}\label{t32}] $\left\vert
\partial ^{\beta }P_{\alpha }^{\prime }\left( x_{0}\right) \right\vert \leq
C\delta ^{\left\vert \alpha \right\vert -\left\vert \beta \right\vert }$ for 
$\alpha \in \mathcal{A}$, $\beta \in \mathcal{M}$; and
\end{itemize}

\begin{itemize}
\item[\refstepcounter{equation}\text{(\theequation)}\label{t33}] $\left\vert
\partial ^{\beta }\hat{P}_{\alpha }^{\prime }\left( x_{0}\right) \right\vert
\leq C\delta ^{\left\vert \alpha \right\vert -\left\vert \beta \right\vert }$
for $\alpha \in \hat{\mathcal{A}}$, $\beta \in \mathcal{M}$.
\end{itemize}

From \eqref{t30}$\cdots$\eqref{t33} and \eqref{t5}, we have also

\begin{itemize}
\item[\refstepcounter{equation}\text{(\theequation)}\label{t34}] $\left\vert
\partial ^{\beta }P_{\alpha }^{\prime }\left( y_{0}\right) -\delta _{\beta
\alpha }\right\vert \leq C\epsilon _{0}\delta ^{\left\vert \alpha
\right\vert -\left\vert \beta \right\vert }$ for $\beta ,\alpha \in \mathcal{%
A}$;
\end{itemize}

\begin{itemize}
\item[\refstepcounter{equation}\text{(\theequation)}\label{t35}] $\left\vert
\partial ^{\beta }\hat{P}_{\alpha }^{\prime }\left( y_{0}\right) -\delta
_{\beta \alpha }\right\vert \leq C\epsilon _{0}\delta ^{\left\vert \alpha
\right\vert -\left\vert \beta \right\vert }$ for $\beta ,\alpha \in \hat{%
\mathcal{A}}$;
\end{itemize}

\begin{itemize}
\item[\refstepcounter{equation}\text{(\theequation)}\label{t36}] $\left\vert
\partial ^{\beta }P_{\alpha }^{\prime }\left( y_{0}\right) \right\vert \leq
C\delta ^{\left\vert \alpha \right\vert -\left\vert \beta \right\vert }$ for 
$\alpha \in \mathcal{A}$, $\beta \in \mathcal{M}$; and
\end{itemize}

\begin{itemize}
\item[\refstepcounter{equation}\text{(\theequation)}\label{t37}] $\left\vert
\partial ^{\beta }\hat{P}_{\alpha }^{\prime }\left( y_{0}\right) \right\vert
\leq C\delta ^{\left\vert \alpha \right\vert -\left\vert \beta \right\vert }$
for $\alpha \in \hat{\mathcal{A}},\beta \in \mathcal{M}$.
\end{itemize}

Next, we prove that there exists $\hat{P}^{\#}\in \mathcal{P}$ with the
following properties:

\begin{itemize}
\item[\refstepcounter{equation}\text{(\theequation)}\label{t38}] $\partial
^{\beta }\left( \hat{P}^{\#}-P^{0}\right) \left( x_{0}\right) =0$ for $\beta
\in \mathcal{A}$;
\end{itemize}

\begin{itemize}
\item[\refstepcounter{equation}\text{(\theequation)}\label{t39}] $\left\vert
\partial ^{\beta }\left( \hat{P}^{\#}-P^{0}\right) \left( x_{0}\right)
\right\vert \leq CM_{0}\delta ^{m-\left\vert \beta \right\vert }$ for $\beta
\in \mathcal{M}$;
\end{itemize}

for a small enough controlled constant $c_{3}$, we have

\begin{itemize}
\item[\refstepcounter{equation}\text{(\theequation)}\label{t40}] $\hat{P}%
^{\#}+\sum_{\alpha \in \mathcal{A}}s_{\alpha }M_{0}\delta ^{m-\left\vert
\alpha \right\vert }P_{\alpha }^{\prime }+\sum_{\alpha \in \hat{\mathcal{A}}%
}t_{\alpha }M_{0}\delta ^{m-\left\vert \alpha \right\vert }\hat{P}_{\alpha
}^{\prime }\in \Gamma _{l_{0}-1}\left( y_{0},CM_{0}\right) $, whenever all $%
\left\vert s_{\alpha }\right\vert $, $\left\vert t_{\alpha }\right\vert $
are less than $c_{3}$.
\end{itemize}

Indeed, if $\mathcal{A}=\emptyset $, we set $\hat{P}^{\#}=\hat{P}^{\prime }
$; then \eqref{t38} holds vacuously, and \eqref{t39}, \eqref{t40} simply
restate \eqref{t29}, \eqref{t26}. Suppose $\mathcal{A}\not=\emptyset $. We
will pick coefficients $s_{\alpha }^{\#}$ ($\alpha \in \mathcal{A}$) for
which

\begin{itemize}
\item[\refstepcounter{equation}\text{(\theequation)}\label{t41}] $\hat{P}%
^{\#}:=\hat{P}^{\prime }+\sum_{\alpha \in \mathcal{A}}s_{\alpha
}^{\#}M_{0}\delta ^{m-\left\vert \alpha \right\vert }P_{\alpha }^{\prime }$
satisfies \eqref{t38}, \eqref{t39}, \eqref{t40}.
\end{itemize}

In fact, with $\hat{P}^\#$ given by \eqref{t41}, equation \eqref{t38} is
equivalent to the system of the linear equations

\begin{itemize}
\item[\refstepcounter{equation}\text{(\theequation)}\label{t42}] $%
\sum_{\alpha \in \mathcal{A}}\left[ \delta ^{\left\vert \beta \right\vert
-\left\vert \alpha \right\vert }\partial ^{\beta }P_{\alpha }^{\prime
}\left( x_{0}\right) \right] s_{\alpha }^{\#}=-M_{0}^{-1}\delta ^{\left\vert
\beta \right\vert -m}\partial ^{\beta }\left( \hat{P}^{\prime }-P^{0}\right)
\left( x_{0}\right) $ $\left( \beta \in \mathcal{A}\right) $.
\end{itemize}

By \eqref{t28}, the right-hand side of \eqref{t42} has absolute
value at most $C\epsilon _{0}$. Hence, by \eqref{t30} and the {\em small $%
\epsilon _{0}$ assumption} \eqref{t10}, we can solve \eqref{t42} for the $%
s_{\alpha }^{\#}$, and we have

\begin{itemize}
\item[\refstepcounter{equation}\text{(\theequation)}\label{t43}] $\left\vert
s_{\alpha }^{\#}\right\vert \leq C\epsilon _{0}$ for all $\alpha \in 
\mathcal{A}$.
\end{itemize}

The resulting $\hat{P}^{\#}$ given by \eqref{t41} then satisfies \eqref{t38}%
. Moreover, for $\beta \in \mathcal{M}$, we have%
\begin{eqnarray*}
\left\vert \partial ^{\beta }\left( \hat{P}^{\#}-P^{0}\right) \left(
x_{0}\right) \right\vert  &\leq &\left\vert \partial ^{\beta }\left( \hat{P}%
^{\prime }-P^{0}\right) \left( x_{0}\right) \right\vert +\sum_{\alpha \in 
\mathcal{A}}\left\vert s_{\alpha }^{\#}\right\vert \cdot M_{0}\delta
^{m-\left\vert \alpha \right\vert }\left\vert \partial ^{\beta }P_{\alpha
}^{\prime }\left( x_{0}\right) \right\vert  \\
&\leq &\left\vert \partial ^{\beta }\left( \hat{P}^{\prime }-P^{0}\right)
\left( x_{0}\right) \right\vert +C\epsilon _{0}\delta ^{m-\left\vert \beta
\right\vert }M_{0}\text{,}
\end{eqnarray*}%
thanks to \eqref{t32}, \eqref{t41}, \eqref{t43}.

Therefore, \eqref{t29} gives%
\begin{equation*}
\left\vert \partial ^{\beta }\left( \hat{P}^{\#}-P^{0}\right) \left(
x_{0}\right) \right\vert \leq CM_{0}\delta ^{m-\left\vert \beta \right\vert }%
\text{ for }\beta \in \mathcal{M}\text{,}
\end{equation*}
proving \eqref{t39}.

Finally, \eqref{t40} follows at once from \eqref{t26}, \eqref{t41}, %
\eqref{t43} and the {\em small $\epsilon_0$ assumption} \eqref{t10}.

Thus, in all cases, there exists $\hat{P}^{\#}$ satisfying \eqref{t38}, %
\eqref{t39}, \eqref{t40}. We fix such a $\hat{P}^{\#}$.

Next, we produce an $(\mathcal{A},\delta ,C)$-basis for $\vec{\Gamma}%
_{l_{0}-1}$ at $(y_{0},M_{0},\hat{P}^{\#})$. To do so, we first suppose that 
$\mathcal{A}\not=\emptyset $, and set

\begin{itemize}
\item[\refstepcounter{equation}\text{(\theequation)}\label{t44}] $P_{\gamma
}^{\#}=\sum_{\alpha \in \mathcal{A}}b_{\gamma \alpha }\delta ^{\left\vert
\gamma \right\vert -\left\vert \alpha \right\vert }P_{\alpha }^{\prime }$
\end{itemize}

for real coefficients $(b_{\gamma \alpha })_{\gamma ,\alpha \in \mathcal{A}}$
to be picked below. For $\beta ,\gamma \in \mathcal{A}$, we have 
\begin{equation*}
\partial ^{\beta }P_{\gamma }^{\#}\left( y_{0}\right) =\delta ^{\left\vert
\gamma \right\vert -\left\vert \beta \right\vert }\cdot \sum_{\alpha \in 
\mathcal{A}}b_{\gamma \alpha }\left[ \delta ^{\left\vert \beta \right\vert
-\left\vert \alpha \right\vert }\partial ^{\beta }P_{\alpha }^{\prime
}\left( y_{0}\right) \right] \text{.}
\end{equation*}

Thanks to \eqref{t34} and the {\em small $\epsilon_0$ assumption \eqref{t10}}, we
may define $(b_{\gamma\alpha})_{\gamma,\alpha \in \mathcal{A}}$ as the
inverse matrix of $\left(\delta^{|\beta|-|\alpha|}\partial^\beta
P_\alpha^{\prime }(y_0)\right)_{\alpha,\beta \in \mathcal{A}}$, and we then
have

\begin{itemize}
\item[\refstepcounter{equation}\text{(\theequation)}\label{t45}] $\partial
^{\beta }P_{\gamma }^{\#}\left( y_{0}\right) =\delta _{\beta \gamma }$ $%
\left( \beta ,\gamma \in \mathcal{A}\right) $
\end{itemize}
and
\begin{itemize}
\item[\refstepcounter{equation}\text{(\theequation)}\label{t46}] $\left\vert
b_{\gamma \alpha }-\delta _{\gamma \alpha }\right\vert \leq C\epsilon _{0} $
for $\gamma ,\alpha \in \mathcal{A}$.
\end{itemize}

In particular,

\begin{itemize}
\item[\refstepcounter{equation}\text{(\theequation)}\label{t47}] $\left\vert
b_{\gamma \alpha }\right\vert \leq C$, 
\end{itemize}
and therefore for $\gamma \in 
\mathcal{A}$, $\beta \in \mathcal{M}$ we have
\begin{itemize}
\item[\refstepcounter{equation}\text{(\theequation)}\label{t48}] 
\begin{eqnarray*}
\left\vert \partial ^{\beta }P_{\gamma }^{\#}\left( y_{0}\right) \right\vert
&\leq &\sum_{\alpha \in \mathcal{A}}\left\vert b_{\gamma \alpha }\right\vert
\delta ^{\left\vert \gamma \right\vert -\left\vert \alpha \right\vert
}\left\vert \partial ^{\beta }P_{\alpha }^{\prime }\left( y_{0}\right)
\right\vert  \\
&\leq &C\sum_{\alpha \in \mathcal{A}}\delta ^{\left\vert \gamma \right\vert
-\left\vert \alpha \right\vert }\delta ^{\left\vert \alpha \right\vert
-\left\vert \beta \right\vert }\leq C^{\prime }\delta ^{\left\vert \gamma
\right\vert -\left\vert \beta \right\vert }\text{,}
\end{eqnarray*}
\end{itemize}
thanks to \eqref{t36}.

Also, for $\gamma \in \mathcal{A}$, we have%
\begin{equation*}
\left[ M_{0}\delta ^{m-\left\vert \gamma \right\vert }P_{\gamma }^{\#}\right]
=\sum_{\alpha \in \mathcal{A}}b_{\gamma \alpha }\cdot \left[ M_{0}\delta
^{m-\left\vert \alpha \right\vert }P_{\alpha }^{\prime }\right] \text{.}
\end{equation*}

Therefore, for a small enough controlled constant $c_4$, we have

\begin{itemize}
\item[\refstepcounter{equation}\text{(\theequation)}\label{t49}] $\hat{P}%
^{\#}+c_{4}M_{0}\delta ^{m-\left\vert \gamma \right\vert }P_{\gamma }^{\#}$, 
$\hat{P}^{\#}-c_{4}M_{0}\delta ^{m-\left\vert \gamma \right\vert }P_{\gamma
}^{\#}\in \Gamma _{l_{0}-1}\left( y_{0},CM_{0}\right) $ for $\gamma \in 
\mathcal{A}$,
\end{itemize}
thanks to \eqref{t40}, which in turn holds thanks to \eqref{t47}. Since we are assuming that $\mathcal{A}
\not=\emptyset $, \eqref{t49} implies that also

\begin{itemize}
\item[\refstepcounter{equation}\text{(\theequation)}\label{t50}] $\hat{P}%
^{\#}\in \Gamma _{l_{0}-1}\left( y_{0},CM_{0}\right) $.
\end{itemize}

Our results \eqref{t45}, \eqref{t48}, \eqref{t49}, \eqref{t50} tell us that $%
\left( P_{\gamma }^{\#}\right) _{\gamma \in \mathcal{A}}$ is an $(\mathcal{A}%
,\delta ,C)$-basis for $\vec{\Gamma}_{l_{0}-1}$ at $\left( y_{0},M_{0},\hat{P%
}^{\#}\right) $.

Thus, we have produced the desired $(\mathcal{A},\delta ,C)$-basis, provided
that $\mathcal{A}\not=\emptyset $. On the other hand, if $\mathcal{A}%
=\emptyset $, then the existence of an $(\mathcal{A},\delta ,C)$-basis for $%
\vec{\Gamma}_{l_{0}-1}$ at $\left( y_{0},M_{0},\hat{P}^{\#}\right) $ is
equivalent to the assertion that

\begin{itemize}
\item[\refstepcounter{equation}\text{(\theequation)}\label{t51}] $\hat{P}%
^{\#}\in \Gamma _{l_{0}-1}\left( y_{0},CM_{0}\right) $,
\end{itemize}

and \eqref{t51} follows at once from \eqref{t40}. Thus, in all cases,

\begin{itemize}
\item[\refstepcounter{equation}\text{(\theequation)}\label{t52}] $\vec{\Gamma%
}_{l_{0}-1}$ has an $\left( \mathcal{A},\delta ,C\right) $-basis at $\left(
y_{0},M_{0},\hat{P}^{\#}\right) $.
\end{itemize}

Similarly, we can produce an $(\hat{\mathcal{A}},\delta ,C)$-basis for $\vec{%
\Gamma}_{l_{0}-1}$ at $\left( y_{0},M_{0},\hat{P}^{\#}\right) $. We suppose
first that $\hat{\mathcal{A}}\not=\emptyset $, and set

\begin{itemize}
\item[\refstepcounter{equation}\text{(\theequation)}\label{t53}] $\hat{P}%
_{\gamma }^{\#}=\sum_{\alpha \in \hat{\mathcal{A}}}\hat{b}_{\beta \alpha
}\delta ^{\left\vert \gamma \right\vert -\left\vert \alpha \right\vert }\hat{%
P}_{\alpha }^{\prime }$ for $\gamma \in \hat{\mathcal{A}}$, with
coefficients $\hat{b}_{\gamma \alpha }$ to be picked below.
\end{itemize}

Thanks to \eqref{t35} and the \emph{small $\epsilon_0$ assumption} %
\eqref{t10}, we can pick the coefficients $\hat{b}_{\gamma\alpha}$ so that

\begin{itemize}
\item[\refstepcounter{equation}\text{(\theequation)}\label{t54}] $\partial
^{\beta }\hat{P}_{\gamma }^{\#}\left( y_{0}\right) =\delta _{\beta ,\gamma }$
for $\beta ,\gamma \in \hat{\mathcal{A}}$
\end{itemize}

and

\begin{itemize}
\item[\refstepcounter{equation}\text{(\theequation)}\label{t55}] $\left\vert 
\hat{b}_{\gamma \alpha }-\delta _{\gamma \alpha }\right\vert \leq C\epsilon
_{0}$ for $\gamma ,\alpha \in \hat{\mathcal{A}}$,
\end{itemize}

hence

\begin{itemize}
\item[\refstepcounter{equation}\text{(\theequation)}\label{t56}] $\left\vert 
\hat{b}_{\gamma \alpha }\right\vert \leq C$ for $\gamma ,\alpha \in \hat{%
\mathcal{A}}$.
\end{itemize}

From \eqref{t37}, \eqref{t53}, \eqref{t56}, we obtain the estimate

\begin{itemize}
\item[\refstepcounter{equation}\text{(\theequation)}\label{t57}] $\left\vert
\partial ^{\beta }\hat{P}_{\gamma }^{\#}\left( y_{0}\right) \right\vert \leq
C\delta ^{\left\vert \gamma \right\vert -\left\vert \beta \right\vert }$ for 
$\gamma \in \hat{\mathcal{A}}$ and $\beta \in \mathcal{M}$,
\end{itemize}

in analogy with \eqref{t48}. Also for $\gamma \in \hat{\mathcal{A}}$, we have%
\begin{equation*}
M_{0}\delta ^{m-\left\vert \gamma \right\vert }\hat{P}_{\gamma
}^{\#}=\sum_{\alpha \in \hat{\mathcal{A}}}\hat{b}_{\gamma \alpha }\left[
M_{0}\delta ^{m-\left\vert \alpha \right\vert }\hat{P}_{\alpha }^{\prime }%
\right] \text{.}
\end{equation*}

Together with \eqref{t40} and \eqref{t56}, this tells us that

\begin{itemize}
\item[\refstepcounter{equation}\text{(\theequation)}\label{t58}] $\hat{P}%
^{\#}+c_{5}M_{0}\delta ^{m-\left\vert \gamma \right\vert }\hat{P}_{\gamma
}^{\#}$, $\hat{P}^{\#}-c_{5}M_{0}\delta ^{m-\left\vert \gamma \right\vert }%
\hat{P}_{\gamma }^{\#}\in \Gamma _{l_{0}-1}\left( y_{0},CM_{0}\right) $ for $%
\gamma \in \hat{\mathcal{A}}$
\end{itemize}

in analogy with \eqref{t49}. Since we are assuming that $\hat{\mathcal{A}}%
\not=\emptyset $, \eqref{t58} implies that

\begin{itemize}
\item[\refstepcounter{equation}\text{(\theequation)}\label{t59}] $\hat{P}%
^{\#}\in \Gamma _{l_{0}-1}\left( y_{0},CM_{0}\right) $.
\end{itemize}

Our results \eqref{t54},\eqref{t57}, \eqref{t58}, \eqref{t59} tell us that $%
\left( \hat{P}_{\gamma }^{\#}\right) _{\gamma \in \hat{\mathcal{A}}}$ is an $%
(\hat{\mathcal{A}},\delta ,C)$-basis for $\vec{\Gamma}_{l_{0}-1}$ at $\left(
y_{0},M_{0},\hat{P}^{\#}\right) $.

Thus, we have produced the desired $(\hat{\mathcal{A}},\delta ,C)$-basis,
provided $\hat{\mathcal{A}}\not=\emptyset $. On the other hand, if $\hat{%
\mathcal{A}}=\emptyset $, then the existence of an $(\hat{\mathcal{A}}%
,\delta ,C)$-basis for $\vec{\Gamma}_{l_{0}-1}$ at $\left( y_{0},M_{0},\hat{P%
}^{\#}\right) $ is equivalent to the assertion that

\begin{itemize}
\item[\refstepcounter{equation}\text{(\theequation)}\label{t60}] $\hat{P}%
^{\#}\in \Gamma _{l_{0}-1}\left( y_{0},CM_{0}\right) $,
\end{itemize}

and \eqref{t60} follows at once from \eqref{t40}. Thus, in all cases,

\begin{itemize}
\item[\refstepcounter{equation}\text{(\theequation)}\label{t61}] $\vec{\Gamma%
}_{l_{0}-1}$ has an $\left( \hat{\mathcal{A}},\delta ,C\right) $-basis at $%
\left( y_{0},M_{0},\hat{P}^{\#}\right) $.
\end{itemize}

Our results \eqref{t52} and \eqref{t61} together yield conclusion \eqref{t6}
of the Transport Lemma (Lemma \ref{lemma-transport}). Also, our results %
\eqref{t38} and \eqref{t39} imply conclusions \eqref{t7} and \eqref{t8},
since $\mathcal{A}$ is monotonic. (See \eqref{t1}.)

Thus, starting from assumptions \eqref{t1}$\cdots$\eqref{t5} and \eqref{t9}, %
\eqref{t10}, we have proven conclusions \eqref{t6}, \eqref{t7}, \eqref{t8}
for our $\hat{P}^\#$.

The proof of the Transport Lemma (Lemma \ref{lemma-transport}) is complete.
\end{proof}

\begin{remark}
The monotonicity of $\mathcal{A}$ was used in the proof of the Transport Lemma only to gurantee that formulas \eqref{t3} and \eqref{t7} are independent of the base point. 
\end{remark}

\part{The Main Lemma}

\section{Statement of the Main Lemma}

\label{statement-of-the-main-lemma}

For $\mathcal{A}\subseteq \mathcal{M}$ monotonic, we define

\begin{itemize}
\item[\refstepcounter{equation}\text{(\theequation)}\label{m1}] $l\left( 
\mathcal{A}\right) =1+3\cdot \#\left\{ \mathcal{A}^{\prime }\subseteq 
\mathcal{M}:\mathcal{A}^{\prime }\text{ monotonic, }\mathcal{A}^{\prime }<%
\mathcal{A}\right\} $.
\end{itemize}

Thus,

\begin{itemize}
\item[\refstepcounter{equation}\text{(\theequation)}\label{m2}] $l\left( 
\mathcal{A}\right) -3\geq l\left( \mathcal{A}^{\prime }\right) $ for $%
\mathcal{A}^{\prime },\mathcal{A}\subseteq \mathcal{M}$ monotonic with $%
\mathcal{A}^{\prime }<\mathcal{A}$.
\end{itemize}

By induction on $\mathcal{A}$ (with respect to the order relation $<$), we
will prove the following result.

{\textbf{Main Lemma for $\mathcal{A}$}}\thinspace \thinspace {\textit{Let $%
\vec{\Gamma}_{0}=\left( \Gamma _{0}\left( x,M\right) \right) _{x\in E,M>0}$
be a $\left( C_{w},\delta _{\max }\right) $-convex shape field, and for $%
l\geq 1$, let $\vec{\Gamma}_{l}=\left( \Gamma _{l}\left( x,M\right) \right)
_{x\in E,M>0}$ be the $l$-th refinement of $\vec{\Gamma}_{0}$. Fix a dyadic
cube $Q_{0}\subset \mathbb{R}^{n}$, a point $x_{0}\in E\cap 5\left(
Q_{0}^{+}\right) $ and a polynomial $P^{0}\in \mathcal{P}$, as well as
positive real numbers $M_{0}$, $\epsilon $, $C_{B}$. We make the following
assumptions.}}

\begin{itemize}
\item[(A1)] \textit{$\vec{\Gamma}_{l\left( \mathcal{A}\right) }$ has an $%
\left( \mathcal{A},\epsilon ^{-1}\delta _{Q_{0}},C_{B}\right) $-basis at $%
\left( x_{0},M_{0},P^{0}\right) $.}

\item[(A2)] $\epsilon ^{-1}\delta _{Q_{0}}\leq \delta _{\max }$.

\item[(A3)] (\textquotedblleft Small $\epsilon $ Assumption")\textit{\ $%
\epsilon $ is less than a small enough constant determined by $C_{B}$, $C_{w}
$, $m$, $n$.}
\end{itemize}

\textit{Then there exists $F\in C^{m}\left( \frac{65}{64}Q_{0}\right) $
satisfying the following conditions.}

\begin{itemize}
\item[(C1)] \textit{$\left\vert \partial ^{\beta }\left( F-P^{0}\right)
\right\vert \leq C\left( \epsilon \right) M_{0}\delta _{Q_{0}}^{m-\left\vert
\beta \right\vert }$ on $\frac{65}{64}Q_{0}$ for $\left\vert \beta
\right\vert \leq m$, where $C\left( \epsilon \right) $ is determined by $%
\epsilon $, $C_{B}$, $C_{w}$, $m$, $n$.}

\item[(C2)] \textit{$J_{z}\left( F\right) \in \Gamma _{0}\left( z,C^{\prime
}\left( \epsilon \right) M_{0}\right) $ for all $z\in E\cap \frac{65}{64}%
Q_{0}$, where $C^{\prime }\left( \epsilon \right) $ is determined by $%
\epsilon $, $C_{B}$, $C_{w}$, $m$, $n$.}
\end{itemize}

\begin{remarks} 
\begin{itemize}
\item We state the Main Lemma only for monotonic $\mathcal{A}$.
\item Note that $x_0$ may fail to belong to $\frac{65}{64}Q_0$, hence the
assertion $J_{x_0}(F)=P^0$ may be meaningless. Even if $x_0 \in \frac{65}{64}%
Q_0$, we do not assert that $J_{x_0}(F)=P^0$. However, see Corollary \ref{lemma-to-previous-results-on-shape-fields} in Section \ref{fp-i} below. 
\end{itemize}
\end{remarks}

\section{The Base Case}

\label{the-base-case} The base case of our induction on $\mathcal{A}$ is the
case $\mathcal{A}= \mathcal{M}$.

In this section, we prove the Main Lemma for $\mathcal{M}$. The hypotheses
of the lemma are as follows:

\begin{itemize}
\item[\refstepcounter{equation}\text{(\theequation)}\label{b1}] $\vec{\Gamma}%
_{0}=\left( \Gamma _{0}\left( x,M\right) \right) _{x\in E,M>0}$ is a $\left(
C_{w},\delta _{\max }\right) $-convex shape field.
\end{itemize}

\begin{itemize}
\item[\refstepcounter{equation}\text{(\theequation)}\label{b2}] $\vec{\Gamma}%
_{1}=\left( \Gamma _{1}\left( x,M\right) \right) _{x\in E,M>0}$ is the first
refinement of $\vec{\Gamma}_{0}$.
\end{itemize}

\begin{itemize}
\item[\refstepcounter{equation}\text{(\theequation)}\label{b3}] $\vec{\Gamma}%
_{1}$ has an $\left( \mathcal{M},\epsilon ^{-1}\delta _{Q_{0}},C_{B}\right) $%
-basis at $\left( x_{0},M_{0},P^{0}\right) $.
\end{itemize}

\begin{itemize}
\item[\refstepcounter{equation}\text{(\theequation)}\label{b4}] $\epsilon
^{-1}\delta _{Q_{0}}\leq \delta _{\max }$.
\end{itemize}

\begin{itemize}
\item[\refstepcounter{equation}\text{(\theequation)}\label{b5}] $\epsilon $
is less than a small enough constant determined by $C_{B}$, $C_{w}$, $m$, $n$%
.
\end{itemize}

\begin{itemize}
\item[\refstepcounter{equation}\text{(\theequation)}\label{b6}] $x_{0}\in
5\left( Q_{0}\right) ^{+}\cap E$.
\end{itemize}

We write $c$, $C$, $C^{\prime }$, etc., to denote constants determined by $%
C_{B}$, $C_{W}$, $m$, $n$. These symbols may denote different constants in
different occurrences.

\begin{itemize}
\item[\refstepcounter{equation}\text{(\theequation)}\label{b7}] Let $z\in
E\cap \frac{65}{64}Q_{0}$.
\end{itemize}

Then \eqref{b6}, \eqref{b7} imply that

\begin{itemize}
\item[\refstepcounter{equation}\text{(\theequation)}\label{b8}] $\left\vert
z-x_{0}\right\vert \leq C\delta _{Q_{0}}=C\epsilon \cdot \left( \epsilon
^{-1}\delta _{Q_{0}}\right) $.
\end{itemize}

From \eqref{b1}, \eqref{b2}, \eqref{b3}, \eqref{b5}, \eqref{b8},
and Lemma \ref{lemma-transport} in Section \ref{transport-lemma} (with $\hat{\mathcal{A}} = \mathcal{A}$, $\hat{P}^0=P^0$), we obtain
a polynomial $\hat{P}^{\#}\in \mathcal{P}$ such that

\begin{itemize}
\item[\refstepcounter{equation}\text{(\theequation)}\label{b9}] $\vec{\Gamma}%
_{0}$ has an $\left( \mathcal{M},\epsilon ^{-1}\delta _{Q_{0}},C^{\prime
}\right) $-basis at $\left( z,M_{0},\hat{P}^{\#}\right) $, and
\end{itemize}

\begin{itemize}
\item[\refstepcounter{equation}\text{(\theequation)}\label{b10}] $\partial
^{\beta }\left( \hat{P}^{\#}-P^{0}\right) =0$ for $\beta \in \mathcal{M}$.
\end{itemize}

From \eqref{b9}, we have $\hat{P}^\# \in \Gamma_0(z,C^{\prime }M_0)$, while %
\eqref{b10} tells us that $\hat{P}^\# = P^0$. Thus,

\begin{itemize}
\item[\refstepcounter{equation}\text{(\theequation)}\label{b11}] $P^{0}\in
\Gamma _{0}\left( z,C^{\prime }M_{0}\right) $ for all $z\in \frac{65}{64}%
Q_{0}\cap E$.
\end{itemize}

Consequently, the function $F:=P^0$ on $\frac{65}{64}Q_0$ satisfies the
conclusions (C1), (C2) of the Main Lemma for $\mathcal{M}$.

This completes the proof of the Main Lemma for $\mathcal{M}$.   $ \blacksquare $

\section{Setup for the Induction Step}

\label{setup-for-the-induction-step}

Fix a monotonic set $\mathcal{A}$ strictly contained in $\mathcal{M}$, and
assume the following

\begin{itemize}
\item[\refstepcounter{equation}\text{(\theequation)}\label{s1}] \underline{%
Induction Hypothesis}: The Main Lemma for $\mathcal{A}^{\prime }$ holds for
all monotonic $\mathcal{A}^{\prime }< \mathcal{A}$.
\end{itemize}

Under this assumption, we will prove the Main Lemma for $\mathcal{A}$. Thus,
let $\vec{\Gamma}_{0}$, $\vec{\Gamma}_{l}$ $(l\geq 1)$, $C_{w}$, $\delta
_{\max }$, $Q_{0}$, $x_{0}$, $P^{0}$, $M_{0}$, $\epsilon $, $C_{B}$ be as in
the hypotheses of the Main Lemma for $\mathcal{A}$. Our goal is to prove the
existence of $F\in C^{m}(\frac{65}{64}Q_{0})$ satisfying conditions (C1) and
(C2). To do so, we introduce a constant $A\geq 1$, and make the following
additional assumptions.

\begin{itemize}
\item[\refstepcounter{equation}\text{(\theequation)}\label{s2}] \underline{%
Large $A$ assumption:} $A$ exceeds a large enough constant determined by $C_B
$, $C_w$, $m$, $n$.
\end{itemize}

\begin{itemize}
\item[\refstepcounter{equation}\text{(\theequation)}\label{s3}] \underline{%
Small $\epsilon$ assumption:} $\epsilon$ is less than a small enough
constant determined by $A$, $C_B$, $C_w$, $m$, $n$.
\end{itemize}

We write $c$, $C$, $C^{\prime }$, etc., to denote constants determined by $%
C_{B}$, $C_{w}$, $m$, $n$. Also we write $c(A)$, $C(A)$, $C^{\prime }(A)$,
etc., to denote constants determined by $A$, $C_{B}$, $C_{w}$, $m$, $n$.
Similarly, we write $C\left( \epsilon \right) $, $c\left( \epsilon \right) $%
, $C^{\prime }\left( \epsilon \right) $, etc., to denote constants
determined by $\epsilon $, $A$, $C_{B}$, $C_{w}$, $m$, $n$. These symbols
may denote different constants in different occurrences.

In place of (C1), (C2), we will prove the existence of a function $F\in
C^{m}\left( \frac{65}{64}Q_{0}\right) $ satisfying

\begin{itemize}
\item[(C*1)] $\left\vert \partial ^{\beta }\left( F-P^{0}\right) \right\vert
\leq C\left( \epsilon \right) M_{0}\delta _{Q_{0}}^{m-\left\vert \beta
\right\vert }$ on $\frac{65}{64}Q_{0}$ for $|\beta |\leq \mathcal{M}$; and

\item[(C*2)] $J_{z}\left( F\right) \in \Gamma _{0}\left( z,C\left( \epsilon
\right) M_{0}\right) $ for all $z\in E\cap \frac{65}{64}Q_{0}$.
\end{itemize}

Conditions (C*1), (C*2) differ from (C1), (C2) in that the constants in
(C*1), (C*2) may depend on $A$.

Once we establish (C*1) and (C*2), we may fix $A$ to be a constant
determined by $C_{B}$, $C_{w}$, $m$, $n$, large enough to satisfy the Large $%
A$ Assumption \eqref{s2}. The Small $\epsilon $ Assumption \eqref{s3} will
then follow from the Small $\epsilon $ Assumption (A3) in the Main Lemma for 
$\mathcal{A}$; and the desired conclusions (C1), (C2) will then follow from
(C*1), (C*2).

Thus, our goal is to prove the existence of $F\in C^{m}\left( \frac{65}{64}%
Q_{0}\right) $ satisfying (C*1) and (C*2), assuming \eqref{s1}, \eqref{s2}, %
\eqref{s3} above, along with hypotheses of the Main Lemma for $\mathcal{A}$.
This will complete our induction on $\mathcal{A}$ and establish the Main
Lemma for all monotonic subsets of $\mathcal{M}$.

\section{Calder\'on-Zygmund Decomposition}

\label{cz-decomposition}

We place ourselves in the setting of Section \ref%
{setup-for-the-induction-step}. Let $Q$ be a dyadic cube. We say that $Q$ is
``OK'' if \eqref{cz1} and \eqref{cz2} below are satisfied.

\begin{itemize}
\item[\refstepcounter{equation}\text{(\theequation)}\label{cz1}] $%
5Q\subseteq 5Q_{0}$.
\end{itemize}

\begin{itemize}
\item[\refstepcounter{equation}\text{(\theequation)}\label{cz2}] Either $%
\#(E\cap 5Q)\leq 1$ or there exists $\hat{A}<\mathcal{A}$ (strict
inequality) for which the following holds:
\end{itemize}

\begin{itemize}
\item[\refstepcounter{equation}\text{(\theequation)}\label{cz3}] For each $%
y\in E\cap 5Q$ there exists $\hat{P}^{y}\in \mathcal{P}$ satisfying

\begin{itemize}
\item[(3a)] $\vec{\Gamma}_{l\left( \mathcal{A}\right) -3}$ has a weak $%
\left( \hat{\mathcal{A}},\epsilon ^{-1}\delta _{Q},A\right) $-basis at $%
\left( y,M_{0},\hat{P}^{y}\right) $.

\item[(3b)] $\left\vert \partial ^{\beta }\left( \hat{P}^{y}-P^{0}\right)
\left( x_{0}\right) \right\vert \leq AM_{0}\left( \epsilon ^{-1}\delta
_{Q_{0}}\right) ^{m-\left\vert \beta \right\vert }$ for all $\beta \in 
\mathcal{M}$.

\item[(3c)] $\partial ^{\beta }\left( \hat{P}^{y}-P^{0}\right) \equiv 0$ for 
$\beta \in \mathcal{A}$.
\end{itemize}
\end{itemize}

\begin{remark}
The argument in this section and the next will depend sensitively on several
details of the above definition. Note that (3a) involves $\vec{\Gamma}_{l(%
\mathcal{A})-3}$ rather than $\vec{\Gamma}_{l(\hat{\mathcal{A}})}$, and that
(3b) involves $x_0$, $\delta_{Q_0}$ rather than $y$, $\delta_Q$. Note also
that the set $\hat{\mathcal{A}}$ in \eqref{cz2}, \eqref{cz3} needn't be
monotonic.
\end{remark}

A dyadic cube $Q$ will be called a \underline{Calder\'on-Zygmund cube} (or a 
\underline{CZ} cube) if it is OK, but no dyadic cube strictly containing $Q$
is OK.

Recall that given any two distinct dyadic cubes $Q$, $Q^{\prime }$, either $Q
$ is strictly contained in $Q^{\prime }$, or $Q^{\prime }$ is strictly
contained in $Q$, or $Q \cap Q^{\prime }=\emptyset$. The first two
alternatives here are ruled out if $Q$, $Q^{\prime }$ are CZ cubes. Hence,
the Calder\'on-Zygmund cubes are pairwise disjoint.

Any CZ cube $Q$ satisfies \eqref{cz1} and is therefore contained in the
interior of $5Q_0$. On the other hand, let $x$ be an interior point of $5Q_0$%
. Then any sufficiently small dyadic cube $Q$ containing $x$ satisfies $5Q
\subset 5Q_0$ and $\#(E \cap 5Q) \leq 1$; such $Q$ are OK. However, any
sufficiently large dyadic cube $Q$ containing $x$ will fail to satisfy $5Q
\subseteq 5Q_0$; such $Q$ are not OK. It follows that $x$ is contained in a
maximal OK dyadic cube. Thus, we have proven

\begin{lemma}
\label{lemma-cz1} The CZ cubes form a partition of the interior of $5Q_0$.
\end{lemma}

Next, we establish

\begin{lemma}
\label{lemma-cz2} Let $Q$, $Q^{\prime }$ be CZ cubes. If $\frac{65}{64}Q
\cap \frac{65}{64}Q^{\prime }\not= \emptyset$, then $\frac{1}{2}\delta_Q
\leq \delta_{Q^{\prime }}\leq 2\delta_Q$.
\end{lemma}

\begin{proof}
Suppose not. Without loss of generality, we may suppose that $\delta_Q \leq 
\frac{1}{4}\delta_{Q^{\prime }}$. Then $\delta_{Q^+} \leq \frac{1}{2}%
\delta_{Q^{\prime }}$, and $\frac{65}{64}Q^+ \cap \frac{65}{64}Q^{\prime
}\not= \emptyset$; hence, $5Q^+ \subset 5Q^{\prime }$. The cube $Q^{\prime }$
is OK. Therefore,

\begin{itemize}
\item[\refstepcounter{equation}\text{(\theequation)}\label{cz4}] $%
5Q^{+}\subset 5Q^{\prime }\subseteq 5Q_{0}$.
\end{itemize}

If $\#\left( E\cap 5Q^{\prime }\right) \leq 1$, then also $\#\left( E\cap
5Q^{+}\right) \leq 1$. Otherwise, there exists $\hat{\mathcal{A}}<\mathcal{A}
$ such that for each $y\in E\cap 5Q^{\prime }$ there exists $\hat{P}^{y}\in 
\mathcal{P}$ satisfying

\begin{itemize}
\item[\refstepcounter{equation}\text{(\theequation)}\label{cz5}] $\vec{\Gamma%
}_{l\left( \mathcal{A}\right) -3}$ has a weak $\left( \hat{\mathcal{A}}%
,\epsilon ^{-1}\delta _{Q^{\prime }},A\right) $-basis at $\left( y,M_{0},%
\hat{P}^{y}\right) $,
\end{itemize}

\begin{itemize}
\item[\refstepcounter{equation}\text{(\theequation)}\label{cz6}] $\left\vert
\partial ^{\beta }\left( \hat{P}^{y}-P^{0}\right) \left( x_{0}\right)
\right\vert \leq AM_{0}\left( \epsilon ^{-1}\delta _{Q_{0}}\right)
^{m-\left\vert \beta \right\vert }$ for $\beta \in \mathcal{M}$, and
\end{itemize}

\begin{itemize}
\item[\refstepcounter{equation}\text{(\theequation)}\label{cz7}] $\partial
^{\beta }\left( \hat{P}^{y}-P^{0}\right) \equiv 0$ for $\beta \in \mathcal{A}
$.
\end{itemize}

For each $y \in E \cap 5Q^+\subseteq E \cap 5Q^{\prime }$, the above $\hat{P}%
^y$ satisfies \eqref{cz6}, \eqref{cz7}; and \eqref{cz5} implies

\begin{itemize}
\item[\refstepcounter{equation}\text{(\theequation)}\label{cz8}] $\vec{\Gamma%
}_{l\left( \mathcal{A}\right) -3}$ has a weak $\left( \hat{\mathcal{A}}%
,\epsilon ^{-1}\delta _{Q^{+}},A\right) $-basis at $\left( y,M_{0},\hat{P}%
^{y}\right) $
\end{itemize}
because $\epsilon ^{-1}\delta _{Q^{+}}<\epsilon ^{-1}\delta _{Q^{\prime }}$,
and because \eqref{cz5}, \eqref{cz8} deal with weak bases. (See remarks in Section \ref{polynomial-bases}.) Thus, \eqref{cz4}
holds, and either $\#\left( E\cap 5Q^{+}\right) \leq 1$ or else our $\hat{%
\mathcal{A}}<\mathcal{A}$ and $\hat{P}^{y}$ $\left( y\in E\cap 5Q^{+}\right) 
$ satisfy \eqref{6}, \eqref{7}, \eqref{8}. This tells us that $Q^{+}$ is OK.
However, $Q^{+}$ strictly contains the CZ cube $Q$; therefore, $Q^{+}$
cannot be OK. This contradiction completes the proof of Lemma \ref{lemma-cz2}%
.
\end{proof}

Note that the proof of Lemma \ref{lemma-cz2} made use of our decision to
involve $x_0$, $\delta_{Q_0}$ rather than $y$, $\delta_Q$ in (3b), as well
as our decision to use weak bases in (3a).

We also have the following easy lemma. 

\begin{lemma}
\label{lemma-cz3} Only finitely many CZ cubes $Q$ satisfy the condition

\begin{itemize}
\item[\refstepcounter{equation}\text{(\theequation)}\label{cz9}] $\frac{65}{%
64}Q \cap \frac{65}{64} Q_0 \not= \emptyset$.
\end{itemize}
\end{lemma}

\section{Auxiliary Polynomials}

\label{auxiliary-polynomials}

We again place ourselves in the setting of Section \ref%
{setup-for-the-induction-step} and we make use of the Calder\'on-Zygmund
decomposition defined in Section \ref{cz-decomposition}.

Recall that $x_0 \in E \cap 5Q_0^+$, and that $\vec{\Gamma}_{l(\mathcal{A})}$
has an $(\mathcal{A}, \epsilon^{-1}\delta_{Q_0}, C_B)$-basis at $%
(x_0,M_0,P^0) $; moreover, $\mathcal{A} \subseteq \mathcal{M}$ is monotonic,
and $\epsilon$ is less than a small enough constant determined by $C_B$, $C_w
$, $m$, $n$.

Let $y\in E\cap 5Q_{0}$. Then $|x_{0}-y|\leq C\delta _{Q_{0}}=(C\epsilon
)(\epsilon ^{-1}\delta _{Q_{0}})$. Applying Lemma \ref{lemma-transport} in Section \ref{transport-lemma} with $\hat{\mathcal{A}} = \mathcal{A}$, $\hat{P}^0=P^0$, we see that there exists $P^{y}\in \mathcal{P}$ with
the following properties.

\begin{itemize}
\item[\refstepcounter{equation}\text{(\theequation)}\label{ap1}] $\vec{\Gamma%
}_{l\left( \mathcal{A}\right) -1}$ has an $\left( \mathcal{A},\epsilon
^{-1}\delta _{Q_{0}},C\right) $-basis $\left( P_{\alpha }^{y}\right)
_{\alpha \in \mathcal{A}}$ at $\left( y,M_{0},P^{y}\right) $,
\end{itemize}

\begin{itemize}
\item[\refstepcounter{equation}\text{(\theequation)}\label{ap2}] $\partial
^{\beta }\left( P^{y}-P^{0}\right) \equiv 0$ for $\beta \in \mathcal{A}$,
\end{itemize}

\begin{itemize}
\item[\refstepcounter{equation}\text{(\theequation)}\label{ap3}] $\left\vert
\partial ^{\beta }\left( P^{y}-P^{0}\right) \left( x_{0}\right) \right\vert
\leq CM_{0}\left( \epsilon ^{-1}\delta _{Q_{0}}\right) ^{m-\left\vert \beta
\right\vert }$ for $\beta \in \mathcal{M}$.
\end{itemize}

We fix $P^{y},P_{\alpha }^{y}$ $\left( \alpha \in \mathcal{A}\right) $ as
above for each $y\in E\cap 5Q_{0}$. We study the relationship between the
polynomials $P^{y},P_{\alpha }^{y}$ $(\alpha \in \mathcal{A})$ and the Calder%
\'{o}n-Zygmund decomposition.

\begin{lemma}[``Controlled Auxiliary Polynomials'']
\label{lemma-ap1} Let $Q \in$ CZ, and suppose that

\begin{itemize}
\item[\refstepcounter{equation}\text{(\theequation)}\label{ap4}] $\frac{65}{%
64}Q\cap \frac{65}{64}Q_{0}\not=\emptyset $.
\end{itemize}

Let

\begin{itemize}
\item[\refstepcounter{equation}\text{(\theequation)}\label{ap5}] $y\in E\cap
5Q_{0}\cap 5Q^{+}$.
\end{itemize}

Then

\begin{itemize}
\item[\refstepcounter{equation}\text{(\theequation)}\label{ap6}] $\left\vert
\partial ^{\beta }P_{\alpha }^{y}\left( y\right) \right\vert \leq C\cdot
\left( \epsilon ^{-1}\delta _{Q}\right) ^{\left\vert \alpha \right\vert
-\left\vert \beta \right\vert }$ for $\alpha \in \mathcal{A}$, $\beta \in 
\mathcal{M}$.
\end{itemize}
\end{lemma}

\begin{proof}

If $\mathcal{A} = \emptyset$, then \eqref{ap6} holds vacuously. Suppose $\mathcal{A} \not= \emptyset$. 

Let $K\geq 1$ be a large enough constant to be picked below, and assume that

\begin{itemize}
\item[\refstepcounter{equation}\text{(\theequation)}\label{ap7}] $%
\max_{\alpha \in \mathcal{A}\text{,}\beta \in \mathcal{M}}\left( \epsilon
^{-1}\delta _{Q}\right) ^{\left\vert \beta \right\vert -\left\vert \alpha
\right\vert }\left\vert \partial ^{\beta }P_{\alpha }^{y}\left( y\right)
\right\vert >K\text{.} $
\end{itemize}

We will derive a contradiction.

Thanks to \eqref{ap1}, we have

\begin{itemize}
\item[\refstepcounter{equation}\text{(\theequation)}\label{ap8}] $%
P^{y},P^{y}\pm CM_{0}\cdot \left( \epsilon ^{-1}\delta _{Q_{0}}\right)
^{m-\left\vert \alpha \right\vert }P_{\alpha }^{y}\in \Gamma _{l\left( 
\mathcal{A}\right) -1}\left( y,CM_{0}\right) $ for $\alpha \in \mathcal{A}$,
\end{itemize}

\begin{itemize}
\item[\refstepcounter{equation}\text{(\theequation)}\label{ap9}] $\partial
^{\beta }P_{\alpha }^{y}\left( y\right) =\delta _{\beta \alpha }$ for $\beta
,\alpha \in \mathcal{A}$,
\end{itemize}

and

\begin{itemize}
\item[\refstepcounter{equation}\text{(\theequation)}\label{ap10}] $%
\left\vert \partial ^{\beta }P_{\alpha }^{y}\left( y\right) \right\vert \leq
C\left( \epsilon ^{-1}\delta _{Q_{0}}\right) ^{\left\vert \alpha \right\vert
-\left\vert \beta \right\vert }$ for $\alpha \in \mathcal{A}$, $\beta \in 
\mathcal{M}$.
\end{itemize}

Also,

\begin{itemize}
\item[\refstepcounter{equation}\text{(\theequation)}\label{ap11}] $5Q\subset
5Q_{0}$ since $Q$ is OK.
\end{itemize}

If $\delta _{Q}\geq 2^{-12}\delta _{Q_{0}}$, then from \eqref{ap10}, %
\eqref{ap11}, we would have

\begin{itemize}
\item[\refstepcounter{equation}\text{(\theequation)}\label{ap14}] $%
\max_{\alpha \in \mathcal{A}\text{,}\beta \in \mathcal{M}}\left( \epsilon
^{-1}\delta _{Q}\right) ^{\left\vert \beta \right\vert -\left\vert \alpha
\right\vert }\left\vert \partial ^{\beta }P_{\alpha }^{y}\left( y\right)
\right\vert \leq C^{\prime }$.
\end{itemize}

We will pick

\begin{itemize}
\item[\refstepcounter{equation}\text{(\theequation)}\label{ap15}] $%
K>C^{\prime }$, with $C^{\prime }$ as in \eqref{ap14}.
\end{itemize}

Then \eqref{ap14} contradicts our assumption \eqref{ap7}.

Thus, we must have

\begin{itemize}
\item[\refstepcounter{equation}\text{(\theequation)}\label{ap16}] $\delta
_{Q}<2^{-12}\delta _{Q_{0}}$.
\end{itemize}

Let

\begin{itemize}
\item[\refstepcounter{equation}\text{(\theequation)}\label{ap17}] $Q=\hat{Q}%
_{0}\subset \hat{Q}_{1}\subset \cdots \subset \hat{Q}_{\nu _{\max }}$ be all
the dyadic cubes containing $Q$ and having sidelength at most $2^{-10}\delta
_{Q_{0}}$.
\end{itemize}

Then

\begin{itemize}
\item[\refstepcounter{equation}\text{(\theequation)}\label{ap18}] $\hat{Q}%
_{0}=Q$, $\delta _{\hat{Q}_{\nu _{\max }}}=2^{-10}\delta _{Q_{0}}$, $\hat{Q}%
_{\nu +1}=\left( \hat{Q}_{\nu }\right) ^{+}$ for $0\leq \nu \leq \nu _{\max
}-1$, and $\nu _{\max }\geq 2$.
\end{itemize}

For $1\leq \nu \leq \nu _{\max }$, we define

\begin{itemize}
\item[\refstepcounter{equation}\text{(\theequation)}\label{ap19}] $X_{\nu
}=\max_{\alpha \in \mathcal{A}\text{,}\beta \in \mathcal{M}}\left( \epsilon
^{-1}\delta _{\hat{Q}_{\nu }}\right) ^{\left\vert \beta \right\vert
-\left\vert \alpha \right\vert }\left\vert \partial ^{\beta }P_{\alpha
}^{y}\left( y\right) \right\vert $.
\end{itemize}

From \eqref{ap7} and \eqref{ap10}, we have

\begin{itemize}
\item[\refstepcounter{equation}\text{(\theequation)}\label{ap20}] $X_{0}>K$, 
$X_{\nu _{\max }}\leq C^{\prime }$,
\end{itemize}

and from \eqref{ap18}, \eqref{ap19}, we have

\begin{itemize}
\item[\refstepcounter{equation}\text{(\theequation)}\label{ap21}] $%
2^{-m}X_{\nu }\leq X_{\nu +1}\leq 2^{m}X_{\nu }$, for $0\leq \nu \leq \nu
_{\max }$.
\end{itemize}

We will pick

\begin{itemize}
\item[\refstepcounter{equation}\text{(\theequation)}\label{ap22}] $%
K>C^{\prime }$ with $C^{\prime }$ as in \eqref{ap20}.
\end{itemize}

Then $\tilde{\nu}:=\min \left\{ \nu :X_{\nu }\leq K\right\} $ and $\tilde{Q}=%
\hat{Q}_{\tilde{\nu}}$ satisfy the following, thanks to \eqref{ap20}, %
\eqref{ap21}, \eqref{ap22}: $\tilde{\nu}\not=0$, hence

\begin{itemize}
\item[\refstepcounter{equation}\text{(\theequation)}\label{ap23}] $\tilde{Q}$
is a dyadic cube strictly containing $Q$; also $2^{-m}K\leq X_{\tilde{\nu}%
}\leq K$,
\end{itemize}

hence

\begin{itemize}
\item[\refstepcounter{equation}\text{(\theequation)}\label{ap24}] $%
2^{-m}K\leq \max_{\alpha \in \mathcal{A}\text{,}\beta \in \mathcal{M}}\left(
\epsilon ^{-1}\delta _{\tilde{Q}}\right) ^{\left\vert \beta \right\vert
-\left\vert \alpha \right\vert }\left\vert \partial ^{\beta }P_{\alpha
}^{y}\left( y\right) \right\vert \leq K$.
\end{itemize}

Also, since $Q\subset \tilde{Q}$, we have $\frac{65}{64}\tilde{Q}\cap \frac{%
65}{64}Q_{0}\not=\emptyset $ by \eqref{ap4}; and since $\delta _{\tilde{Q}%
}\leq 2^{-10}\delta _{Q_{0}}$, we conclude that

\begin{itemize}
\item[\refstepcounter{equation}\text{(\theequation)}\label{ap25}] $5\tilde{Q}%
\subset 5Q_{0}$.
\end{itemize}

From \eqref{ap8}, \eqref{ap10}, and \eqref{ap25}, we have

\begin{itemize}
\item[\refstepcounter{equation}\text{(\theequation)}\label{ap26}] $%
P^{y},P^{y}\pm cM_{0}\left( \epsilon ^{-1}\delta _{\tilde{Q}}\right)
^{m-\left\vert \alpha \right\vert }P_{\alpha }^{y}\in \Gamma _{l\left( 
\mathcal{A}\right) -1}\left( y,CM_{0}\right) \subset \Gamma _{l\left( 
\mathcal{A}\right) -2}\left( y,CM_{0}\right) $ for $\alpha \in \mathcal{A}$;
\end{itemize}

and

\begin{itemize}
\item[\refstepcounter{equation}\text{(\theequation)}\label{ap27}] $%
\left\vert \partial ^{\beta }P_{\alpha }^{y}\left( y\right) \right\vert \leq
C\left( \epsilon ^{-1}\delta _{\tilde{Q}}\right) ^{\left\vert \alpha
\right\vert -\left\vert \beta \right\vert }$ for $\alpha \in \mathcal{A}$, $%
\beta \in \mathcal{M}$, $\beta \geq \alpha $.
\end{itemize}

Our results \eqref{ap9}, \eqref{ap26}, \eqref{ap27} tell us that

\begin{itemize}
\item[\refstepcounter{equation}\text{(\theequation)}\label{ap28}] $\left(
P_{\alpha }^{y}\right) _{\alpha \in \mathcal{A}}$ is a weak $\left( \mathcal{%
A},\epsilon ^{-1}\delta _{\tilde{Q}},C\right) $-basis for $\vec{\Gamma}%
_{l\left( \mathcal{A}\right) -2}$ at $\left( y,M_{0},P^{y}\right) $.
\end{itemize}

Note also that

\begin{itemize}
\item[\refstepcounter{equation}\text{(\theequation)}\label{ap29}] $\epsilon
^{-1}\delta _{\tilde{Q}}\leq \epsilon ^{-1}\delta _{Q_{0}}\leq \delta _{\max
}$, by \eqref{ap25} and hypothesis (A2) of the Main Lemma for $\mathcal{A}$.
\end{itemize}

Moreover,

\begin{itemize}
\item[\refstepcounter{equation}\text{(\theequation)}\label{ap30}] $\vec{%
\Gamma}_{l\left( \mathcal{A}\right) -2}$ is $\left( C,\delta _{\max }\right) 
$-convex, thanks to Lemma \ref{lemma-wsf4} (B).
\end{itemize}

If we take

\begin{itemize}
\item[\refstepcounter{equation}\text{(\theequation)}\label{ap31}] $K\geq
C^{\ast }$ for a large enough $C^{\ast }$,
\end{itemize}

then \eqref{ap24}, \eqref{ap28}$\cdots $\eqref{ap31} and the Relabeling
Lemma (Lemma \ref{lemma-pb2}) produce a monotonic set $\hat{\mathcal{A}}%
\subset \mathcal{M}$, such that

\begin{itemize}
\item[\refstepcounter{equation}\text{(\theequation)}\label{ap32}] $\hat{%
\mathcal{A}}<\mathcal{A}$ (strict inequality)
\end{itemize}

and

\begin{itemize}
\item[\refstepcounter{equation}\text{(\theequation)}\label{ap33}] $\vec{%
\Gamma}_{l\left( \mathcal{A}\right) -2}$ has an $\left( \hat{\mathcal{A}}%
,\epsilon ^{-1}\delta _{\tilde{Q}},C\right) $-basis at $\left(
y,M_{0},P^{y}\right) $.
\end{itemize}

Also, from \eqref{ap9}, \eqref{ap24}, \eqref{ap26}, we see that

\begin{itemize}
\item[\refstepcounter{equation}\text{(\theequation)}\label{ap34}] $\left(
P_{\alpha }^{y}\right) _{\alpha \in \mathcal{A}}$ is an $\left( \mathcal{A}%
,\epsilon ^{-1}\delta _{\tilde{Q}},CK\right) $-basis for $\vec{\Gamma}%
_{l\left( \mathcal{A}\right) -2}$ at $\left( y,M_{0},P^{y}\right) $.
\end{itemize}

We now pick

\begin{itemize}
\item[\refstepcounter{equation}\text{(\theequation)}\label{ap35}] $K=\hat{C}$
(a constant determined by $C_{B}$, $C_{w}$, $m$, $n$), with $\hat{C}\geq 1 $
large enough to satisfy \eqref{ap15}, \eqref{ap22}, \eqref{ap31}.
\end{itemize}

Then \eqref{ap33} and \eqref{ap34} tell us that

\begin{itemize}
\item[\refstepcounter{equation}\text{(\theequation)}\label{ap36}] $\vec{%
\Gamma}_{l\left( \mathcal{A}\right) -2}$ has both an $\left( \hat{\mathcal{A}%
},\epsilon ^{-1}\delta _{\tilde{Q}},C\right) $-basis and an $\left( \mathcal{%
A},\epsilon ^{-1}\delta _{\tilde{Q}},C\right) $-basis at $\left(
y,M_{0},P^{y}\right) $.
\end{itemize}

Let $z\in E\cap 5\tilde{Q}$. Then $z,y\in 5\tilde{Q}^{+}$ (see \eqref{ap5}), hence

\begin{itemize}
\item[\refstepcounter{equation}\text{(\theequation)}\label{ap37}] $%
\left\vert z-y\right\vert \leq C\delta _{\tilde{Q}}=C\epsilon \cdot \left(
\epsilon ^{-1}\delta _{\tilde{Q}}\right) $.
\end{itemize}

From \eqref{ap36}, \eqref{ap37}, the Small $\epsilon $ Assumption and Lemma %
\ref{lemma-transport} (and our hypothesis that $\mathcal{A}$ is monotonic;
see Section \ref{setup-for-the-induction-step}), we obtain a polynomial $%
\check{P}^{z}\in \mathcal{P}$, such that

\begin{itemize}
\item[\refstepcounter{equation}\text{(\theequation)}\label{ap38}] $\vec{\Gamma}
_{l\left( \mathcal{A}\right) -3}$ has an $\left( \hat{\mathcal{A}},\epsilon
^{-1}\delta _{\tilde{Q}},C\right) $-basis at $\left( z,M_{0},\check{P}%
^{z}\right) $,
\end{itemize}

\begin{itemize}
\item[\refstepcounter{equation}\text{(\theequation)}\label{ap39}] $\partial
^{\beta }\left( \check{P}^{z}-P^{y}\right) \equiv 0$ for $\beta \in \mathcal{%
A}$,
\end{itemize}

and

\begin{itemize}
\item[\refstepcounter{equation}\text{(\theequation)}\label{ap40}] $%
\left\vert \partial ^{\beta }\left( \check{P}^{z}-P^{y}\right) \left(
y\right) \right\vert \leq CM_{0}\left( \epsilon ^{-1}\delta _{\tilde{Q}%
}\right) ^{m-\left\vert \beta \right\vert }$ for $\beta \in \mathcal{M}$.
\end{itemize}

From \eqref{ap25} and \eqref{ap40}, we have

\begin{itemize}
\item[\refstepcounter{equation}\text{(\theequation)}\label{ap41}] $%
\left\vert \partial ^{\beta }\left( \check{P}^{z}-P^{y}\right) \left(
y\right) \right\vert \leq CM_{0}\left( \epsilon ^{-1}\delta _{Q_{0}}\right)
^{m-\left\vert \beta \right\vert }$ for $\beta \in \mathcal{M}$.
\end{itemize}

Since $y\in 5Q_{0}$ by hypothesis of Lemma \ref{lemma-ap1}, while $x_{0}\in
5Q_{0}^{+}$ by hypothesis of the Main Lemma for $\mathcal{A}$, we have $%
\left\vert x_{0}-y\right\vert \leq C\delta _{Q_{0}}$, and therefore %
\eqref{ap41} implies that

\begin{itemize}
\item[\refstepcounter{equation}\text{(\theequation)}\label{ap42}] $%
\left\vert \partial ^{\beta }\left( \check{P}^{z}-P^{y}\right) \left(
x_{0}\right) \right\vert \leq CM_{0}\left( \epsilon ^{-1}\delta
_{Q_{0}}\right) ^{m-\left\vert \beta \right\vert }$ for $\beta \in \mathcal{M%
}$.
\end{itemize}

From \eqref{ap2}, \eqref{ap3}, \eqref{ap39}, \eqref{ap42}, we now have

\begin{itemize}
\item[\refstepcounter{equation}\text{(\theequation)}\label{ap43}] $\partial
^{\beta }\left( \check{P}^{z}-P^{0}\right) \equiv 0$ for $\beta \in \mathcal{%
A}$
\end{itemize}

and

\begin{itemize}
\item[\refstepcounter{equation}\text{(\theequation)}\label{ap44}] $%
\left\vert \partial^{\beta }\left( \check{P}^{z}-P^{0}\right) \left(
x_{0}\right) \right\vert \leq CM_{0}\left( \epsilon ^{-1}\delta
_{Q_{0}}\right) ^{m-\left\vert \beta \right\vert }$ for $\beta \in \mathcal{M%
}$.
\end{itemize}

Our results \eqref{ap38}, \eqref{ap43}, \eqref{ap44} hold for every $z\in
E\cap 5\tilde{Q}$.

We recall the Large $A$ Assumption in the Section \ref%
{setup-for-the-induction-step}. Then \eqref{ap25}, \eqref{ap32}, \eqref{ap38}%
, \eqref{ap43}, \eqref{ap44} yield the following results: $5\tilde{Q}\subset
5Q_{0}$, $\hat{\mathcal{A}}<\mathcal{A}$ (strict inequality).

For every $z\in E\cap 5\tilde{Q}$, there exists $\check{P}^{z}\in \mathcal{P}
$ such that

\begin{itemize}
\item $\vec{\Gamma}_{l\left( \mathcal{A}\right) -3}$ has an $\left( \hat{%
\mathcal{A}},\epsilon ^{-1}\delta _{\tilde{Q}},A\right) $-basis at $\left(
z,M_{0},\check{P}^{z}\right) $.

\item $\partial ^{\beta }\left( \check{P}^{z}-P^{0}\right) \equiv 0$ for $%
\beta \in \mathcal{A}$.

\item $\left\vert \partial ^{\beta }\left( \check{P}^{z}-P^{0}\right) \left(
x_{0}\right) \right\vert \leq AM_{0}\left( \epsilon ^{-1}\delta
_{Q_{0}}\right) ^{m-\left\vert \beta \right\vert }$ for $\beta \in \mathcal{M%
}$.
\end{itemize}

Comparing the above results with the definition of an OK cube, we conclude
that $\tilde{Q}$ is OK.

However, since $\tilde{Q}$ properly contains the CZ cube $Q$, (see %
\eqref{ap23}), $\tilde{Q}$ cannot be OK.

This contradiction proves that our assumption \eqref{ap7} must be false.

Thus, $\left\vert \partial ^{\beta }P_{\alpha }^{y}\left( y\right)
\right\vert \leq K\left( \epsilon ^{-1}\delta _{Q}\right) ^{\left\vert
\alpha \right\vert -\left\vert \beta \right\vert }$ for $\alpha \in \mathcal{%
A}$, $\beta \in \mathcal{M}$.

Since we picked $K=\hat{C}$ in \eqref{ap35}, this implies the estimate %
\eqref{ap6}, completing the proof of Lemma \ref{lemma-ap1}.
\end{proof}

As an easy consequence of Lemma \ref{lemma-ap1}, we have the following result.

\begin{corollary}
\label{cor-to-lemma-ap1}

Let $Q\in $ CZ, and suppose $\frac{65}{64}Q\cap \frac{65}{64}%
Q_{0}\not=\emptyset $. Let $y\in E\cap 5Q_{0}\cap 5Q^{+}$. Then $\left(
P_{\alpha }^{y}\right) _{\alpha \in \mathcal{A}}$ is an $\left( \mathcal{A}%
,\epsilon ^{-1}\delta _{Q},C\right) $-basis for $\vec{\Gamma}_{l\left( 
\mathcal{A}\right) -1}$ at $\left( y,M_{0},P^{y}\right) $.
\end{corollary}

\begin{lemma}[\textquotedblleft Consistency of Auxiliary
Polynomials\textquotedblright ]
\label{lemma-ap2} Let $Q,Q^{\prime }\in $ CZ, with

\begin{itemize}
\item[\refstepcounter{equation}\text{(\theequation)}\label{ap49}] $\frac{65}{%
64}Q\cap \frac{65}{64}Q_{0}\not=\emptyset $, $\frac{65}{64}Q^{\prime }\cap 
\frac{65}{64}Q_{0}\not=\emptyset $
\end{itemize}

and

\begin{itemize}
\item[\refstepcounter{equation}\text{(\theequation)}\label{ap50}] $\frac{65}{%
64}Q\cap \frac{65}{64}Q^{\prime }\not=\emptyset $.
\end{itemize}

Let

\begin{itemize}
\item[\refstepcounter{equation}\text{(\theequation)}\label{ap51}] $y\in
E\cap 5Q_{0}\cap 5Q^{+}$, $y^{\prime }\in E\cap 5Q_{0}\cap 5\left( Q^{\prime
}\right) ^{+}$.
\end{itemize}

Then

\begin{itemize}
\item[\refstepcounter{equation}\text{(\theequation)}\label{ap52}] $%
\left\vert \partial ^{\beta }\left( P^{y}-P^{y^{\prime }}\right) \left(
y^{\prime }\right) \right\vert \leq CM_{0}\left( \epsilon ^{-1}\delta
_{Q}\right) ^{m-\left\vert \beta \right\vert }$ for $\beta \in \mathcal{M}$.
\end{itemize}
\end{lemma}

\begin{proof}
Suppose first that $\delta _{Q}\geq 2^{-20}\delta _{Q_{0}}$. Then \eqref{ap3}
(applied to $y$ and to $y^{\prime }$) tells us that 
\begin{equation*}
\left\vert \partial ^{\beta }\left( P^{y}-P^{y^{\prime }}\right) \left(
x_{0}\right) \right\vert \leq CM_{0}\left( \epsilon ^{-1}\delta
_{Q_{0}}\right) ^{m-\left\vert \beta \right\vert }\text{ for }\beta \in 
\mathcal{M}\text{.}
\end{equation*}%
Hence, $\left\vert \partial ^{\beta }\left( P^{y}-P^{y^{\prime }}\right)
\left( y^{\prime }\right) \right\vert \leq C^{\prime }M_{0}\left( \epsilon
^{-1}\delta _{Q_{0}}\right) ^{m-\left\vert \beta \right\vert }\leq C^{\prime
\prime }M_{0}\left( \epsilon ^{-1}\delta _{Q}\right) ^{m-\left\vert \beta
\right\vert }$ for $\beta \in \mathcal{M}$, since $x_{0}$, $y^{\prime }\in
5Q_{0}^{+}$. Thus, \eqref{ap52} holds if $\delta _{Q}\geq 2^{-20}\delta
_{Q_{0}}$. Suppose

\begin{itemize}
\item[\refstepcounter{equation}\text{(\theequation)}\label{ap53}] $\delta
_{Q}<2^{-20}\delta _{Q_{0}}$.
\end{itemize}

By \eqref{ap50} and Lemma \ref{lemma-cz2}, we have

\begin{itemize}
\item[\refstepcounter{equation}\text{(\theequation)}\label{ap54}] $\delta
_{Q},\delta _{Q^{\prime }}\leq 2^{-20}\delta _{Q_{0}}$ and $\frac{1}{2}%
\delta _{Q}\leq \delta _{Q^{\prime }}\leq 2\delta _{Q}$.
\end{itemize}

Together with \eqref{ap49}, this implies that

\begin{itemize}
\item[\refstepcounter{equation}\text{(\theequation)}\label{ap55}] $%
5Q^{+},5\left( Q^{\prime }\right) ^{+}\subseteq 5Q_{0}$.
\end{itemize}

From Corollary \ref{cor-to-lemma-ap1}, we have

\begin{itemize}
\item[\refstepcounter{equation}\text{(\theequation)}\label{ap56}] $\vec{%
\Gamma}_{l\left( \mathcal{A}\right) -1}$ has an $\left( \mathcal{A},\epsilon
^{-1}\delta _{Q^{\prime }},C\right) $-basis at $\left( y^{\prime
},M_{0},P^{y^{\prime }}\right) $.
\end{itemize}

From \eqref{ap50}, \eqref{ap51}, \eqref{ap54}, we have

\begin{itemize}
\item[\refstepcounter{equation}\text{(\theequation)}\label{ap57}] $%
\left\vert y-y^{\prime }\right\vert \leq C\delta _{Q^{\prime }}=C\epsilon
\left( \epsilon ^{-1}\delta _{Q^{\prime }}\right) $.
\end{itemize}

We recall from \eqref{ap54} and the hypotheses of the Main Lemma for $%
\mathcal{A}$ that

\begin{itemize}
\item[\refstepcounter{equation}\text{(\theequation)}\label{ap58}] $\epsilon
^{-1}\delta _{Q^{\prime }}\leq \epsilon ^{-1}\delta _{Q_{0}}\leq \delta
_{\max }$,
\end{itemize}

and we recall from Section \ref{setup-for-the-induction-step} that

\begin{itemize}
\item[\refstepcounter{equation}\text{(\theequation)}\label{ap59}] $\mathcal{A%
}$ is monotonic.
\end{itemize}

Thanks to \eqref{ap56}$\cdots $\eqref{ap59}, Lemma \ref{lemma-transport}
in Section \ref{transport-lemma} (with $\hat{\mathcal{A}} = \mathcal{A}$, $\hat{P}^0=P^0$) produces a polynomial $P^{\prime }\in 
\mathcal{P}$ such that

\begin{itemize}
\item[\refstepcounter{equation}\text{(\theequation)}\label{ap60}] $\vec{%
\Gamma}_{l\left( \mathcal{A}\right) -2}$ has an $\left( \mathcal{A},\epsilon
^{-1}\delta _{Q^{\prime }},C\right) $-basis at $\left( y,M_{0},P^{\prime
}\right) $;
\end{itemize}

\begin{itemize}
\item[\refstepcounter{equation}\text{(\theequation)}\label{ap61}] $\partial
^{\beta }\left( P^{\prime }-P^{y^{\prime }}\right) \equiv 0$ for $\beta \in 
\mathcal{A}$;
\end{itemize}

and

\begin{itemize}
\item[\refstepcounter{equation}\text{(\theequation)}\label{ap62}] $%
\left\vert \partial ^{\beta }\left( P^{\prime }-P^{y^{\prime }}\right)
\left( y^{\prime }\right) \right\vert \leq CM_{0}\left( \epsilon ^{-1}\delta
_{Q^{\prime }}\right) ^{m-\left\vert \beta \right\vert }$ for $\beta \in 
\mathcal{M}$.
\end{itemize}

From \eqref{ap60} we have in particular that

\begin{itemize}
\item[\refstepcounter{equation}\text{(\theequation)}\label{ap63}] $P^{\prime
}\in \Gamma _{l\left( \mathcal{A}\right) -2}\left( y,CM_{0}\right) $,
\end{itemize}

and from \eqref{ap62} and \eqref{ap54} we obtain

\begin{itemize}
\item[\refstepcounter{equation}\text{(\theequation)}\label{ap64}] $%
\left\vert \partial ^{\beta }\left( P^{y^{\prime }}-P^{\prime }\right)
\left( y^{\prime }\right) \right\vert \leq CM_{0}\left( \epsilon ^{-1}\delta
_{Q}\right) ^{m-\left\vert \beta \right\vert }$ for $\beta \in \mathcal{M}$.
\end{itemize}

If we knew that

\begin{itemize}
\item[\refstepcounter{equation}\text{(\theequation)}\label{ap65}] $%
\left\vert \partial ^{\beta }\left( P^{y}-P^{\prime }\right) \left( y\right)
\right\vert \leq M_{0}\left( \epsilon ^{-1}\delta _{Q}\right) ^{m-\left\vert
\beta \right\vert }$ for $\beta \in \mathcal{M}$,
\end{itemize}

then also $\left\vert \partial ^{\beta }\left( P^{y}-P^{\prime }\right)
\left( y^{\prime }\right) \right\vert \leq C^{\prime }M_{0}\left( \epsilon
^{-1}\delta _{Q}\right) ^{m-\left\vert \beta \right\vert }$ for $\beta \in 
\mathcal{M}$ since $\left\vert y-y^{\prime }\right\vert \leq C\delta _{Q}$
thanks to \eqref{ap50}, \eqref{ap51}, \eqref{ap54}. Consequently, by %
\eqref{ap64}, we would have $\left\vert \partial ^{\beta }\left(
P^{y^{\prime }}-P^{y}\right) \left( y^{\prime }\right) \right\vert \leq
CM_{0}\left( \epsilon ^{-1}\delta _{Q}\right) ^{m-\left\vert \beta
\right\vert }$ for $\beta \in \mathcal{M}$, which is our desired inequality %
\eqref{ap52}. Thus, Lemma \ref{lemma-ap2} will follow if we can prove %
\eqref{ap65}.

Suppose \eqref{ap65} fails.

Corollary \ref{cor-to-lemma-ap1} shows that $\vec{\Gamma}_{l\left( \mathcal{A%
}\right) -1}$ has an $\left( \mathcal{A},\epsilon ^{-1}\delta _{Q},C\right) $%
-basis at $\left( y,M_{0},P^{y}\right) $. Since $\Gamma _{l\left( \mathcal{A}%
\right) -1}\left( x,M\right) \subset \Gamma _{l\left( \mathcal{A}\right)
-2}\left( x,M\right) $ for all $x\in E$, $M>0$, it follows that

\begin{itemize}
\item[\refstepcounter{equation}\text{(\theequation)}\label{ap66}] $\vec{%
\Gamma}_{l\left( \mathcal{A}\right) -2}$ has an $\left( \mathcal{A},\epsilon
^{-1}\delta _{Q},C\right) $-basis at $\left( y,M_{0},P^{y}\right) $.
\end{itemize}

From \eqref{ap2} and \eqref{ap61} (applied to $y$ and $y^{\prime }$), we see
that

\begin{itemize}
\item[\refstepcounter{equation}\text{(\theequation)}\label{ap67}] $\partial
^{\beta }\left( P^{y}-P^{\prime }\right) \equiv 0$ for $\beta \in \mathcal{A}
$.
\end{itemize}

Since we are assuming that \eqref{ap65} fails, we have

\begin{itemize}
\item[\refstepcounter{equation}\text{(\theequation)}\label{ap68}] $%
\max_{\beta \in \mathcal{M}}\left( \epsilon ^{-1}\delta _{Q}\right)
^{\left\vert \beta \right\vert }\left\vert \partial ^{\beta }\left(
P^{y}-P^{\prime }\right) \left( y\right) \right\vert \geq M_{0}\left(
\epsilon ^{-1}\delta _{Q}\right) ^{m}$.
\end{itemize}

Also, from \eqref{ap54} and the hypotheses of the Main Lemma for $\mathcal{A}
$, we have

\begin{itemize}
\item[\refstepcounter{equation}\text{(\theequation)}\label{ap69}] $\epsilon
^{-1}\delta _{Q}<\epsilon ^{-1}\delta _{Q_{0}}\leq \delta _{\max }$.
\end{itemize}

From Lemma \ref{lemma-wsf4} (B), we know that

\begin{itemize}
\item[\refstepcounter{equation}\text{(\theequation)}\label{ap70}] $\vec{%
\Gamma}_{l\left( \mathcal{A}\right) -2}$ is $\left( C,\delta _{\max }\right) 
$-convex.
\end{itemize}

Our results \eqref{ap63}, \eqref{ap66}$\cdots$\eqref{ap70} and Lemma \ref%
{lemma-pb3} produce a set $\hat{\mathcal{A}}\subseteq \mathcal{M}$ and a
polynomial $\hat{P}\in \mathcal{P}$, with the following properties:

\begin{itemize}
\item[\refstepcounter{equation}\text{(\theequation)}\label{ap71}] $\hat{%
\mathcal{A}}$ is monotonic;
\end{itemize}

\begin{itemize}
\item[\refstepcounter{equation}\text{(\theequation)}\label{ap72}] $\hat{%
\mathcal{A}}<\mathcal{A}$ (strict inequality);
\end{itemize}

\begin{itemize}
\item[\refstepcounter{equation}\text{(\theequation)}\label{ap73}] $\vec{%
\Gamma}_{l\left( \mathcal{A}\right) -2}$ has an $\left( \hat{\mathcal{A}}%
,\epsilon ^{-1}\delta _{Q},C\right) $-basis at $\left( y,M_{0},\hat{P}%
\right) $;
\end{itemize}

\begin{itemize}
\item[\refstepcounter{equation}\text{(\theequation)}\label{ap74}] $\partial
^{\beta }\left( \hat{P}-P^{y}\right) \equiv 0$ for $\beta \in \mathcal{A}$
(recall, $\mathcal{A}$ is monotonic);
\end{itemize}

and

\begin{itemize}
\item[\refstepcounter{equation}\text{(\theequation)}\label{ap75}] $%
\left\vert \partial ^{\beta }\left( \hat{P}-P^{y}\right) \left( y\right)
\right\vert \leq CM_0\left( \epsilon ^{-1}\delta _{Q}\right) ^{m-\left\vert
\beta \right\vert }$ for $\beta \in \mathcal{M}$.
\end{itemize}

Now let $z\in E\cap 5Q^{+}$. We recall that $\mathcal{A}$ is monotonic, and
that \eqref{ap66}, \eqref{ap67}, \eqref{ap73}, \eqref{ap74}, \eqref{ap75} hold. Moreover,
since $y,z\in 5Q^{+}$, we have $\left\vert y-z\right\vert \leq C\delta
_{Q}=C\epsilon \left( \epsilon ^{-1}\delta _{Q}\right) $. Thanks to the
above remarks and the Small $\epsilon $ Assumption, we may apply Lemma \ref%
{lemma-transport} to produce $\check{P}^{z}\in \mathcal{P}$ satisfying the
following conditions.

\begin{itemize}
\item[\refstepcounter{equation}\text{(\theequation)}\label{ap76}] $\vec{%
\Gamma}_{l\left( \mathcal{A}\right) -3}$ has an $\left( \hat{\mathcal{A}}%
,\epsilon ^{-1}\delta _{Q},C\right) $-basis at $\left( z,M_{0},\check{P}%
^{z}\right) $.
\end{itemize}

\begin{itemize}
\item[\refstepcounter{equation}\text{(\theequation)}\label{ap77}] $\partial
^{\beta }\left( \check{P}^{z}-P^{y}\right) \equiv 0$ for $\beta \in \mathcal{%
A}$.
\end{itemize}

\begin{itemize}
\item[\refstepcounter{equation}\text{(\theequation)}\label{ap78}] $%
\left\vert \partial ^{\beta }\left( \check{P}^{z}-P^{y}\right) \left(
y\right) \right\vert \leq CM_{0}\left( \epsilon ^{-1}\delta _{Q}\right)
^{m-\left\vert \beta \right\vert }$ for $\beta \in \mathcal{M}$.
\end{itemize}

By \eqref{ap76} and the Large $A$ Assumption,

\begin{itemize}
\item[\refstepcounter{equation}\text{(\theequation)}\label{ap79}] $\vec{%
\Gamma}_{l\left( \mathcal{A}\right) -3}$ has an $\left( \hat{\mathcal{A}}%
,\epsilon ^{-1}\delta _{Q^{+}},A\right) $-basis at $\left( z,M_{0},\check{P}%
^{z}\right) $.
\end{itemize}

By \eqref{ap2} and \eqref{ap77}, we have

\begin{itemize}
\item[\refstepcounter{equation}\text{(\theequation)}\label{ap80}] $\partial
^{\beta }\left( \check{P}^{z}-P^{0}\right) \equiv 0$ for $\beta \in \mathcal{%
A}$.
\end{itemize}

By \eqref{ap54} and \eqref{ap78}, we have $\left\vert \partial ^{\beta
}\left( \check{P}^{z}-P^{y}\right) \left( y\right) \right\vert \leq
CM_{0}\left( \epsilon ^{-1}\delta _{Q_{0}}\right) ^{m-\left\vert \beta
\right\vert }$ for $\beta \in \mathcal{M}$, hence $\left\vert \partial
^{\beta }\left( \check{P}^{z}-P^{y}\right) \left( x_{0}\right) \right\vert
\leq CM_{0}\left( \epsilon ^{-1}\delta _{Q_{0}}\right) ^{m-\left\vert \beta
\right\vert }$ for $\beta \in \mathcal{M}$, since $x_0,y\in 5Q_{0}^{+}$.
Together with \eqref{ap3} and the Large $A$ Assumption, this yields the
estimate

\begin{itemize}
\item[\refstepcounter{equation}\text{(\theequation)}\label{ap81}] $%
\left\vert \partial ^{\beta }\left( \check{P}^{z}-P^{0}\right) \left(
x_{0}\right) \right\vert \leq AM_{0}\left( \epsilon ^{-1}\delta
_{Q_{0}}\right) ^{m-\left\vert \beta \right\vert }$ for $\beta \in \mathcal{M%
}$.
\end{itemize}

We have proven \eqref{ap79}, \eqref{ap80}, \eqref{ap81} for each $z\in E\cap
5Q^{+}$. Thus, $5Q^{+}\subset 5Q_{0}$ (see \eqref{ap55}), $\hat{\mathcal{A}}<%
\mathcal{A}$ (strict inequality; see \eqref{ap72}), and for each $z\in E\cap
5Q^{+}$ there exists $\check{P}^{z}\in \mathcal{P}$ satisfying \eqref{ap79}, \eqref{ap80}, \eqref{ap81}.

Comparing the above results with the definition of an OK cube, we see that $%
Q^{+}$ is OK. On the other hand $Q^{+}$ cannot be OK, since it properly
contains the CZ cube $Q$. Assuming that \eqref{ap65} fails, we have derived
a contradiction. Thus, \eqref{ap65} holds, completing the proof of Lemma \ref%
{lemma-ap2}.
\end{proof}

\section{Good News About CZ Cubes}

\label{gn}

In this section we again place ourselves in the setting of Section \ref%
{setup-for-the-induction-step}, and we make use of the auxiliary polynomials 
$P^{y}$ and the CZ cubes $Q$ defined above.

\begin{lemma}
\label{lemma-gn1} Let $Q \in$ CZ, with

\begin{itemize}
\item[\refstepcounter{equation}\text{(\theequation)}\label{gn1}] $\frac{65}{%
64}Q\cap \frac{65}{64}Q_{0}\not=\emptyset $
\end{itemize}

and

\begin{itemize}
\item[\refstepcounter{equation}\text{(\theequation)}\label{gn2}] $\#\left(
E\cap 5Q\right) \geq 2$.
\end{itemize}

Let

\begin{itemize}
\item[\refstepcounter{equation}\text{(\theequation)}\label{gn3}] $y\in E\cap
5Q$.
\end{itemize}

Then there exist a set $\mathcal{A}^\# \subseteq \mathcal{M}$ and a
polynomial $P^\# \in \mathcal{P}$ with the following properties.

\begin{itemize}
\item[\refstepcounter{equation}\text{(\theequation)}\label{gn4}] $\mathcal{A}%
^\#$ is monotonic.
\end{itemize}

\begin{itemize}
\item[\refstepcounter{equation}\text{(\theequation)}\label{gn5}] $\mathcal{A}%
^\# < \mathcal{A}$ (strict inequality).
\end{itemize}

\begin{itemize}
\item[\refstepcounter{equation}\text{(\theequation)}\label{gn6}] $\vec{\Gamma%
}_{l\left( \mathcal{A}\right) -3}$ has an $\left( \mathcal{A}^{\#},\epsilon
^{-1}\delta _{Q},C\left( A\right) \right) $-basis at $\left(
y,M_{0},P^{\#}\right) $.
\end{itemize}

\begin{itemize}
\item[\refstepcounter{equation}\text{(\theequation)}\label{gn7}] $\left\vert
\partial ^{\beta }\left( P^{\#}-P^{y}\right) \left( y\right) \right\vert
\leq C\left( A\right) M_{0}\left( \epsilon ^{-1}\delta _{Q}\right)
^{m-\left\vert \beta \right\vert }$ for $\beta \in \mathcal{M}$.
\end{itemize}
\end{lemma}

\begin{proof}
Recall that

\begin{itemize}
\item[\refstepcounter{equation}\text{(\theequation)}\label{gn8}] $\partial
^{\beta }\left( P^{y}-P^{0}\right) \equiv 0$ for $\beta \in \mathcal{A}$
(see \eqref{ap2} in Section \ref{auxiliary-polynomials})
\end{itemize}

and that

\begin{itemize}
\item[\refstepcounter{equation}\text{(\theequation)}\label{gn9}] $5Q
\subseteq 5Q_0$, since $Q$ is OK.
\end{itemize}

Thanks to \eqref{gn3} and \eqref{gn9}, Corollary \ref{cor-to-lemma-ap1} in
Section \ref{auxiliary-polynomials} applies, and it tells us that

\begin{itemize}
\item[\refstepcounter{equation}\text{(\theequation)}\label{gn10}] $\vec{%
\Gamma}_{l\left( \mathcal{A}\right) -1}$ has an $\left( \mathcal{A},\epsilon
^{-1}\delta _{Q},C\right) $-basis at $\left( y,M_{0},P^{y}\right) $.
\end{itemize}

On the other hand, $Q$ is OK and $\#(E\cap 5Q) \geq 2$; hence, there exist $%
\hat{\mathcal{A}} \subseteq \mathcal{M}$ and $\hat{P} \in \mathcal{P}$ with
the following properties

\begin{itemize}
\item[\refstepcounter{equation}\text{(\theequation)}\label{gn11}] $\vec{%
\Gamma}_{l\left( \mathcal{A}\right) -3}$ has a weak $\left( \hat{\mathcal{A}}%
,\epsilon ^{-1}\delta _{Q},A\right) $-basis at $\left( y,M_{0},\hat{P}%
\right) $.
\end{itemize}

\begin{itemize}
\item[\refstepcounter{equation}\text{(\theequation)}\label{gn12}] $%
\left\vert \partial ^{\beta }\left( \hat{P}-P^{0}\right) \left( x_{0}\right)
\right\vert \leq AM_{0}\left( \epsilon ^{-1}\delta _{Q_{0}}\right)
^{m-\left\vert \beta \right\vert }$ for $\beta \in \mathcal{M}$.
\end{itemize}

\begin{itemize}
\item[\refstepcounter{equation}\text{(\theequation)}\label{gn13}] $\partial
^{\beta }\left( \hat{P}-P^{0}\right) \equiv 0$ for $\beta \in \mathcal{A}$.
\end{itemize}

\begin{itemize}
\item[\refstepcounter{equation}\text{(\theequation)}\label{gn14}] $\hat{%
\mathcal{A}} < \mathcal{A}$ (strict inequality).
\end{itemize}

We consider separately two cases.

\underline{Case 1:} Suppose that

\begin{itemize}
\item[\refstepcounter{equation}\text{(\theequation)}\label{gn15}] $%
\left\vert \partial ^{\beta }\left( \hat{P}-P^{y}\right) \left( y\right)
\right\vert \leq M_{0}\left( \epsilon ^{-1}\delta _{Q}\right) ^{m-\left\vert
\beta \right\vert }$ for $\beta \in \mathcal{M}$.
\end{itemize}

By Lemma \ref{lemma-wsf4} (B),

\begin{itemize}
\item[\refstepcounter{equation}\text{(\theequation)}\label{gn16}] $\vec{%
\Gamma}_{l\left( \mathcal{A}\right) -3}$ is $\left( C,\delta _{\max }\right) 
$-convex.
\end{itemize}

Also, \eqref{gn9} and hypothesis (A2) of the Main Lemma for $\mathcal{A}$
give

\begin{itemize}
\item[\refstepcounter{equation}\text{(\theequation)}\label{gn17}] $\epsilon
^{-1}\delta _{Q}\leq \epsilon ^{-1}\delta _{Q_{0}}\leq \delta _{\max }$.
\end{itemize}

Applying \eqref{gn11}, \eqref{gn16}, \eqref{17} and Lemma \ref{lemma-pb2},
we obtain a set $\mathcal{A}^\# \subseteq \mathcal{M}$ such that

\begin{itemize}
\item[\refstepcounter{equation}\text{(\theequation)}\label{gn18}] $\mathcal{A%
}^{\#}\leq \hat{\mathcal{A}}$,
\end{itemize}

\begin{itemize}
\item[\refstepcounter{equation}\text{(\theequation)}\label{gn19}] $
\mathcal{A}^\#$ is monotonic,
\end{itemize}

and

\begin{itemize}
\item[\refstepcounter{equation}\text{(\theequation)}\label{gn20}] $\vec{%
\Gamma}_{l\left( \mathcal{A}\right) -3}$ has an $\left( \mathcal{A}%
^{\#},\epsilon ^{-1}\delta _{Q},C\left( A\right) \right) $-basis at $\left(
y,M_{0},\hat{P}\right) $.
\end{itemize}

Setting $P^\# =\hat{P}$, we obtain the desired conclusions \eqref{gn4}$\cdots
$\eqref{gn7} at once from \eqref{gn14}, \eqref{gn15}, \eqref{gn18}, %
\eqref{gn19}, and \eqref{gn20}.

Thus, Lemma \ref{lemma-gn1} holds in Case 1.

\underline{Case 2:} Suppose that $\left\vert \partial ^{\beta }\left( \hat{P}%
-P^{y}\right) \left( y\right) \right\vert >M_{0}\left( \epsilon ^{-1}\delta
_{Q}\right) ^{m-\left\vert \beta \right\vert }$ for some $\beta \in \mathcal{%
M}$, i.e.,

\begin{itemize}
\item[\refstepcounter{equation}\text{(\theequation)}\label{gn21}] $%
\max_{\beta \in \mathcal{M}}\left( \epsilon ^{-1}\delta _{Q}\right)
^{\left\vert \beta \right\vert }\left\vert \partial ^{\beta }\left( \hat{P}%
-P^{y}\right) \left( y\right) \right\vert >M_{0}\left( \epsilon ^{-1}\delta
_{Q}\right) ^{m}$.
\end{itemize}

From \eqref{gn11} we have

\begin{itemize}
\item[\refstepcounter{equation}\text{(\theequation)}\label{gn22}] $\hat{P}%
\in \Gamma _{l\left( \mathcal{A}\right) -3}\left( y,AM_{0}\right) $.
\end{itemize}

Since $\Gamma _{l(\mathcal{A})-1}(x,M)\subseteq \Gamma _{l\left( \mathcal{A}%
\right) -3}\left( x,M\right) $ for all $x\in E,M>0$, \eqref{gn10} implies
that

\begin{itemize}
\item[\refstepcounter{equation}\text{(\theequation)}\label{gn23}] $\vec{%
\Gamma}_{l\left( \mathcal{A}\right) -3}$ has an $\left( \mathcal{A},\epsilon
^{-1}\delta _{Q},C\right) $-basis at $\left( y,M_{0},P^{y}\right) $.
\end{itemize}

As in Case 1,

\begin{itemize}
\item[\refstepcounter{equation}\text{(\theequation)}\label{gn24}] $\vec{%
\Gamma}_{l\left( \mathcal{A}\right) -3}$ is $\left( C,\delta _{\max }\right) 
$-convex,
\end{itemize}

and

\begin{itemize}
\item[\refstepcounter{equation}\text{(\theequation)}\label{gn25}] $\epsilon
^{-1}\delta _{Q}\leq \epsilon ^{-1}\delta _{Q_{0}}\leq \delta _{\max }$.
\end{itemize}

From \eqref{gn8} and \eqref{gn13} we have

\begin{itemize}
\item[\refstepcounter{equation}\text{(\theequation)}\label{gn26}] $\partial
^{\beta }\left( \hat{P}-P^{y}\right) \equiv 0$ for $\beta \in \mathcal{A}$.
\end{itemize}

Thanks to \eqref{gn21}$\cdots$\eqref{gn26} and Lemma \ref{lemma-pb3} there exist 
$\mathcal{A}^\# \subseteq \mathcal{M}$ and $P^\# \in \mathcal{P}$ with the
following properties: $\mathcal{A}^\#$ is monotonic; $\mathcal{A}^\# <%
\mathcal{A}$ (strict inequality); $\vec{\Gamma}_{l(\mathcal{A})-3}$ has an $(%
\mathcal{A}^\#,\epsilon^{-1}\delta_{Q},C(A))$-basis at $(y,M_0,P^\#)$; $%
\partial^\beta(P^\#-P^y)\equiv 0$ for $\beta \in \mathcal{A}$; $%
|\partial^\beta (P^\# -P^y)(y)|\leq M_0 (\epsilon^{-1}\delta_Q)^{m-|\beta|}$
for $\beta \in \mathcal{M}$.

Thus, $\mathcal{A}^\#$ and $P^\#$ satisfy \eqref{gn4}$\cdots$\eqref{gn7},
proving Lemma \ref{lemma-gn1} in Case 2.

We have seen that Lemma \ref{lemma-gn1} holds in all cases.
\end{proof}

\begin{remarks}
\begin{itemize} \item The analysis of Case 2 in the proof of Lemma \ref{lemma-gn1} is a new
ingredient, with no analogue in our previous work on Whitney problems.

\item The proof of Lemma \ref{lemma-gn1} gives a $\hat{P}$ that satisfies
also $\partial^\beta(\hat{P}-P^0)\equiv 0$ for $\beta \in \mathcal{A}$, but
we make no use of that.

\item Note that $x_0$ and $\delta_{Q_0}$ appear in \eqref{gn12}, rather than
the desired $y, \delta_Q$. Consequently, \eqref{gn12} is of no help in the
proof of Lemma \ref{lemma-gn1}.
\end{itemize}
\end{remarks}

In the proof of our next result, we use our Induction Hypothesis that the
Main Lemma for $\mathcal{A}^{\prime }$ holds whenever $\mathcal{A}^{\prime }<%
\mathcal{A}$ and $\mathcal{A}^{\prime }$ is monotonic. (See Section \ref%
{setup-for-the-induction-step}.)

\begin{lemma}
\label{lemma-gn2}

Let $Q\in $ CZ. Suppose that

\begin{itemize}
\item[\refstepcounter{equation}\text{(\theequation)}\label{gn27}] $\frac{65}{%
64}Q\cap \frac{65}{64}Q_{0}\not=\emptyset $
\end{itemize}

and

\begin{itemize}
\item[\refstepcounter{equation}\text{(\theequation)}\label{gn28}] $\#\left(
E\cap 5Q\right) \geq 2$.
\end{itemize}

Let

\begin{itemize}
\item[\refstepcounter{equation}\text{(\theequation)}\label{gn29}] $y\in
E\cap 5Q$.
\end{itemize}

Then there exists $F^{y,Q} \in C^m(\frac{65}{64}Q)$ such that

\begin{itemize}
\item[(*1)] $\left\vert \partial ^{\beta }\left( F^{y,Q}-P^{y}\right)
\right\vert \leq C\left( \epsilon \right) M_{0}\delta _{Q}^{m-\left\vert
\beta \right\vert }$ on $\frac{65}{64}Q$, for $\left\vert \beta \right\vert
\leq m$; and

\item[(*2)] $J_{z}\left( F^{y,Q}\right) \in \Gamma _{0}\left( z,C\left(
\epsilon \right) M_{0}\right) $ for all $z\in E\cap \frac{65}{64}Q$.
\end{itemize}
\end{lemma}

\begin{proof}
Our hypotheses \eqref{gn27}, \eqref{gn28}, \eqref{gn29} are precisely the
hypotheses of Lemma \ref{lemma-gn1}. Let $\mathcal{A}^{\#}$, $P^{\#}$
satisfy the conclusions \eqref{gn4}$\cdots $\eqref{gn7} of that Lemma.
Recall the definition of $l(\mathcal{A})$; see \eqref{gn1}, \eqref{gn2} in
Section \ref{statement-of-the-main-lemma}. We have $l(\mathcal{A}^{\#})\leq
l(\mathcal{A})-3$ since $\mathcal{A}^{\#}<\mathcal{A}$; hence \eqref{gn6}
implies that

\begin{itemize}
\item[\refstepcounter{equation}\text{(\theequation)}\label{gn30}] $\vec{%
\Gamma}_{l\left( \mathcal{A}^{\#}\right) }$ has an $\left( \mathcal{A}%
^{\#},\epsilon ^{-1}\delta _{Q},C\left( A\right) \right) $-basis at $\left(
y,M_{0},P^{\#}\right) $.
\end{itemize}

Also, since $Q$ is OK, we have $5Q \subseteq 5Q_0$, hence $\delta_Q \leq
\delta_{Q_0}$. Hence, hypothesis (A2) of the Main Lemma for $\mathcal{A}$
implies that

\begin{itemize}
\item[\refstepcounter{equation}\text{(\theequation)}\label{gn31}] $\epsilon
^{-1}\delta _{Q}\leq \delta _{\max }$.
\end{itemize}

By \eqref{gn4}, \eqref{gn5}, and our Inductive Hypothesis, the Main Lemma
holds for $\mathcal{A}^{\#}$. Thanks to \eqref{gn29}, \eqref{gn30}, %
\eqref{gn31} and the Small $\epsilon $ Assumption in Section \ref%
{setup-for-the-induction-step}, the Main Lemma for $\mathcal{A}^{\#}$ now
yields a function $F\in C^{m}\left( \frac{65}{64}Q\right) $, such that

\begin{itemize}
\item[\refstepcounter{equation}\text{(\theequation)}\label{gn32}] $%
\left\vert \partial ^{\beta }\left( F-P^{\#}\right) \right\vert \leq C\left(
\epsilon \right) M_{0}\delta_{Q} ^{m-\left\vert \beta \right\vert }$ on $\frac{65%
}{64}Q$, for $\left\vert \beta \right\vert \leq m$; and
\end{itemize}

\begin{itemize}
\item[\refstepcounter{equation}\text{(\theequation)}\label{gn33}] $%
J_{z}\left( F\right) \in \Gamma _{0}\left( z,C\left( \epsilon \right)
M_{0}\right) $ for all $z\in E\cap \frac{65}{64}Q$.
\end{itemize}

Thanks to conclusion \eqref{gn7} of Lemma \ref{lemma-gn1} (together with %
\eqref{gn29}), we have also

\begin{itemize}
\item[\refstepcounter{equation}\text{(\theequation)}\label{gn34}] $%
\left\vert \partial ^{\beta }\left( P^{\#}-P^{y}\right) \right\vert \leq
C\left( \epsilon \right) M_{0}\delta _{Q}^{m-\left\vert \beta \right\vert }$
on $\frac{65}{64}Q$ for $\left\vert \beta \right\vert \leq m$.
\end{itemize}

(Recall that $P^\#-P^y$ is a polynomial of degree at most $m-1$.) Taking $%
F^{y,Q}=F$, we may read off the desired conclusions (*1) and (*2) from %
\eqref{gn32}, \eqref{gn33}, \eqref{gn34}.

The proof of Lemma \ref{lemma-gn2} is complete.
\end{proof}

\section{Local Interpolants}

In this section, we again place ourselves in the setting of Section \ref%
{setup-for-the-induction-step}. We make use of the Calder\'{o}n-Zygmund
cubes $Q$ and the auxiliary polynomials $P^{y}$ defined above. Let

\begin{itemize}
\item[\refstepcounter{equation}\text{(\theequation)}\label{li0}] $\mathcal{Q}%
=\left\{ Q\in CZ:\frac{65}{64}Q\cap \frac{65}{64}Q_{0}\not=\emptyset
\right\} $.
\end{itemize}

For each $Q\in \mathcal{Q}$, we define a function $F^{Q}\in C^{m}\left( 
\frac{65}{64}Q\right) $ and a polynomial $P^{Q}\in \mathcal{P}$. We proceed
by cases. We say that $Q \in \mathcal{Q}$ is

\begin{description}
\item[Type 1] if $\# (E \cap 5Q) \geq 2$,

\item[Type 2] if $\# (E \cap 5Q) = 1$,

\item[Type 3] if $\#(E\cap 5Q)=0$ and $\delta _{Q}\leq \frac{1}{1024}\delta
_{Q_{0}}$, and

\item[Type 4] if $\#(E\cap 5Q)=0$ and $\delta _{Q}>\frac{1}{1024}\delta
_{Q_{0}}$.
\end{description}

\underline{If $Q$ is of Type 1}, then we pick a point $y_{Q}\in E\cap 5Q$,
and set $P^{Q}=P^{y_{Q}}$. Applying Lemma \ref{lemma-gn2}, we obtain a
function $F^{Q}\in C^{m}\left( \frac{65}{64}Q\right) $ such that

\begin{itemize}
\item[\refstepcounter{equation}\text{(\theequation)}\label{li1}] $\left\vert
\partial ^{\beta }\left( F^{Q}-P^{Q}\right) \right\vert \leq C\left(
\epsilon \right) M_{0}\delta _{Q}^{m-\left\vert \beta \right\vert }$ on $%
\frac{65}{64}Q$, for $\left\vert \beta \right\vert \leq m$; and
\end{itemize}

\begin{itemize}
\item[\refstepcounter{equation}\text{(\theequation)}\label{li2}] $%
J_{z}\left( F^{Q}\right) \in \Gamma _{0}\left( z,C\left( \epsilon \right)
M_{0}\right) $ for all $z\in E\cap \frac{65}{64}Q$.
\end{itemize}

\underline{If $Q$ is of Type 2}, then we let $y_{Q}$ be the one and only
point of $E\cap 5Q$, and define $F^{Q}=P^{Q}=P^{y_{Q}}$. Then \eqref{li1}
holds trivially. If $y_{Q}\not\in \frac{65}{64}Q$ then \eqref{li2} holds
vacuously.

If $y_{Q}\in \frac{65}{64}Q$, then \eqref{li2} asserts that $P^{y_{Q}}\in
\Gamma _{0}\left( y_{Q},C\left( \epsilon \right) M_{0}\right) $. Thanks to %
\eqref{li1} in Section \ref{auxiliary-polynomials}, we know that $%
P^{y_{Q}}\in \Gamma _{l\left( \mathcal{A}\right) -1}\left(
y_{Q},CM_{0}\right) \subset \Gamma _{0}\left( y_{Q},C\left( \epsilon \right)
M_{0}\right) $. Thus, \eqref{li1} and \eqref{li2} hold also when $Q$ is of
Type 2.

\underline{If $Q$ is of Type 3}, then $5Q^{+}\subset 5Q_{0}$, since $\frac{65%
}{64}Q\cap \frac{65}{64}Q_{0}\not=\emptyset $ and $\delta _{Q}\leq \frac{1}{%
1024}\delta _{Q_{0}}$. However, $Q^{+}$ cannot be OK, since $Q$ is a CZ
cube. Therefore $\#\left( E\cap 5Q^{+}\right) \geq 2$. We pick $y_{Q}\in
E\cap 5Q^{+}$, and set $F^{Q}=P^{Q}=P^{y_{Q}}$. Then \eqref{li1} holds
trivially, and \eqref{li2} holds vacuously.

\underline{If $Q$ is of Type 4}, then we set $F^{Q}=P^{Q}=P^{0}$, and again %
\eqref{li1} holds trivially, and \eqref{li2} holds vacuously.

Note that if $Q$ is of Type 1, 2, or 3, then we have defined a point $y_{Q}$%
, and we have $P^Q=P^{y_Q}$ and 

\begin{itemize}
\item[\refstepcounter{equation}\text{(\theequation)}\label{li3}] $y_{Q}\in
E\cap 5Q^{+}\cap 5Q_{0}$.
\end{itemize}

(If $Q$ is of Type 1 or 2, then $y_{Q}\in E\cap 5Q$ and $5Q\subseteq 5Q_{0}$
since $Q$ is OK. If $Q$ is of Type 3, then $y_{Q}\in E\cap 5Q^{+}$ and $%
5Q^{+}\subset 5Q_{0}$). We have picked $F^{Q}$ and $P^{Q}$ for all $Q\in 
\mathcal{Q}$, and \eqref{li1}, \eqref{li2} hold in all cases.

\begin{lemma}[\textquotedblleft Consistency of the $P^{Q}$\textquotedblright 
]
\label{lemma-li1} Let $Q,Q^{\prime }\in \mathcal{Q}$, and suppose $\frac{65}{%
64}Q\cap \frac{65}{64}Q^{\prime }\not=\emptyset $. Then

\begin{itemize}
\item[\refstepcounter{equation}\text{(\theequation)}\label{li4}] $\left\vert
\partial ^{\beta }\left( P^{Q}-P^{Q^{\prime }}\right) \right\vert \leq
C\left( \epsilon \right) M_{0}\delta _{Q}^{m-\left\vert \beta \right\vert }$
on $\frac{65}{64}Q\cap \frac{65}{64}Q^{\prime }$, for $\left\vert \beta
\right\vert \leq m$.
\end{itemize}
\end{lemma}

\begin{proof}
Suppose first that neither $Q$ nor $Q^{\prime }$ is Type 4. Then $%
P^{Q}=P^{y_{Q}}$ and $P^{Q^{\prime }}=P^{y_{Q^{\prime }}}$ with $y_{Q}\in
E\cap 5Q^{+}\cap 5Q_{0}$, $y_{Q^{\prime }}\in E\cap 5\left( Q^{\prime
}\right) ^{+}\cap 5Q_{0}$. Thanks to Lemma \ref{lemma-ap2}, we have%
\begin{equation*}
\left\vert \partial ^{\beta }\left( P^{Q}-P^{Q^{\prime }}\right) \left(
y_{Q}\right) \right\vert \leq C\left( \epsilon \right) M_{0}\delta
_{Q}^{m-\left\vert \beta \right\vert }\text{ for }\beta \in \mathcal{M}\text{%
,}
\end{equation*}
which implies \eqref{li4}, since $y_{Q}\in 5Q^{+}$ and $P^{Q}-P^{Q^{\prime
}} $ is an $(m-1)^{rst}$ degree polynomial.

Next, suppose that $Q$ and $Q^{\prime }$ are both Type 4.

Then by definition $P^{Q}=P^{Q^{\prime }}=P^{0}$, and consequently %
\eqref{li4} holds trivially.

Finally, suppose that exactly one of $Q$, $Q^{\prime }$ is of Type 4.

Since $\delta _{Q}$ and $\delta _{Q^{\prime }}$, differ by at most a factor
of $2$, the cubes $Q$ and $Q^{\prime }$ may be interchanged without loss of
generality. Hence, we may assume that $Q^{\prime }$ is of Type 4 and $Q$ is
not. By definition of Type 4,

\begin{itemize}
\item[\refstepcounter{equation}\text{(\theequation)}\label{li5}] $\delta
_{Q^{\prime }}>\frac{1}{1024}\delta _{Q_{0}}$; hence also $\delta _{Q}\geq 
\frac{1}{1024}\delta _{Q_{0}}$,
\end{itemize}

since $\delta _{Q}$, $\delta _{Q^{\prime }}$, are powers of $2$ that differ
by at most a factor of $2$.

Since $Q^{\prime }$ is of Type 4 and $Q$ is not, we have $P^Q= P^{y_Q}$ and $%
P^{Q^{\prime }} = P^0$, with

\begin{itemize}
\item[\refstepcounter{equation}\text{(\theequation)}\label{li6}] $y_{Q}\in
E\cap 5Q^{+}\cap 5Q_{0}$.
\end{itemize}

Thus, in this case, \eqref{li4} asserts that

\begin{itemize}
\item[\refstepcounter{equation}\text{(\theequation)}\label{li7}] $\left\vert
\partial ^{\beta }\left( P^{y_{Q}}-P^{0}\right) \right\vert \leq C\left(
\epsilon \right) M_{0}\delta _{Q}^{m-\left\vert \beta \right\vert }$ on $%
\frac{65}{64}Q\cap \frac{65}{64}Q^{\prime }$, for $\left\vert \beta
\right\vert \leq m$.
\end{itemize}

However, by \eqref{li6} above, property \eqref{ap3} in Section \ref%
{auxiliary-polynomials} gives the estimate

\begin{itemize}
\item[\refstepcounter{equation}\text{(\theequation)}\label{li8}] $\left\vert
\partial ^{\beta }\left( P^{y_{Q}}-P^{0}\right) \left( x_{0}\right)
\right\vert \leq C\left( \epsilon \right) M_{0}\delta _{Q_{0}}^{m-\left\vert
\beta \right\vert }$ for $\left\vert \beta \right\vert \leq m-1$.
\end{itemize}

Recall from the hypotheses of the Main Lemma for $\mathcal{A}$ that $%
x_{0}\in 5\left( Q_{0}\right) ^{+}$. Since $P^{y_{Q}}-P^{0}$ is an $%
(m-1)^{rst}$ degree polynomial, we conclude from \eqref{li8} that

\begin{itemize}
\item[\refstepcounter{equation}\text{(\theequation)}\label{li9}] $\left\vert
\partial ^{\beta }\left( P^{y_{Q}}-P^{0}\right) \right\vert \leq C\left(
\epsilon \right) M_{0}\delta _{Q_{0}}^{m-\left\vert \beta \right\vert }$ on $%
5Q$, for $\left\vert \beta \right\vert \leq m$.
\end{itemize}

The desired inequality \eqref{li7} now follows from \eqref{li5} and %
\eqref{li9}. Thus, \eqref{li4} holds in all cases.

The proof of Lemma \ref{lemma-li1} is complete.
\end{proof}

From estimate \eqref{li1}, Lemma \ref{lemma-cz2} and Lemma \ref{lemma-li1},
we immediately obtain the following.

\begin{corollary}
\label{cor-to-lemma-li1} Let $Q,Q^{\prime }\in \mathcal{Q}$ and suppose that 
$\frac{65}{64}Q\cap \frac{65}{64}Q^{\prime }\not=\emptyset $. Then

\begin{itemize}
\item[\refstepcounter{equation}\text{(\theequation)}\label{li10}] $%
\left\vert \partial ^{\beta }\left( F^{Q}-F^{Q^{\prime }}\right) \right\vert
\leq C\left( \epsilon \right) M_{0}\delta _{Q}^{m-\left\vert \beta
\right\vert }$ on $\frac{65}{64}Q\cap \frac{65}{64}Q^{\prime }$, for $%
\left\vert \beta \right\vert \leq m$.
\end{itemize}
\end{corollary}

Regarding the polynomials $P^{Q}$, we make the following simple observation.

\begin{lemma}
\label{lemma-li2} We have

\begin{itemize}
\item[\refstepcounter{equation}\text{(\theequation)}\label{li11}] $%
\left\vert \partial ^{\beta }\left( P^{Q}-P^{0}\right) \right\vert \leq
C\left( \epsilon \right) M_{0}\delta _{Q_{0}}^{m-\left\vert \beta
\right\vert }$ on $\frac{65}{64}Q$, for $\left\vert \beta \right\vert \leq m$
and $Q\in \mathcal{Q}$.
\end{itemize}
\end{lemma}

\begin{proof}
Recall that if $Q$ is of Type 1, 2, or 3, then $P^{Q}=P^{y_{Q}}$ for some $%
y_{Q}\in E \cap 5Q_{0}$. From estimate \eqref{ap3} in Section \ref%
{auxiliary-polynomials}, we know that

\begin{itemize}
\item[\refstepcounter{equation}\text{(\theequation)}\label{li12}] $%
\left\vert \partial ^{\beta }\left( P^{Q}-P^{0}\right) \left( x_{0}\right)
\right\vert \leq C\left( \epsilon \right) M_{0}\delta _{Q_{0}}^{m-\left\vert
\beta \right\vert }$ for $\left\vert \beta \right\vert \leq m-1$.
\end{itemize}

Since $x_{0}\in 5Q_{0}^{+}$ (see the hypotheses of the Main Lemma for $%
\mathcal{A}$) and $P^{Q}-P^{0}$ is a polynomial of degree at most $m-1$, and
since $\frac{65}{64}Q\subset 5Q\subset 5Q_{0}$ (because $Q$ is OK), estimate %
\eqref{li12} implies the desired estimate \eqref{li11}.

If instead, $Q$ is of Type 4, then by definition $P^{Q}=P^{0}$, hence
estimate \eqref{li11} holds trivially.

Thus, \eqref{li11} holds in all cases.
\end{proof}

\begin{corollary}
\label{cor-to-lemma-li2} For $Q\in \mathcal{Q}$ and $|\beta |\leq m$, we
have $\left\vert \partial ^{\beta }\left( F^{Q}-P^{0}\right) \right\vert
\leq C\left( \epsilon \right) M_{0}\delta _{Q_{0}}^{m-\left\vert \beta
\right\vert }$ on $\frac{65}{64}Q$.
\end{corollary}
\begin{proof}
Recall that, since $Q$ is OK, we have $5Q\subset 5Q_{0}$. The desired
estimate therefore follows from estimates \eqref{li1} and \eqref{li11}.
\end{proof}

\section{Completing the Induction}

\label{completing-the-induction}

We again place ourselves in the setting of Section \ref%
{setup-for-the-induction-step}. We use the CZ cubes $Q$ and the functions $%
F^Q$ defined above. We recall several basic results from earlier sections.

\begin{itemize}
\item[\refstepcounter{equation}\text{(\theequation)}\label{ci1}] $\vec{\Gamma%
}_{0}$ is a $\left( C,\delta _{\max }\right) $-convex shape field.
\end{itemize}

\begin{itemize}
\item[\refstepcounter{equation}\text{(\theequation)}\label{ci2}] $\epsilon
^{-1}\delta _{Q_{0}}\leq \delta _{\max }$, hence $\epsilon ^{-1}\delta
_{Q}\leq \delta _{\max }$ for $Q\in $ CZ.
\end{itemize}

\begin{itemize}
\item[\refstepcounter{equation}\text{(\theequation)}\label{ci3}] The cubes $%
Q\in $ CZ partition the interior of $5Q_{0}$.
\end{itemize}

\begin{itemize}
\item[\refstepcounter{equation}\text{(\theequation)}\label{ci4}] For $%
Q,Q^{\prime }\in $ CZ, if $\frac{65}{64}Q\cap \frac{65}{64}Q^{\prime
}\not=\emptyset $, then $\frac{1}{2}\delta _{Q}\leq \delta _{Q^{\prime
}}\leq 2\delta _{Q}$.
\end{itemize}

Let

\begin{itemize}
\item[\refstepcounter{equation}\text{(\theequation)}\label{ci5}] $\mathcal{Q}%
=\left\{ Q\in CZ:\frac{65}{64}Q\cap \frac{65}{64}Q_{0}\not=\emptyset
\right\} $.
\end{itemize}

Then

\begin{itemize}
\item[\refstepcounter{equation}\text{(\theequation)}\label{ci6}] $\mathcal{Q}
$ is finite.
\end{itemize}

For each $Q \in \mathcal{Q}$, we have

\begin{itemize}
\item[\refstepcounter{equation}\text{(\theequation)}\label{ci7}] $F^{Q}\in
C^{m}\left( \frac{65}{64}Q\right) $,
\end{itemize}

\begin{itemize}
\item[\refstepcounter{equation}\text{(\theequation)}\label{ci8}] $%
J_{z}\left( F^{Q}\right) \in \Gamma _{0}\left( z,C\left( \epsilon \right)
M_{0}\right) $ for $z\in E\cap \frac{65}{64}Q$, and
\end{itemize}

\begin{itemize}
\item[\refstepcounter{equation}\text{(\theequation)}\label{ci9}] $\left\vert
\partial ^{\beta }\left( F^{Q}-P^{0}\right) \right\vert \leq C\left(
\epsilon \right) M_{0}\delta _{Q_{0}}^{m-\left\vert \beta \right\vert }$ on $%
\frac{65}{64}Q$, for $\left\vert \beta \right\vert \leq m$.
\end{itemize}

\begin{itemize}
\item[\refstepcounter{equation}\text{(\theequation)}\label{ci10}] For each $%
Q,Q^{\prime }\in \mathcal{Q}$, if $\frac{65}{64}Q\cap \frac{65}{64}Q^{\prime
}\not=\emptyset $, then $\left\vert \partial ^{\beta }\left(
F^{Q}-F^{Q^{\prime }}\right) \right\vert \leq C\left( \epsilon \right)
M_{0}\delta _{Q}^{m-\left\vert \beta \right\vert }$ on $\frac{65}{64}Q\cap 
\frac{65}{64}Q^{\prime }$, for $\left\vert \beta \right\vert \leq m$.
\end{itemize}

We introduce a Whitney partition of unity adapted to the cubes $Q\in $ CZ.
For each $Q\in $ CZ, let $\tilde{\theta}_{Q}\in C^{m}\left( \mathbb{R}%
^{n}\right) $ satisfy  $\tilde{\theta}_{Q}=1$ on $Q$, support$\left( \tilde{\theta}%
_{Q}\right) \subset \frac{65}{64}Q$, $\left\vert \partial ^{\beta }
\tilde{\theta}_{Q}\right\vert \leq C\delta _{Q}^{-\left\vert \beta
\right\vert }$ for $\left\vert \beta \right\vert \leq m$.

Set 
\begin{itemize}
\item[\refstepcounter{equation}\text{(\theequation)}\label{ci11}]$\theta _{Q}=\tilde{\theta}_{Q}\cdot \left( \sum_{Q^{\prime }\in
CZ}\left( \tilde{\theta}_{Q^{\prime }}\right) ^{2}\right) ^{-1/2}$. 
\end{itemize}
Note that we have \begin{itemize}
\item[\refstepcounter{equation}\text{(\theequation)}\label{ci14}] $%
\sum_{Q\in \mathcal{Q}}\theta _{Q}^{2}=1$ on $\frac{65}{64}Q_{0}$.
\end{itemize}
We define

\begin{itemize}
\item[\refstepcounter{equation}\text{(\theequation)}\label{ci15}] $%
F=\sum_{Q\in \mathcal{Q}}\theta _{Q}^{2}F^{Q}$.
\end{itemize}

For each $Q\in \mathcal{Q}$, \eqref{ci7}, \eqref{ci11} show
that $\theta _{Q}^{2}F^{Q}\in C^{m}\left( \mathbb{R}^{n}\right) $. Since
also $\mathcal{Q}$ is finite (see \eqref{ci6}), it follows that

\begin{itemize}
\item[\refstepcounter{equation}\text{(\theequation)}\label{ci16}] $F\in
C^{m}\left( \mathbb{R}^{n}\right) $.
\end{itemize}

Moreover, for any $x\in \frac{65}{64}Q_{0}$ and any $\beta $ of order $%
|\beta |\leq m$, we have

\begin{itemize}
\item[\refstepcounter{equation}\text{(\theequation)}\label{ci17}] $\partial
^{\beta }F\left( x\right) =\sum_{Q\in \mathcal{Q}\left( x\right) }\partial
^{\beta }\left\{ \theta _{Q}^{2}F^{Q}\right\} $, where
\end{itemize}

\begin{itemize}
\item[\refstepcounter{equation}\text{(\theequation)}\label{ci18}] $\mathcal{Q%
}\left( x\right) =\left\{ Q\in \mathcal{Q}:x\in \frac{65}{64}Q\right\} $.
\end{itemize}

Note that

\begin{itemize}
\item[\refstepcounter{equation}\text{(\theequation)}\label{ci19}] $\#\left( 
\mathcal{Q}\left( x\right) \right) \leq C$, by \eqref{ci4}.
\end{itemize}

Let $\hat{Q}$ be the CZ cube containing $x$. (There is one and only one such
cube, thanks to \eqref{ci3}; recall that we suppose that $x\in \frac{65}{64}%
Q_{0}$.) Then $\hat{Q}\in \mathcal{Q}(x)$, and \eqref{ci17} may be written
in the form

\begin{itemize}
\item[\refstepcounter{equation}\text{(\theequation)}\label{ci20}] $\partial
^{\beta }\left( F-P^{0}\right) \left( x\right) =\partial ^{\beta }\left( F^{%
\hat{Q}}-P^{0}\right) \left( x\right) +\sum_{Q\in \mathcal{Q}\left( x\right)
}\partial ^{\beta }\left\{ \theta _{Q}^{2}\cdot \left( F^{Q}-F^{\hat{Q}%
}\right) \right\} \left( x\right) $.
\end{itemize}

(Here we use \eqref{ci14}.) The first term on the right in \eqref{ci20} has
absolute value at most $C\left( \epsilon \right) M_{0}\delta
_{Q_{0}}^{m-\left\vert \beta \right\vert }$; see \eqref{ci9}. At most $C$
distinct cubes $Q$ enter into the second term on the right in \eqref{ci20};
see \eqref{ci19}. For each $Q\in \mathcal{Q}(x)$, we have 
\begin{equation*}
\left\vert \partial ^{\beta }\left\{ \theta _{Q}^{2}\cdot \left( F^{Q}-F^{%
\hat{Q}}\right) \right\} \left( x\right) \right\vert \leq C\left( \epsilon
\right) M_{0}\delta _{Q}^{m-\left\vert \beta \right\vert }\text{,}
\end{equation*}%
by \eqref{ci10} and \eqref{ci11}. Hence, for each $Q\in \mathcal{Q}(x)$, we
have 
\begin{equation*}
\left\vert \partial ^{\beta }\left\{ \theta _{Q}^{2}\cdot \left( F^{Q}-F^{%
\hat{Q}}\right) \right\} \left( x\right) \right\vert \leq C\left( \epsilon
\right) M_{0}\delta _{Q_{0}}^{m-\left\vert \beta \right\vert };
\end{equation*}%
see \eqref{ci3}.

The above remarks and \eqref{ci20} together yield the estimate

\begin{itemize}
\item[\refstepcounter{equation}\text{(\theequation)}\label{ci21}] $%
\left\vert \partial ^{\beta }\left( F-P^{0}\right) \right\vert \leq C\left(
\epsilon \right) M_{0}\delta _{Q_{0}}^{m-\left\vert \beta \right\vert }$ on $%
\frac{65}{64}Q_{0}$, for $\left\vert \beta \right\vert \leq m$.
\end{itemize}

Moreover, let $z\in E\cap \frac{65}{64}Q_{0}$. Then 
\begin{equation*}
J_{z}\left( F\right) =\sum_{Q\in \mathcal{Q}\left( z\right) }J_{z}\left(
\theta _{Q}\right) \odot _{z}J_{z}\left( \theta _{Q}\right) \odot
_{z}J_{z}\left( F^{Q}\right) \text{ (see \eqref{ci17});}
\end{equation*}%
\begin{equation*}
\left\vert \partial ^{\beta }\left[ J_{z}\left( \theta _{Q}\right) \right]
\left( z\right) \right\vert \leq C\delta _{Q}^{-\left\vert \beta \right\vert
}\text{ for }\left\vert \beta \right\vert \leq m-1\text{, }Q\in \mathcal{Q}%
\left( z\right) \text{ (see \eqref{ci11});}
\end{equation*}%
\begin{equation*}
\sum_{Q\in \mathcal{Q}\left( z\right) }\left[ J_{z}\left( \theta _{Q}\right) %
\right] \odot _{z}\left[ J_{z}\left( \theta _{Q}\right) \right] =1
\end{equation*}%
(see \eqref{ci14} and note that $J_{z}(\theta _{Q})=0$ for $Q\not\in 
\mathcal{Q}(z)$ by \eqref{ci11} and \eqref{ci18}); 
\begin{equation*}
J_{z}\left( F^{Q}\right) \in \Gamma _{0}\left( z,C\left( \epsilon \right)
M_{0}\right) \text{ for }Q\in \mathcal{Q}\left( z\right) \text{ (see %
\eqref{ci8});}
\end{equation*}%
\begin{equation*}
\left\vert \partial ^{\beta }\left\{ J_{z}\left( F^{Q}\right) -J_{z}\left(
F^{Q^{\prime }}\right) \right\} \left( z\right) \right\vert \leq C\left(
\epsilon \right) M_{0}\delta _{Q}^{m-\left\vert \beta \right\vert }
\end{equation*}%
for $\left\vert \beta \right\vert \leq m-1$, $Q,Q^{\prime }\in \mathcal{Q}%
\left( z\right) $ (see \eqref{ci10}); $\#\left( \mathcal{Q}\left( z\right) \right) \leq C$ (see \eqref{ci19}); $\delta _{Q}\leq \delta _{\max }$ (see \eqref{ci2});
$\vec{\Gamma}_{0}$ is a $\left( C,\delta _{\max }\right) $-convex shape
field (see \eqref{ci1}). 

The above results, together with Lemma \ref%
{lemma-wsf2}, tell us that

\begin{itemize}
\item[\refstepcounter{equation}\text{(\theequation)}\label{ci22}] $%
J_{z}\left( F\right) \in \Gamma _{0}\left( z,C\left( \epsilon \right)
M_{0}\right) $ for all $z\in E\cap \frac{65}{64}Q_{0}$.
\end{itemize}

From \eqref{ci16}, \eqref{ci21}, \eqref{ci22}, we see at once that the
restriction of $F$ to $\frac{65}{64}Q_{0}$ belongs to $C^{m}\left( \frac{65}{%
64}Q_{0}\right) $ and satisfies conditions (C*1) and (C*2) in Section \ref%
{setup-for-the-induction-step}. As we explained in that section, once we
have found a function in $C^{m}\left( \frac{65}{64}Q_{0}\right) $ satisfying
(C*1) and (C*2), our induction on $\mathcal{A}$ is complete. Thus, we have
proven the Main Lemma for all monotonic $\mathcal{A}\subseteq \mathcal{M}$. $\blacksquare$

\section{Restatement of the Main Lemma}

\label{rml}

An equivalent version of the Main Lemma for $\mathcal{A}=\emptyset $ reads as follows. 

\theoremstyle{plain} 
\newtheorem*{thm Restated Main Lemma}{Restated Main
Lemma}%
\begin{thm Restated Main Lemma}Let $\vec{\Gamma}_{0}=\left( \Gamma _{0}\left( x,M\right) \right) _{x\in
E,M>0}$ be a $\left( C_{w},\delta _{\max }\right) $-convex shape
field. For $l\geq 1$, let $\vec{\Gamma}_{l}=\left( \Gamma _{l}\left(
x,M\right) \right) _{x\in E,M>0}$ be the $l^{th}$-refinement of $\vec{\Gamma}_{0}$. Fix a dyadic cube $Q_0$ of sidelength $\delta _{Q_{0}}\leq \epsilon
\delta _{\max }$, where $\epsilon >0$ is a small enough constant determined
by $m$, $n$, $C_{w}$. Let $x_{0}\in E\cap 5Q_{0}^{+}$, and let $P_{0}\in \Gamma
_{l\left( \emptyset \right) }\left( x_{0},M_{0}\right) $.

Then there exists a function $F\in C^{m}\left( \frac{65}{64}Q_{0}\right) $,
satisfying 

\begin{itemize}
\item $\left\vert \partial ^{\beta }\left( F-P_{0}\right) \left( x\right)
\right\vert \leq C_{\ast }M_{0}\delta _{Q_{0}}^{m-\left\vert \beta
\right\vert }$ for $x\in \frac{65}{64}Q_{0}$, $\left\vert \beta \right\vert
\leq m$; and 
\item $J_{z}\left( F\right) \in \Gamma _{0}\left( z,C_{\ast }M_{0}\right) $
for all $z\in E\cap \frac{65}{64}Q_{0}$;
\end{itemize}
where $C_{\ast }$ is determined by $C_{w}$, $m$, $n$.\end{thm Restated Main Lemma}

\section{Tidying Up}

\label{tu}

In this section, we remove from the Restated Main Lemma the small constant $%
\epsilon$ and the assumption that $Q_0$ is dyadic.

\begin{theorem}
\label{theorem-tu1} Let $\vec{\Gamma}_{0}=\left( \Gamma _{0}\left(
x,M\right) \right) _{x\in E,M>0}$ be a $\left( C_{w},\delta _{\max }\right) $%
-convex shape field. For $l\geq 1$, let $\vec{\Gamma}_{l}=\left( \Gamma
_{l}\left( x,M\right) \right) _{x\in E,M>0}$ be the $l^{th}$-refinement of $%
\vec{\Gamma}_{0}$. Fix a cube $Q_0$ of sidelength $\delta _{Q_{0}}\leq
\delta _{\max }$, a point $x_0 \in E \cap 5Q_0$, and a real number $M_0>0$.
Let $P_{0}\in \Gamma _{l\left( \emptyset \right) +1}\left(
x_{0},M_{0}\right) $.

Then there exists a function $F\in C^{m}\left(Q_{0}\right) $ satisfying the
following, with $C_{\ast }$ determined by $C_{w}$, $m$, $n$.

\begin{itemize}
\item $\left\vert \partial ^{\beta }\left( F-P_{0}\right) \left( x\right)
\right\vert \leq C_{\ast }M_{0}\delta _{Q_{0}}^{m-\left\vert \beta
\right\vert }$ for $x\in Q_{0}$, $\left\vert \beta \right\vert \leq m$; and

\item $J_{z}\left( F\right) \in \Gamma _{0}\left( z,C_{\ast }M_{0}\right) $
for all $z\in E\cap Q_{0}$.
\end{itemize}
\end{theorem}

\begin{proof}[Sketch of Proof]
Let $\epsilon >0$ be the small constant in the statement of the Restated
Main Lemma in Section \ref{rml}. In particular, $\epsilon $ is determined by 
$C_{w}$, $m$, $n$. We write $c$, $C$, $C^{\prime }$, etc., to denote
constants determined by $C_{w}$, $m$, $n$. These symbols may denote
different constants in different occurrences.

We cover $CQ_{0}$ by a grid of dyadic cubes $\{Q_{\nu }\}$, all having the same
sidelength $\delta _{Q_{\nu }}$, with $\frac{\epsilon }{20}\delta
_{Q_{0}}\leq \delta _{Q_{\nu }}\leq \epsilon \delta _{Q_{0}}$, and all
contained in $C^{\prime }Q_{0}$. 

For each $Q_{\nu }$ with $E\cap \frac{65}{64}Q_{\nu }\not=\emptyset $, we
pick a point $x_{\nu }\in E\cap \frac{65}{64}Q_{\nu }$; by definition of the 
$l^{th}$-refinement, there exists $P_{\nu }\in \Gamma _{l(\emptyset
)}(x_{\nu },M_{0})$ such that $\left\vert \partial ^{\beta }\left( P_{\nu
}-P_{0}\right) \left( x_{0}\right) \right\vert \leq CM_{0}\delta
_{Q_{0}}^{m-\left\vert \beta \right\vert }$ for $\beta \in \mathcal{M}$.

Since $x_{\nu }\in E\cap \frac{65}{64}Q_{\nu }$, $P_{\nu }\in \Gamma
_{l\left( \emptyset \right) }$, and $\delta _{Q_{\nu }}\leq \epsilon \delta
_{Q_{0}}\leq \epsilon \delta _{\max }$, the Restated Main Lemma applies to $%
x_{\nu },P_{\nu },Q_{\nu }$ to produce $F_{\nu }\in C^{m}\left( \frac{65}{64}%
Q_{\nu }\right) $ satisfying

\begin{itemize}
\item[\refstepcounter{equation}\text{(\theequation)}\label{tu2}] $\left\vert
\partial ^{\beta }\left( F_{\nu }-P_{\nu }\right) \left( x\right)
\right\vert \leq CM_{0}\delta _{Q_{\nu }}^{m-\left\vert \beta \right\vert
}\leq CM_{0}\delta _{Q_{0}}^{m-\left\vert \beta \right\vert }$ for $x\in 
\frac{65}{64}Q_{\nu }$, $\left\vert \beta \right\vert \leq m$;
\end{itemize}

and

\begin{itemize}
\item[\refstepcounter{equation}\text{(\theequation)}\label{tu3}] $%
J_{z}\left( F_{\nu }\right) \in \Gamma _{0}\left( z,CM_{0}\right) $ for all $%
z\in E\cap \frac{65}{64}Q_{\nu }$.
\end{itemize}

We have produced such $F_{\nu }$ for those $\nu $ satisfying $E\cap \frac{65%
}{64}Q_{\nu }\not=\emptyset $. If instead $E\cap \frac{65}{64}Q_{\nu
}=\emptyset $, then we set $F_{\nu }=P_{0}$. 

Next, we introduce a partition of unity. We fix cutoff functions $\theta
_{\nu }\in C^{m}\left( \mathbb{R}^{n}\right) $ satisfying

\begin{itemize}
\item[\refstepcounter{equation}\text{(\theequation)}\label{tu6}] support $%
\theta _{\nu }\subset \frac{65}{64}Q_{\nu }$, $\left\vert \partial ^{\beta
}\theta _{\nu }\right\vert \leq C\delta _{Q_{0}}^{-\left\vert \beta
\right\vert }$ for $\left\vert \beta \right\vert \leq m$, $\sum_{\nu }\theta
_{\nu }^{2}=1$ on $Q_{0}$.
\end{itemize}

We then define

\begin{itemize}
\item[\refstepcounter{equation}\text{(\theequation)}\label{tu7}] $%
F=\sum_{\nu }\theta _{\nu }^{2}F_{\nu }$ on $Q_{0}$.
\end{itemize}

One checks easily that $F$ satisfies the conclusions of
Theorem \ref{theorem-tu1}.
\end{proof}

\part{Applications}

\section{Finiteness Principle I} \label{fp-i}

In this section we prove a finiteness principle for shape fields.

Let $\vec{\Gamma}_{0}=\left( \Gamma _{0}\left( x,M\right) \right) _{x\in
E,M>0}$ be a shape field. For $l\geq 1$, let $\vec{\Gamma}_{l}=\left( \Gamma
_{l}\left( x,M\right) \right) _{x\in E,M>0}$ be the $l^{th}$-refinement of $%
\vec{\Gamma}_{0}$. Fix $M_{0}>0$. For $x\in E$, $S\subset E$, define

\begin{itemize}
\item[\refstepcounter{equation}\text{(\theequation)}\label{fp1}] $\Gamma
\left( x,S\right) =\left\{ 
\begin{array}{c}
P^{x}:\vec{P}=\left( P^{y}\right) _{y\in S\cup \left\{ x\right\} }\in
Wh\left( S\cup \left\{ x\right\} \right) \text{, }\left\Vert \vec{P}%
\right\Vert _{\dot{C}^{m}\left( S\cup \left\{ x\right\} \right) }\leq M_{0},
\\ 
P^{y}\in \Gamma _{0}\left( y,M_{0}\right) \text{ for all }y\in S\cup \left\{
x\right\} \text{.}%
\end{array}%
\right\} $
\end{itemize}

(See Section \ref{notation-and-preliminaries} for the definition of $%
Wh(\cdot )$ and $||\cdot ||_{\dot{C}^{m}\left( \cdot \right)}$.) Note that

\begin{itemize}
\item[\refstepcounter{equation}\text{(\theequation)}\label{fp2}] $\Gamma
\left( x,\emptyset \right) =\Gamma _{0}\left( x,M_{0}\right) $.
\end{itemize}

Define

\begin{itemize}
\item[\refstepcounter{equation}\text{(\theequation)}\label{fp3}] $\Gamma
_{l}^{fp}\left( x,M_{0}\right) =\bigcap_{S\subset E,\#\left( S\right) \leq
\left( D+2\right) ^{l}}\Gamma \left( x,S\right) $ for $l\geq 0$, where
\end{itemize}

\begin{itemize}
\item[\refstepcounter{equation}\text{(\theequation)}\label{fp4}] $D=\dim 
\mathcal{P}$.
\end{itemize}

Note that

\begin{itemize}
\item[\refstepcounter{equation}\text{(\theequation)}\label{fp5}] $\Gamma
_{0}^{fp}\left( x,M_{0}\right) \subseteq \Gamma _{0}\left( x,M_{0}\right) $,
thanks to \eqref{fp2}.
\end{itemize}

Each $\Gamma (x,S)$, $\Gamma _{l}^{fp}\left( x,M_{0}\right) $ is a (possibly
empty) convex subset of $\mathcal{P}$.

As a consequence of Helly's theorem, we have the following result.
\begin{lemma}
\label{lemma-fp1} Let $x\in E$, $l\geq 0$. Suppose $\Gamma
(x,S)\not=\emptyset $ for all $S\subset E$ with $\#\left( S\right) \leq
\left( D+2\right) ^{l+1}$. Then $\Gamma _{l}^{fp}\left( x,M_{0}\right)
\not=\emptyset $.
\end{lemma}


\begin{lemma}
\label{lemma-fp2} For $x \in E, l \geq 0$, we have

\begin{itemize}
\item[\refstepcounter{equation}\text{(\theequation)}\label{fp6}] $\Gamma
_{l}^{fp}\left( x,M_{0}\right) \subseteq \Gamma _{l}\left( x,M_{0}\right) $.
\end{itemize}
\end{lemma}

\begin{proof}
We use induction on $l$. The base case $l=0$ is our observation \eqref{fp5}.
For the induction step, fix $l\geq 1$. We will prove \eqref{6} under the
inductive assumption

\begin{itemize}
\item[\refstepcounter{equation}\text{(\theequation)}\label{fp7}] $\Gamma
_{l-1}^{fp}\left( y,M_{0}\right) \subseteq \Gamma _{l-1}\left(
y,M_{0}\right) $ for all $y\in E$.
\end{itemize}

Thus, let $P\in \Gamma _{l}^{fp}\left( x,M_{0}\right) $ be given. We must
prove that $P\in \Gamma _{l}\left( x,M_{0}\right) $, which means that given $%
y\in E$ there exists

\begin{itemize}
\item[\refstepcounter{equation}\text{(\theequation)}\label{fp8}] $P^{\prime
}\in \Gamma _{l-1}\left( y,M_{0}\right) $ such that $\left\vert \partial
^{\beta }\left( P-P^{\prime }\right) \left( x\right) \right\vert \leq
M_{0}\left\vert x-y\right\vert ^{m-\left\vert \beta \right\vert }$ for $%
\left\vert \beta \right\vert \leq m-1$.
\end{itemize}

We will prove that there exists

\begin{itemize}
\item[\refstepcounter{equation}\text{(\theequation)}\label{fp9}] $P^{\prime
}\in \Gamma _{l-1}^{fp}\left( y,M_{0}\right) $ such that $\left\vert
\partial ^{\beta }\left( P-P^{\prime }\right) \left( x\right) \right\vert
\leq M_{0}\left\vert x-y\right\vert ^{m-\left\vert \beta \right\vert }$ for $%
\left\vert \beta \right\vert \leq m-1$.
\end{itemize}

Thanks to our inductive hypothesis \eqref{fp7}, we see that \eqref{fp9}
implies \eqref{fp8}. Therefore, to complete the proof of the Lemma, it is
enough to prove the existence of a $P^{\prime }$ satisfying \eqref{fp9}. For 
$S\subset E$, define 
\begin{equation*}
\hat{\Gamma}\left( S\right) =\left\{ 
\begin{array}{c}
P^{y}:\vec{P}=\left( P^{z}\right) _{z\in S\cup \left\{ x,y\right\} }\in
Wh\left( S\cup \left\{ x,y\right\} \right) \text{, }P^{x}=P\text{,}%
\left\Vert \vec{P}\right\Vert _{\dot{C}^{m}\left( S\cup \left\{ x,y\right\}
\right) }\leq M_{0}, \\ 
P^{z}\in \Gamma _{0}\left( z,M_{0}\right) \text{ for all }z\in S\cup \left\{
x,y\right\} \text{.}%
\end{array}%
\right\}
\end{equation*}%
By definition,

\begin{itemize}
\item[\refstepcounter{equation}\text{(\theequation)}\label{fp10}] $\hat{%
\Gamma}\left( S\right) \subset \Gamma \left( y,S\right) $ for $S\subset E$.
\end{itemize}

Let $S_{1},\cdots ,S_{D+1}\subset E$ with $\#\left( S_{i}\right) \leq \left(
D+2\right) ^{l-1}$ for each $i$.

Then $\hat{\Gamma}\left( S_{1}\right) \cap \cdots \cap \hat{\Gamma}\left(
S_{D+1}\right) \supset \hat{\Gamma}\left( S_{1}\cup \cdots \cup
S_{D+1}\right) $, and $\#\left( S_{1}\cup \cdots \cup S_{D+1}\cup \left\{
y\right\} \right) \leq \left( D+1\right) \left( D+2\right) ^{l-1}+1\leq
\left( D+2\right) ^{l}$. Since $P\in \Gamma _{l}^{fp}\left( x,M_{0}\right) $%
, it follows that there exists $\vec{P}=\left( P^{z}\right) _{z\in S_{1}\cup
\cdots \cup S_{D+1}\cup \left\{ x,y\right\} }\in Wh\left( S_{1}\cup \cdots
\cup S_{D+1}\cup \left\{ x,y\right\} \right) $ such that $P^{x}=P$, $%
\left\Vert \vec{P}\right\Vert _{\dot{C}^{m}\left( S_{1}\cup \cdots \cup
S_{D+1}\cup \left\{ x,y\right\} \right) }\leq M_{0}$, $P^{z}\in \Gamma
_{0}\left( z,M_{0}\right) $ for all $z\in S_{1}\cup \cdots \cup S_{D+1}\cup
\left\{ x,y\right\} $. We then have $P^{y}\in \hat{\Gamma}\left( S_{1}\cup
\cdots \cup S_{D+1}\right) $, hence $\hat{\Gamma}\left( S_{1}\right) \cap
\cdots \cap \hat{\Gamma}\left( S_{D+1}\right) \supset \hat{\Gamma}\left(
S_{1}\cup \cdots \cup S_{D+1}\right) \not=\emptyset $.

By Helly's theorem, there exists

\begin{itemize}
\item[\refstepcounter{equation}\text{(\theequation)}\label{fp11}] $P^{\prime
}\in \bigcap_{S\subset E,\#\left( S\right) \leq \left( D+2\right) ^{l-1}}%
\hat{\Gamma}\left( S\right) $.
\end{itemize}

In particular, $P^{\prime }\in \hat{\Gamma}\left( \emptyset \right) $, which
implies that 
\begin{equation*}
\left\vert \partial ^{\beta }\left( P-P^{\prime }\right) \left( x\right)
\right\vert \leq M_{0}\left\vert x-y\right\vert ^{m-\left\vert \beta
\right\vert }\text{ for }\left\vert \beta \right\vert \leq m-1\text{.}
\end{equation*}%
Also, \eqref{fp10}, \eqref{fp11} imply that 
\begin{equation*}
P^{\prime }\in \bigcap_{S\subset E,\#\left( S\right) \leq \left( D+2\right)
^{l-1}}\Gamma \left( y,S\right) =\Gamma _{l-1}^{fp}\left( y,M_{0}\right) 
\text{.}
\end{equation*}%
Thus, $P^{\prime }$ satisfies \eqref{fp9}, completing the proof of Lemma \ref%
{lemma-fp2}.
\end{proof}

\begin{theorem}[Finiteness Principle for Shape Fields]
\label{theorem-fp-for-wsf} For a large enough $k^{\#}$ determined by $m$, $n$%
, the following holds. Let $\vec{\Gamma}_{0}=\left( \Gamma _{0}\left(
x,M\right) \right) _{x\in E,M>0}$ be a $\left( C_{w},\delta _{\max }\right) $%
-convex shape field and let $Q_{0}\subset \mathbb{R}^{n}$ be a cube of
sidelength $\delta _{Q_{0}}\leq \delta _{\max }$. Also, let $x_{0}\in E\cap
5Q_{0}$ and $M_{0}>0$ be given. Assume that for each $S\subset E$ with $%
\#\left( S\right) \leq k^{\#}$ there exists a Whitney field $\vec{P}%
^{S}=\left( P^{z}\right) _{z\in S}$ such that 
\begin{equation*}
\left\Vert \vec{P}^{S}\right\Vert _{\dot{C}^{m}\left( S\right) }\leq M_{0}%
\text{,}
\end{equation*}%
and 
\begin{equation*}
P^{z}\in \Gamma _{0}\left( z,M_{0}\right) \text{ for all }z\in S\text{.}
\end{equation*}%
Then there exist $P^{0}\in \Gamma _{0}\left( x_{0},M_{0}\right) $ and $F\in
C^{m}\left( Q_{0}\right) $ such that the following hold, with a constant $%
C_{\ast }$ determined by $C_{w}$, $m$, $n$:

\begin{itemize}
\item $J_{z}(F)\in \Gamma _{0}\left( z,C_{\ast }M_{0}\right) $ for all $z\in
E\cap Q_{0}$.

\item $|\partial ^{\beta }\left( F-P^{0}\right) \left( x\right) |\leq
C_{\ast }M_{0}\delta _{Q_{0}}^{m-\left\vert \beta \right\vert }$ for all $%
x\in Q_{0}$, $\left\vert \beta \right\vert \leq m$.

\item In particular, $\left\vert \partial ^{\beta }F\left( x\right)
\right\vert \leq C_{\ast }M_{0}$ for all $x\in Q_{0}$, $\left\vert \beta
\right\vert =m$.
\end{itemize}
\end{theorem}

\begin{proof}
For $l\geq 1$, define $\vec{\Gamma}_{l}=\left( \Gamma _{l}\left( x,M\right)
\right) _{x\in E,M>0}$ and $\vec{\Gamma}_{l}^{fp}=\left( \Gamma
_{l}^{fp}\left( x,M\right) \right) _{x\in E,M>0}$ as in Lemmas \ref%
{lemma-fp1} and \ref{lemma-fp2}. We take $l_{\ast }=100+l\left( \emptyset
\right) $ and $k^{\#}=100+\left( D+2\right) ^{l_{\ast }+100}$. (For the
definition of $l\left( \emptyset \right) $, see Section \ref%
{statement-of-the-main-lemma}.)

Lemmas \ref{lemma-fp1} and \ref{lemma-fp2} show that $\Gamma _{l_{\ast
}}^{fp}\left( x_{0},M_{0}\right) $ is nonempty, hence $\Gamma _{l\left(
\emptyset \right) +1}\left( x_{0},M_{0}\right) $ is nonempty. Pick any $%
P^{0}\in \Gamma _{l\left( \emptyset \right) +1}\left( x_{0},M_{0}\right)
\subset \Gamma _{0}\left( x_{0},M_{0}\right) $. Then Theorem \ref%
{theorem-tu1} in Section \ref{tu} produces a function $F\in C^{m}(Q_{0})$
with the desired properties.
\end{proof}

The finiteness principle for $\vec{\Gamma}=\left(\Gamma(x,M)\right)_{x \in E, M>0}$ stated in the Introduction follows easily from Theorem \ref{theorem-fp-for-wsf}, under the following assumptions on $\vec{\Gamma}$: 
\begin{itemize}
\item $\vec{\Gamma}$ is a $(C_w,\delta_{\max})$-convex shape field, with $\delta_{\max}=1$.
\item Any $P \in \Gamma(x,M)$ satisfies $|\partial^\beta P(x)| \leq M$ for $|\beta| \leq m-1$. 
\end{itemize}

We have the following corollary to Theorem \ref{theorem-fp-for-wsf}.

\begin{corollary}
\label{lemma-to-previous-results-on-shape-fields}Let $Q_{0}$ be a cube of
sidelength $\delta _{Q_{0}}\leq \delta _{\max }$, and let $x_{0}\in E\cap
Q_{0}$. Let $\vec{\Gamma}_0 = (\Gamma_0(x,M))_{x \in E, M >0}$ be a $%
(C_w,\delta_{\max})$-convex shape field, and for $l \geq 1$, let $\vec{\Gamma%
}_l = (\Gamma_l(x,M))_{ x\in E, M >0}$ be its $l^{th}$ refinement. Let $M_{0}>0$, and let $P_{0}\in \Gamma _{l_{\ast }}(x_{0},M_{0})$,
where $l_{\ast }$ is a large enough integer constant determined by $m$ and $%
n $.

Then there exists $F \in C^m(Q_0)$ such that

\begin{itemize}
\item $|\partial^\alpha F(x)| \leq C_*M_0$ for $x \in Q_0$, $|\alpha|=m$.

\item $J_{x}(F)\in \Gamma _{0}\left( x,C_{\ast }M_{0}\right) $ for all $x\in
E\cap Q_{0}$.

\item $J_{x_{0}}\left( F\right) =P_{0}$.
\end{itemize}

Here, $C_*$ depends only on $m$, $n$, $C_w$.
\end{corollary}

\begin{remark}
The corollary strengthens Theorem \ref{theorem-fp-for-wsf}, because we now obtain $J_{x_{0}}\left(
F\right) =P_{0}$ rather than the weaker assertion $\left\vert \partial
^{\beta }\left( F-P_{0}\right) \right\vert \leq C_{\ast }M_{0}\delta
_{Q_{0}}^{m-\left\vert \beta \right\vert }$ for $\left\vert \beta
\right\vert \leq m$. The $l_{\ast }$ here is much bigger than in the proof of Theorem \ref{theorem-fp-for-wsf}; maybe one
can do better than here.
\end{remark}

To prove the corollary, we use a simple \textquotedblleft clustering lemma",
namely Lemma~3.4 in \cite{f-2006}.

\begin{lemma}[Clustering Lemma]
\label{clustering-lemma}Let $S\subset \mathbb{R}^{n}$, with $2\leq \#\left(
S\right) \leq k^{\#}$. Then $S$ can be partitioned into nonempty subsets $%
S_{0}$, $S_{1}$, $\cdots $, $S_{\nu _{\max }}$, such that $\#(S_\nu) < \#(S)$ for each $\nu$, and $\text{dist}\left( S_{\nu
},S_{\mu }\right) \geq c\cdot \text{diam}\left( S\right) $, for $\mu \not=\nu $,
with $c$ depending only on $k^{\#}$.
\end{lemma}

We use Lemma \ref{clustering-lemma} to prove the following lemma.

\begin{lemma}
\label{refinement-lemma1}Let $\vec{\Gamma}_{0}=\left( \Gamma _{0}\left(
x,M\right) \right) _{x\in E,M>0}$ be a shape field. For $l\geq 1$, let $\vec{%
\Gamma}_{l}=\left( \Gamma _{l}\left( x,M\right) \right) _{x\in E,M>0}$ be
the $l^{th}$ refinement of $\vec{\Gamma}_{0}$. Let $l_{\ast }\geq 1$, $%
M_{0}>0,$ $x_{0}\in E,$ and let $P_{0}\in \Gamma _{l_{\ast }}\left(
x_{0},M_{0}\right) $. Then for any $S\subset E$ with $x_{0}\in S$ and $%
\#\left( S\right) \leq l_{\ast }$, there exists a Whitney field $\vec{P}%
^{S}=\left( P^{x}\right) _{x\in S}$ such that $P^{x}\in \Gamma _{0}\left(
x,C_{\ast }M_{0}\right) $ for $x\in S$; 
\[
\left\vert \partial ^{\beta }\left( P^{x}-P^{y}\right) \left( x\right)
\right\vert \leq C_{\ast }M_{0}\left\vert x-y\right\vert ^{m-\beta }\text{
for }x,y\in S\text{, }\left\vert \beta \right\vert \leq m-1\text{;} 
\]%
and $P^{x_{0}}=P_{0}.$ Here, $C_{\ast }$ depends only on $m$, $n$, $l_{\ast
} $.
\end{lemma}

\begin{proof}[Proof of Lemma \protect\ref{refinement-lemma1}]
Throughout the proof of Lemma \ref{refinement-lemma1}, $C_{\ast }$ denotes a
constant determined by $m$, $n$, $l_{\ast }$.

We use induction on $l_{\ast }$.

In the base case, $l_{\ast }=1,$ so $S\equiv \left\{ x_{0}\right\} $, and we
can take our Whitney field $\vec{P}^{S}$ to consist of the single polynomial 
$P_{0}$.

For the induction step, we fix $l_{\ast }\geq 2$, and make the inductive
assumption: Lemma \ref{refinement-lemma1} holds with $\left( l_{\ast
}-1\right) $ in place of $l_{\ast }$. We will then prove Lemma \ref%
{refinement-lemma1} for the given $l_{\ast }$.

Let $\vec{\Gamma}_{l},$ $M_{0},$ $x_{0}$, $P_{0}$, $S$ be as in the
hypotheses of Lemma \ref{refinement-lemma1}.

If $\#\left( S\right) \leq l_{\ast }-1$, then the conclusion of Lemma \ref%
{refinement-lemma1} follows instantly from our inductive assumption. Suppose 
$\#\left( S\right) =l_{\ast }$.

Let $S_{0},S_{1},\cdots ,S_{\nu _{\max }}$ be a partition of $S$ as in Lemma %
\ref{clustering-lemma}. We may assume that the $S_{\nu }$ are numbered so
that $x_{0}\in S_{0}.$ For each $\nu =1,\cdots ,\nu _{\max }$, we pick a
point $x_{\nu }\in S_{\nu }$.

Recall that $P_{0}\in \Gamma _{l_{\ast }}\left( x_{0},M_{0}\right) $. By
definition of the $l^{th}$ refinement, for each $\nu =1,\cdots ,\nu _{\max }$%
, there exists

\begin{itemize}
\item[\refstepcounter{equation}\text{(\theequation)}\label{refined1}] $P_{\nu }\in \Gamma _{l_{\ast }-1}\left( x_{\nu },M_{0}\right) $ such that $%
\left\vert \partial ^{\beta }\left( P_{\nu }-P_{0}\right) \left(
x_{0}\right) \right\vert \leq M_{0}\left\vert x_{\nu }-x_{0}\right\vert
^{m-\left\vert \beta \right\vert }$ for $\left\vert \beta \right\vert \leq
m-1$.
\end{itemize}

Fix such $P_{\nu }$ for $\nu =1,\cdots ,\nu _{\max }$; and note that \eqref{refined1}
holds also for $\nu =0.$

For each $\nu =0,1,\cdots ,\nu _{\max }$, we have $P_{\nu }\in \Gamma
_{l_{\ast }-1}\left( x_{\nu },M_{0}\right) $, $x_{\nu }\in S_{\nu }\subset E$,
and $\#\left( S_{\nu }\right) \leq l_{\ast }-1$.

Hence our inductive assumption produces a Whitney field $\vec{P}^{S_{\nu
}}=\left( P_{\nu }^{x}\right) _{x\in S_{\nu }}$ such that $P_{\nu }^{x}\in
\Gamma _{0}\left( x,C_{\ast }M_{0}\right) $ for $x\in S_{\nu }$, $\left\vert
\partial ^{\beta }\left( P_{\nu }^{x}-P_{\nu }^{y}\right) \left( x\right)
\right\vert \leq C_{\ast }M_{0}\left\vert x-y\right\vert ^{m-\left\vert
\beta \right\vert }$ for $x,y\in S_{\nu }$, $\left\vert \beta \right\vert
\leq m-1$, $P_{\nu }^{x_{\nu }}=P_{\nu }$.

We now combine the $\vec{P}^{S_{\nu }}$ into a single Whitney field $\vec{P}%
^{S}=\left( P^{x}\right) _{x\in S}$ by setting $P^{x}=P_{\nu }^{x}$ for $%
x\in S_{\nu }$. Note that $P^{x}\in \Gamma _{0}\left( x,C_{\ast
}M_{0}\right) $ for $x\in S$, and that $P^{x_{0}}=P_{0}^{x_{0}}=P_{0}$. We
check that

\begin{itemize}
\item[\refstepcounter{equation}\text{(\theequation)}\label{refined2}] $\left\vert \partial ^{\beta }\left( P^{x}-P^{y}\right) \left( x\right)
\right\vert \leq C_{\ast }M_{0}\left\vert x-y\right\vert ^{m-\left\vert
\beta \right\vert }$ for $x,y$ $\in S$, $\left\vert \beta \right\vert \leq
m-1$.\end{itemize}

We already know \eqref{refined2} for $x,y$ belonging to the same $S_{\nu }$.

Suppose $x\in S_{\nu }$, $y\in S_{\mu }$ with $\mu \not=\nu $. Then $%
\left\vert x-y\right\vert $ is comparable to diam $\left( S\right) $ so \eqref{refined2}
asserts that

\begin{itemize}
\item[\refstepcounter{equation}\text{(\theequation)}\label{refined3}]  $\left\vert \partial ^{\beta }\left( P^{x}-P^{y}\right) \left( x\right)
\right\vert \leq C_{\ast }M_{0}$ (diam $\left( S\right) $)$%
^{m-\left\vert \beta \right\vert }$ for $\left\vert \beta \right\vert \leq
m-1$. \end{itemize}

We know that

\begin{itemize}
\item[\refstepcounter{equation}\text{(\theequation)}\label{refined4}]  $\left\vert \partial ^{\beta }\left( P^{x}-P_{\nu }\right) \left( x_{\nu
}\right) \right\vert =\left\vert \partial ^{\beta }\left( P_{\nu
}^{x}-P_{\nu }^{x_{\nu }}\right) \left( x_{\nu }\right) \right\vert \leq
C_{\ast }M_{0}\left\vert x-x_{\nu }\right\vert ^{m-\left\vert \beta
\right\vert }\leq$ \\ $C_{\ast }M_{0} \left( \text{diam}\left( S\right) \right)
^{m-\left\vert \beta \right\vert }$ for $\left\vert \beta \right\vert \leq
m-1.$ \end{itemize}

Similarly,

\begin{itemize}
\item[\refstepcounter{equation}\text{(\theequation)}\label{refined5}]  $\left\vert \partial ^{\beta }\left( P^{y}-P_{\mu }\right) \left( x_{\mu
}\right) \right\vert \leq C_{\ast }M_{0}\left( \text{diam}\left( S\right) \right)
^{m-\left\vert \beta \right\vert }$ for $\left\vert \beta \right\vert \leq
m-1$.\end{itemize}

\begin{itemize}
\item[\refstepcounter{equation}\text{(\theequation)}\label{refined6}] Also, $\left\vert \partial ^{\beta }\left( P_{\nu }-P_{0}\right) \left(
x_{0}\right) \right\vert \leq C_{\ast }M_{0}\left( \text{diam}\left( S\right)
\right) ^{m-\left\vert \beta \right\vert }$ for $\left\vert \beta
\right\vert \leq m-1$, and \end{itemize}

\begin{itemize}
\item[\refstepcounter{equation}\text{(\theequation)}\label{refined7}] $\left\vert \partial ^{\beta }\left( P_{\mu }-P_{0}\right) \left(
x_{0}\right) \right\vert \leq C_{\ast }M_{0}\left( \text{diam}\left( S\right)
\right) ^{m-\left\vert \beta \right\vert }$ for $\left\vert \beta
\right\vert \leq m-1$.\end{itemize}

Because the points $x_{0},$ $x_{\mu }$, $x_{\nu }$, $x$ all lie in $S$, the
distance between any two of these points is at most diam $\left( S\right)$.

Hence, the estimates \eqref{refined4},$\cdots $,\eqref{refined7} together imply \eqref{refined3}, completing the
proof of \eqref{refined2}. This completes our induction on $l_{\ast },$ thus establishing
Lemma \ref{refinement-lemma1}.
\end{proof}

\begin{proof}[Proof of Corollary \ref{lemma-to-previous-results-on-shape-fields}]
We set $\hat{\Gamma}=\left( \hat{\Gamma}\left( x,M\right) \right) _{x\in
E,M>0}$, where 
\[
\hat{\Gamma}\left( x,M\right) =\left\{ 
\begin{array}{c}
\Gamma _{0}\left( x,M\right)  \\ 
\left\{ P_{0}\right\} 
\end{array}%
\right. 
\begin{array}{l}
\text{if }x\in E\setminus \left\{ x_{0}\right\}  \\ 
\text{if }x=x_{0}%
\end{array}%
.
\]%
One checks trivially that $\hat{\Gamma}$ is a $\left( C_{w},\delta _{\max
}\right) $-convex shape field. 

By applying Lemma \ref{refinement-lemma1} with $l_{\ast }=k^{\#}+1$ (so that
if necessary we can add $x_{0}$ into the set $S$ below), we obtain the
following conclusion:

\begin{itemize}
\item[\refstepcounter{equation}\text{(\theequation)}\label{refined8}] Given $S \subset E$ with $\#(S)\leq k^\#$, there exists a Whitney field $\vec{P}^{S}=\left( P^{x}\right) _{x\in S}$ such that $%
P^{x}\in \hat{\Gamma}\left( x,C_{\ast }M_{0}\right) $ for $x\in S,$ and $%
\left\vert \partial ^{\beta }\left( P^{x}-P^{y}\right) \left( x\right)
\right\vert \leq C_{\ast }M_{0}\left\vert x-y\right\vert ^{m-\left\vert
\beta \right\vert }$ for $x,y\in S,$ $\left\vert \beta \right\vert \leq m-1$.
\end{itemize}

Here, $C_{\ast }$ depends only on $m$, $n$, $k^{\#}$.

For large enough $k^{\#}$ depending only on $m,$ $n,$ \eqref{refined8} and Theorem \ref{theorem-fp-for-wsf}
together imply the conclusion of our corollary.

The proof of Corollary \ref{lemma-to-previous-results-on-shape-fields} is
complete. 
\end{proof}

\section{Finiteness Principle II}

\label{fpii}

\begin{proof}[Proof of Theorem \ref{Th3} (A)]
Let us first set up notation. We write $c$, $C$, $C^{\prime }$, etc., to
denote constants determined by $m$, $n$, $D$; these symbols may denote
different constants in different occurrences. We will work with $C^{m}$
vector and scalar-valued functions on $\mathbb{R}^{n}$, and also with $%
C^{m+1}$ scalar-valued functions on $\mathbb{R}^{n+D}$. We use Roman letters
($x$, $y$, $z$$,\cdots $) to denote points of $\mathbb{R}^{n}$, and Greek
letters $(\xi ,\eta ,\zeta ,\cdots )$ to denote points of $\mathbb{R}^{D}$.
We denote points of the $\mathbb{R}^{n+D}$ by $(x,\xi )$, $(y,\eta )$, etc.
As usual, $\mathcal{P}$ denotes the vector space of real-valued polynomials of degree at
most $m-1$ on $\mathbb{R}^{n}$. We write $\mathcal{P}^{D}$ to denote the
direct sum of $D$ copies of $\mathcal{P}$. If $F\in C^{m-1}(\mathbb{R}^{n},%
\mathbb{R}^{D})$ with $F(x)=\left( F_{1}\left( x\right) ,\cdots ,F_{D}\left(
x\right) \right) $ for $x\in \mathbb{R}^{n}$, then $J_{x}(F):=(J_{x}\left(
F_{1}\right) ,\cdots ,J_{x}\left( F_{D}\right) )\in \mathcal{P}^{D}$.

We write $\mathcal{P}^{+}$ to denote the vector space of real-valued polynomials of
degree at most $m$ on $\mathbb{R}^{n+D}$. If $F\in C^{m+1}\left( \mathbb{R}%
^{n+D}\right) $, then we write $J_{\left( x,\xi \right) }^{+}F\in \mathcal{P}%
^{+}$ to denote the $m^{th}$-degree Taylor polynomial of $F$ at the point $%
\left( x,\xi \right) \in \mathbb{R}^{n+D}$.

When we work with $\mathcal{P}^{+}$, we write $\odot _{\left( x,\xi \right)
} $ to denote the multiplication 
\begin{equation*}
P\odot _{\left( x,\xi \right) }Q:=J_{\left( x,\xi \right) }^{+}\left(
PQ\right) \in \mathcal{P}^{+}\text{ for }P,Q\in \mathcal{P}^{+}\text{.}
\end{equation*}

We will use Theorem \ref{theorem-fp-for-wsf} for $C^{m+1}$-functions on $%
\mathbb{R}^{n+D}$. (Compare with \cite{fl-2014}.) Thus, $m+1$ and $n+D$ will play the r\^oles of $m$, $n$,
respectively, when we apply that theorem.

We take $k^{\#}$ as in Theorem \ref{theorem-fp-for-wsf}, where we use $%
m+1,n+D$ in place of $m,n$, respectively.

We now introduce the relevant shape field.

Let $E^{+}=\left\{ \left( x,0\right) \in \mathbb{R}^{n+D}:x\in E\right\} $.
For $\left( x_{0},0\right) \in E^{+}$ and $M>0$, let

\begin{itemize}
\item[\refstepcounter{equation}\text{(\theequation)}\label{fpii3}] $\Gamma
\left( \left( x_0,0\right) ,M\right) =\left\{ 
\begin{array}{c}
P\in \mathcal{P}^{+}:P\left( x_{0},0\right) =0,\nabla _{\xi }P\left(
x_{0},0\right) \in K\left( x_{0}\right) , \\ 
\left\vert \partial _{x}^{\alpha }\partial _{\xi }^{\beta }P\left(
x_{0},0\right) \right\vert \leq M\text{ for }\left\vert \alpha \right\vert
+\left\vert \beta \right\vert \leq m%
\end{array}%
\right\} \mathcal{\subset \mathcal{P}}^{+}$.
\end{itemize}

Let $\vec{\Gamma}=\left( \Gamma \left( x_{0},0\right) ,M\right) $ $_{\left(
x_0,0\right) \in E^{+},M>0}$.

\begin{lemma}
\label{lemma-bfp1} $\vec{\Gamma}$ is a $(C,1)$-convex shape field.
\end{lemma}

The proof of the lemma follows easily from the following observation: If $%
P_{1},P_{2}\in \Gamma \left( \left( x_{0},0\right) ,M\right) $ and $%
Q_{1},Q_{2}\in \mathcal{P}^{+}$, then $%
P_{1}\left( x_{0},0\right) =P_{2}\left( x_{0},0\right) =0$, hence $%
P:=Q_{1}\odot _{\left( x_{0},0\right) }Q_{1}\odot _{\left( x_{0},0\right)
}P_{1}+Q_{2}\odot _{\left( x_{0},0\right) }Q_{2}\odot _{\left(
x_{0},0\right) }P_{2}$ satisfies $P\left( x_{0},0\right) =0$ and $\nabla _{\xi }P\left( x_{0},0\right) =\left( Q_{1}\left( x_{0},0\right)
\right) ^{2}\nabla _{\xi }P_{1}\left( x_{0},0\right) +\left( Q_{2}\left(
x_{0},0\right) \right) ^{2}\nabla _{\xi }P_{2}\left( x_{0},0\right)$.

\begin{lemma}
\label{lemma-bfp2} Let $S^{+}\subset E^{+}$ with $\#\left( S^{+}\right) \leq
k^{\#}$. Then there exists $\vec{P}=\left( P^{z}\right) _{z\in S^{+}}$, with
each $P^{z}\in \mathcal{P}^{+}$, such that

\begin{itemize}
\item[\refstepcounter{equation}\text{(\theequation)}\label{fpii16}] $%
P^{z}\in \Gamma \left( z,C\right) $ for each $z\in S^{+}$, and
\end{itemize}

\begin{itemize}
\item[\refstepcounter{equation}\text{(\theequation)}\label{fpii17}] $%
\left\vert \partial _{x}^{\alpha }\partial _{\xi }^{\beta }\left(
P^{z}-P^{z^{\prime }}\right) \left( z\right) \right\vert \leq C\left\vert
z-z^{\prime }\right\vert ^{\left( m+1\right) -\left\vert \alpha \right\vert
-\left\vert \beta \right\vert }$ for $z$, $z^{\prime }\in S^{+}$ and $%
\left\vert \alpha \right\vert +\left\vert \beta \right\vert \leq m$.
\end{itemize}
\end{lemma}

\begin{proof}[Proof of Lemma \protect\ref{lemma-bfp2}]
Since $E^{+}=E\times \left\{ 0\right\} $, we have $S^{+}=S\times \left\{
0\right\} $ for an $S\subset E$ with $\#\left( S\right) \leq k^{\#}$. By
hypothesis of Theorem \ref{Th3} (A), there exists $%
F^{S}\in C^{m}\left( \mathbb{R}^{n},\mathbb{R}^{D}\right) $ such that

\begin{itemize}
\item[\refstepcounter{equation}\text{(\theequation)}\label{fpii18}] $%
\left\Vert F^{S}\right\Vert _{C^{m}\left( \mathbb{R}^{n},\mathbb{R}%
^{D}\right) }\leq 1$ and $F^{S}\left( x_{0}\right) \in K\left( x_{0}\right) $
for all $x_{0}\in S$.
\end{itemize}

Let $F^{S}\left( x\right) =\left( F_{1}^{S}\left( x\right) ,\cdots
,F_{D}^{S}\left( x\right) \right) $ for $x\in \mathbb{R}^{n}$, and let $\vec{P%
}=\left( P^{\left( x_{0},0\right) }\right) _{\left( x_{0},0\right) \in
S\times \left\{ 0\right\} }$ with

\begin{itemize}
\item[\refstepcounter{equation}\text{(\theequation)}\label{fpii19}] $%
P^{\left( x_{0},0\right) }\left( x,\xi \right) =\sum_{i=1}^{D}\xi _{i}\left[
J_{x_{0}}\left( F_{i}^{S}\right) \left( x\right) \right] $ for $x\in \mathbb{%
R}^{n}$, $\xi =\left( \xi _{1},\cdots ,\xi _{D}\right) \in \mathbb{R}^{D}$.
\end{itemize}

Each $P^{(x_0,0)}$ belongs to $\mathcal{P}^+$ and satisfies

\begin{itemize}
\item[\refstepcounter{equation}\text{(\theequation)}\label{fpii20}] $%
P^{\left( x_{0},0\right) }\left( x_{0},0\right) =0$, $\nabla _{\xi
}P^{\left( x_{0},0\right) }\left( x_{0},0\right) \in K\left( x_{0}\right) $,
and
\end{itemize}

and

\begin{itemize}
\item[\refstepcounter{equation}\text{(\theequation)}\label{fpii21}] $%
\left\vert \partial _{x}^{\alpha }\partial _{\xi }^{\beta }P^{\left(
x_{0},0\right) }\left( x_{0},0\right) \right\vert \leq C$ for $\left\vert
\alpha \right\vert +\left\vert \beta \right\vert \leq m$,
\end{itemize}
thanks to \eqref{fpii18}, \eqref{fpii19}. Our results \eqref{fpii20}, %
\eqref{fpii21} and definition \eqref{fpii3} together imply \eqref{fpii16}.
We pass to \eqref{fpii17}. Let $\left( x_{0},0\right) ,\left( y_{0},0\right)
\in S^{+}=S\times \left\{ 0\right\} $. From \eqref{fpii18}, \eqref{fpii19},
we have 
\begin{eqnarray*}
\left\vert \partial _{x}^{\alpha }\partial _{\xi _{j}}\left( P^{\left(
x_{0},0\right) }-P^{\left( y_{0},0\right) }\right) \left( x_{0},0\right)
\right\vert  &=&\left\vert \partial _{x}^{\alpha }\left( J_{x_{0}}\left(
F_{j}^{S}\right) -J_{y_{0}}\left( F_{j}^{S}\right) \right) \left(
x_{0}\right) \right\vert  \\
&\leq &C\left\vert x_{0}-y_{0}\right\vert ^{m-\left\vert \alpha \right\vert }
\\
&=&C\left\vert x_{0}-y_{0}\right\vert ^{\left( m+1\right) -\left( \left\vert
\alpha \right\vert +1\right) }
\end{eqnarray*}%
for $\left\vert \alpha \right\vert \leq m-1$, $j=1,\cdots ,D$. For $|\beta
|\not=1$, we have 
\begin{equation*}
\partial _{x}^{\alpha }\partial _{\xi }^{\beta }\left( P^{\left(
x_{0},0\right) }-P^{\left( y_{0},0\right) }\right) \left( x_{0},0\right) =0
\end{equation*}%
by \eqref{fpii19}. The above remarks imply \eqref{fpii17}, completing the
proof of Lemma \ref{lemma-bfp2}.
\end{proof}

\begin{lemma}
\label{lemma-bfp3} Given a cube $Q\subset \mathbb{R}^{n}$ of sidelength $%
\delta _{Q}=1$, there exists $F^{Q}\in C^{m}(Q,\mathbb{R}^{D})$ such that

\begin{itemize}
\item[\refstepcounter{equation}\text{(\theequation)}\label{fpii22}] $%
\left\vert \partial ^{\alpha }F^{Q}\left( x\right) \right\vert \leq C$ for $%
x\in Q$, $\left\vert \alpha \right\vert \leq m$; and
\end{itemize}

\begin{itemize}
\item[\refstepcounter{equation}\text{(\theequation)}\label{fpii23}] $%
F^{Q}\left( z\right) \in K\left( z\right) $ for all $z\in E\cap Q$.
\end{itemize}
\end{lemma}

\begin{proof}[Proof of Lemma \protect\ref{lemma-bfp3}]
If $E\cap Q=\emptyset $, then we can just take $F^{Q}\equiv 0$. Otherwise,
pick $x_{00}\in E\cap Q$, let $Q^{\prime }\in \mathbb{R}^{D}$ be a cube of
sidelength $\delta _{Q^{\prime }}=1$, containing the origin in its interior,
and apply Theorem \ref{theorem-fp-for-wsf} (with $m+1$, $n+D$ in place of $m$%
, $n$, respectively) to the shape field $\vec{\Gamma}=\left( \Gamma \left(
x_{0},0\right) ,M\right) _{\left( x_{0},0\right) \in E^{+},M>0}$ given by %
\eqref{fpii3}, the cube $Q_{0}:=Q\times Q^{\prime }\subset \mathbb{R}^{n+D}$%
, the point $(x_{00},0)$, and the number $M_{0}=C$.

Lemmas \ref{lemma-bfp1} and \ref{lemma-bfp2} tell us that the above data
satisfy the hypotheses of Theorem \ref{theorem-fp-for-wsf}. Applying Theorem %
\ref{theorem-fp-for-wsf}, we obtain

\begin{itemize}
\item[\refstepcounter{equation}\text{(\theequation)}\label{fpii24}] $%
P_{0}\in \Gamma \left( \left( x_{00},0\right) ,C\right) $ and $F\in
C^{m+1}\left( Q\times Q^{\prime }\right) $ such that
\end{itemize}

\begin{itemize}
\item[\refstepcounter{equation}\text{(\theequation)}\label{fpii25}] $%
\left\vert \partial _{x}^{\alpha }\partial _{\xi }^{\beta }\left(
F-P_{0}\right) \left( x,\xi \right) \right\vert \leq C$ for $\left\vert
\alpha \right\vert +\left\vert \beta \right\vert \leq m+1$ and $\left( x,\xi
\right) \in Q\times Q^{\prime }$; and
\end{itemize}

\begin{itemize}
\item[\refstepcounter{equation}\text{(\theequation)}\label{fpii26}] $%
J_{\left( z,0\right) }^{+}\left( F\right) \in \Gamma \left( \left(
z,0\right) ,C\right) $ for all $z\in E\cap Q$.
\end{itemize}

By \eqref{fpii24}, \eqref{fpii26} and definition \eqref{fpii3}, we have

\begin{itemize}
\item[\refstepcounter{equation}\text{(\theequation)}\label{fpii27}] $%
\left\vert \partial _{x}^{\alpha }\partial _{\xi }^{\beta }P_{0}\left(
x_{00},0\right) \right\vert \leq C$ for $\left\vert \alpha \right\vert
+\left\vert \beta \right\vert \leq m$
\end{itemize}

and

\begin{itemize}
\item[\refstepcounter{equation}\text{(\theequation)}\label{fpii28}] $\nabla
_{\xi }F\left( z,0\right) \in K\left( z\right) $ for all $z\in E\cap Q$.
\end{itemize}

Since $\left( x_{00},0\right) \in Q\times Q^{\prime }$ and $\delta _{Q\times
Q^{\prime }}=1$, \eqref{fpii27} implies that 
\begin{equation*}
\left\vert \partial _{x}^{\alpha }\partial _{\xi }^{\beta }P_{0}\left( x,\xi
\right) \right\vert \leq C
\end{equation*}%
for $\left( x,\xi \right) \in Q\times Q^{\prime }$, $\left\vert \alpha
\right\vert +\left\vert \beta \right\vert \leq m+1$. (Recall that $P_{0}$ is
a polynomial of degree at most $m$.) Together with \eqref{fpii25}, this
implies that

\begin{itemize}
\item[\refstepcounter{equation}\text{(\theequation)}\label{fpii29}] $%
\left\vert \partial _{x}^{\alpha }\partial _{\xi }^{\beta }F\left( x,\xi
\right) \right\vert \leq C$ for $\left( x,\xi \right) \in Q\times Q^{\prime }
$, $\left\vert \alpha \right\vert +\left\vert \beta \right\vert \leq m+1$.
\end{itemize}

Taking 
\begin{equation*}
F^{Q}\left( x\right) =\nabla _{\xi }F\left( x,0\right) \text{ for }x\in Q%
\text{,}
\end{equation*}%
we learn from \eqref{fpii28}, \eqref{fpii29} that $F^{Q}\in C^{m}\left( Q,%
\mathbb{R}^{D}\right) $; $\left\vert \partial ^{\alpha }F^{Q}\left( x\right)
\right\vert \leq C$ for $x\in Q$, $\left\vert \alpha \right\vert \leq m$;
and $F^{Q}\left( z\right) \in K\left( z\right) $ for all $z\in E\cap Q$.
Thus, $F^{Q}$ satisfies \eqref{fpii22} and \eqref{fpii23}, completing the
proof of Lemma \ref{lemma-bfp3}.
\end{proof}

It is now trivial to complete the proof of Theorem \ref{Th3} (A) by using a partition of unity on $\mathbb{R}^n$.
\end{proof}

\bigskip

\begin{proof}[Proof of Theorem \ref{Th3} (B)]
We write $c$, $C$, $C^{\prime }$, etc., to denote constants determined by $m$%
, $n$, $D$. These symbols may denote different constants in different
occurrences.

Suppose first that $E$ is finite.

Given $F^{S}$ as in the hypothesis of Theorem \ref{Th3} (B), set $\vec{P}=\left( P^{x}\right)
_{x\in S}$, with $P^{x}=J_{x}\left( F^{S}\right) \in \mathcal{P}^{D}$. Then $%
P^{x}\left( x\right) \in K\left( x\right) $ for $x\in S$, $\left\vert
\partial ^{\beta }P^{x}\left( x\right) \right\vert \leq C$ for $x\in S$, $%
\left\vert \beta \right\vert \leq m-1$, and $\left\vert \partial ^{\beta
}\left( P^{x}-P^{y}\right) \left( x\right) \right\vert \leq C\left\vert
x-y\right\vert ^{m-\left\vert \beta \right\vert }$ for $x,y\in S$, $%
\left\vert \beta \right\vert \leq m-1$.

By Whitney's extension theorem for finite sets, there exists

\begin{itemize}
\item $\tilde{F}^{S}\in C^{m}\left( \mathbb{R}^{n},\mathbb{R}^{D}\right) $
such that

\item $\left\Vert \tilde{F}^{S}\right\Vert _{C^{m}\left( \mathbb{R}^{n},%
\mathbb{R}^{D}\right) }\leq C$, and
\end{itemize}

$J_{x}\left( \tilde{F}^{S}\right) =P^{x}$ for $x\in S$; in particular,

\begin{itemize}
\item $\tilde{F}^{S}\left( x\right) =F^{S}\left( x\right) \in K\left(
x\right) $ for $x\in S$.
\end{itemize}

Thanks to the above bullet points, we have satisfied the hypotheses of Theorem \ref{Th3} (A). Hence, we obtain $F\in C^{m}\left( 
\mathbb{R}^{n},\mathbb{R}^{D}\right) \subset C^{m-1,1}\left( \mathbb{R}^{n},%
\mathbb{R}^{D}\right) $ such that $\left\Vert F\right\Vert _{C^{m-1,1}\left( 
\mathbb{R}^{n},\mathbb{R}^{D}\right) }\leq C\left\Vert F\right\Vert
_{C^{m}\left( \mathbb{R}^{n},\mathbb{R}^{D}\right) }\leq C^{\prime }$, and $%
F\left( x\right) \in K\left( x\right) $ for each $x\in E$. Thus, we have
proven Theorem \ref{Th3} (B) in the case of
finite $E$.

Next, suppose $E$ is an arbitrary subset of a cube $Q\subset \mathbb{R}^{n}$%
. Then 
\begin{equation*}
X=\left\{ F\in C^{m-1,1}\left( Q,\mathbb{R}^{D}\right) :\left\Vert
F\right\Vert _{C^{m-1,1}\left( Q,\mathbb{R}^{D}\right) }\leq C\right\} 
\end{equation*}%
is compact in the topology of the $C^{m-1}(Q,\mathbb{R}^{D})$-norm, by
Ascoli's theorem.

For each $x\in E$, let 
\begin{equation*}
X(x)=\left\{ F\in X:F\left( x\right) \in K\left( x\right) \right\} \text{.}
\end{equation*}

Then each $X(x)$ is a closed subset of $X$, since $K(x)\subset \mathbb{R}%
^{n} $ is closed. Moreover, given finitely many points $x_{1},\cdots
,x_{N}\in E$, we have $X(x_{1})\cap \cdots \cap X(x_{N})\not=\emptyset $,
thanks to Theorem \ref{Th3} (B) in the known
case of finite sets. 

Consequently, $\bigcap_{x\in E}X(x)\not=\emptyset $. Thus, there exists $%
F\in C^{m-1,1}\left( Q,\mathbb{R}^{D}\right) $ such that

\begin{itemize}
\item[\refstepcounter{equation}\text{(\theequation)}\label{afc3}] $%
\left\Vert F\right\Vert _{C^{m-1,1}\left( Q,\mathbb{R}^{D}\right) }\leq C$
and $F\left( x\right) \in K\left( x\right) $ for all $x\in E$.
\end{itemize}

We have achieved \eqref{afc3} under the assumption $E\subset Q$.

Finally, Theorem \ref{Th3} (B) for arbitrary $E \subset \mathbb{R}^n$ follows from the known case $E \subset Q$ by an obvious argument using a partition of unity. 
\end{proof}

\bigskip

\bibliographystyle{plain}

\begin{thebibliography}{10}

\bibitem{bm-2007}
Edward Bierstone and Pierre~D. Milman.
\newblock {$\scr C^m$}-norms on finite sets and {$\scr C^m$} extension
  criteria.
\newblock {\em Duke Math. J.}, 137(1):1--18, 2007.

\bibitem{bmp-2003}
Edward Bierstone, Pierre~D. Milman, and Wies{\l}aw Paw{\l}ucki.
\newblock Differentiable functions defined in closed sets. {A} problem of
  {W}hitney.
\newblock {\em Invent. Math.}, 151(2):329--352, 2003.

\bibitem{bmp-2006}
Edward Bierstone, Pierre~D. Milman, and Wies{\l}aw Paw{\l}ucki.
\newblock Higher-order tangents and {F}efferman's paper on {W}hitney's
  extension problem.
\newblock {\em Ann. of Math. (2)}, 164(1):361--370, 2006.

\bibitem{bs-1985}
Yuri Brudnyi and Pavel Shvartsman.
\newblock A linear extension operator for a space of smooth functions defined
  on a closed subset in {$\mathbb{R}^n$}.
\newblock {\em Dokl. Akad. Nauk SSSR}, 280(2):268--272, 1985.

  
  \bibitem{bs-1994-a}
Yuri Brudnyi and Pavel Shvartsman.
\newblock The traces of differentiable functions to subsets of {$\mathbb{R}^n$}.
\newblock {\em Linear and complex analysis. Problem book 3. Part II}, Lecture Notes in Mathematics, vol. 1574, Springer-Verlag, Berlin, 1994,  279--281.

\bibitem{bs-1994}
Yuri Brudnyi and Pavel Shvartsman.
\newblock Generalizations of {W}hitney's extension theorem.
\newblock {\em Internat. Math. Res. Notices}, (3):129 ff., approx.\ 11 pp.\
  (electronic), 1994.

\bibitem{bs-1997}
Yuri Brudnyi and Pavel Shvartsman.
\newblock The {W}hitney problem of existence of a linear extension operator.
\newblock {\em J. Geom. Anal.}, 7(4):515--574, 1997.

\bibitem{bs-1998}
Yuri Brudnyi and Pavel Shvartsman.
\newblock The trace of jet space {$J^k\Lambda^\omega$} to an arbitrary closed
  subset of {$\mathbb{ R}^n$}.
\newblock {\em Trans. Amer. Math. Soc.}, 350(4):1519--1553, 1998.

\bibitem{bs-2001}
Yuri Brudnyi and Pavel Shvartsman.
\newblock Whitney's extension problem for multivariate
  {$C^{1,\omega}$}-functions.
\newblock {\em Trans. Amer. Math. Soc.}, 353(6):2487--2512 (electronic), 2001.

\bibitem{f-2005}
Charles Fefferman.
\newblock A sharp form of {W}hitney's extension theorem.
\newblock {\em Ann. of Math. (2)}, 161(1):509--577, 2005.

\bibitem{f-2005-a}
Charles Fefferman.
\newblock A generalized sharp {W}hitney theorem for jets.
\newblock {\em Rev. Mat. Iberoamericana}, 21(2):577--688, 2005.


\bibitem{f-2006}
Charles Fefferman.
\newblock Whitney's extension problem for {$C^m$}.
\newblock {\em Ann. of Math. (2)}, 164(1):313--359, 2006.

\bibitem{f-2007}
Charles Fefferman.
\newblock {$C^m$} extension by linear operators.
\newblock {\em Ann. of Math. (2)}, 166(3):779--835, 2007.




\bibitem{f-2009-b}
Charles Fefferman.
\newblock Whitney's extension problems and interpolation of data.
\newblock {\em Bull. Amer. Math. Soc. (N.S.)}, 46(2):207--220, 2009.



\bibitem{fb1}
Charles Fefferman and Bo'az Klartag.
\newblock Fitting a {$C^m$}-smooth function to data. {I}.
\newblock {\em Ann. of Math. (2)}, 169(1):315--346, 2009.

\bibitem{fb2}
Charles Fefferman and Bo'az Klartag.
\newblock Fitting a {$C^m$}-smooth function to data. {II}.
\newblock {\em Rev. Mat. Iberoam.}, 25(1):49--273, 2009.

\bibitem{fl-2014}
Charles Fefferman and Garving~K. Luli.
\newblock The {B}renner-{H}ochster-{K}oll\'ar and {W}hitney problems for
  vector-valued functions and jets.
\newblock {\em Rev. Mat. Iberoam.}, 30(3):875--892, 2014.

\bibitem{fil-2016}
Charles Fefferman, Arie Israel, and Garving~K. Luli.
\newblock Interpolation of data by smooth non-negative functions.
\newblock {\em to appear}.

\bibitem{gl-1958}
Georges Glaeser.
\newblock \'{E}tude de quelques alg\`ebres tayloriennes.
\newblock {\em J. Analyse Math.}, 6:1--124; erratum, insert to 6 (1958), no. 2,
  1958.

\bibitem{lgw-2009}
Erwan Le~Gruyer.
\newblock Minimal {L}ipschitz extensions to differentiable functions defined on
  a {H}ilbert space.
\newblock {\em Geom. Funct. Anal.}, 19(4):1101--1118, 2009.

\bibitem{rock-convex}
R.~Tyrrell Rockafellar.
\newblock {\em Convex analysis}.
\newblock Princeton Landmarks in Mathematics. Princeton University Press,
  Princeton, NJ, 1997.
\newblock Reprint of the 1970 original, Princeton Paperbacks.



\bibitem{pavel-1982}
Pavel Shvartsman.
\newblock The traces of functions of two variables satisfying to the Zygmund
  condition.
\newblock In {\em Studies in the Theory of Functions of Several Real Variables,
  (Russian)}, pages 145--168. Yaroslav. Gos. Univ., Yaroslavl, 1982.

\bibitem{pavel-1984}
Pavel Shvartsman.
\newblock Lipschitz sections of set-valued mappings and traces of functions
  from the {Z}ygmund class on an arbitrary compactum.
\newblock {\em Dokl. Akad. Nauk SSSR}, 276(3):559--562, 1984.

\bibitem{pavel-1986}
Pavel Shvartsman.
\newblock Lipschitz sections of multivalued mappings.
\newblock In {\em Studies in the theory of functions of several real variables
  ({R}ussian)}, pages 121--132, 149. Yaroslav. Gos. Univ., Yaroslavl,
  1986.

\bibitem{pavel-1987}
Pavel Shvartsman.
\newblock Traces of functions of Zygmund class.
\newblock {\em (Russian) Sibirsk. Mat. Zh.}, (5):203--215, 1987. English transl. in {\em Siberian Math. J.} 28 (1987), 853--863.

\bibitem{pavel-1992}
Pavel Shvartsman.
\newblock K-functionals of weighted Lipschitz spaces and Lipschitz selections
  of multivalued mappings.
\newblock In {\em Interpolation spaces and related topics}, (Haifa, 1990), 245--268,
Israel Math. Conf. Proc., 5, BarIlan Univ., Ramat Gan, 1992.

\bibitem{s-2001}
Pavel Shvartsman.
\newblock On {L}ipschitz selections of affine-set valued mappings.
\newblock {\em Geom. Funct. Anal.}, 11(4):840--868, 2001.

\bibitem{s-2002}
Pavel Shvartsman.
\newblock Lipschitz selections of set-valued mappings and {H}elly's theorem.
\newblock {\em J. Geom. Anal.}, 12(2):289--324, 2002.


\bibitem{pavel-2004}
Pavel Shvartsman.
\newblock Barycentric selectors and a {S}teiner-type point of a convex body in
  a {B}anach space.
\newblock {\em J. Funct. Anal.}, 210(1):1--42, 2004.


\bibitem{s-2008}
Pavel Shvartsman.
\newblock The {W}hitney extension problem and {L}ipschitz selections of
  set-valued mappings in jet-spaces.
\newblock {\em Trans. Amer. Math. Soc.}, 360(10):5529--5550, 2008.

\bibitem{jw-1973}
John~C. Wells.
\newblock Differentiable functions on {B}anach spaces with {L}ipschitz
  derivatives.
\newblock {\em J. Differential Geometry}, 8:135--152, 1973.

\bibitem{whitney-1934}
Hassler Whitney.
\newblock Analytic extensions of differentiable functions defined in closed
  sets.
\newblock {\em Trans. Amer. Math. Soc.}, 36(1):63--89, 1934.

\bibitem{zobin-1998}
Nahum Zobin.
\newblock Whitney's problem on extendability of functions and an intrinsic metric.
\newblock {\em Advances in Math.}, 133(1):96--132, 1998.

\bibitem{zobin-1999}
Nahum Zobin.
\newblock Extension of smooth functions from finitely connected planar domains.
\newblock {\em Journal of Geom. Analysis}, 9(3):489--509, 1999.
\end{thebibliography}
\def\cprime{$'$} \def\cprime{$'$}

\vspace{1mm} 

\begin{flushleft} 
Charles Fefferman \\
Affiliation: Department of Mathematics, Princeton University, Fine Hall Washington Road, Princeton, New Jersey, 08544, USA \\
Email: cf$\MVAt$math.princeton.edu  \\
\vspace{2mm} 
Arie Israel \\ 
Affiliation: The University of Texas at Austin, 
Department of Mathematics,
2515 Speedway Stop C1200, 
Austin, Texas, 78712-1202, USA \\ 
Email: arie$\MVAt$math.utexas.edu \\
\vspace{2mm} 
Garving K. Luli \\ 
Affiliation: Department of Mathematics, 
University of California, Davis, 
One Shields Ave, 
Davis, California, 95616, USA \\ 
Email: kluli@$\MVAt$math.ucdavis.edu 
\end{flushleft}

\end{document}